\documentclass[a4paper,12pt,twoside]{amsart}
\usepackage{amsmath}
\usepackage{amsthm}
\usepackage{amssymb}
\theoremstyle{plain}
\newtheorem{thm}{\indent\bf Theorem}[section]
\newtheorem{lem}[thm]{\indent\bf Lemma}
\newtheorem{prop}[thm]{\indent\bf Proposition}
\newtheorem{cor}[thm]{\indent\bf Corollary}
\theoremstyle{definition}
\newtheorem{rem}{\indent\it Remark}[section]
\newtheorem{exa}{\indent\it Example}[section]
\numberwithin{equation}{section}
\numberwithin{figure}{section}
\def \re {\mathrm{Re\,}}
\def \im {\mathrm{Im\,}}

\begin{document}
\title[Fourth Painlev\'e transcendents]{ 
Kapaev's global asymptotics of the fourth
Painlev\'e transcendents. Elliptic asymptotics}
\author[Shun Shimomura]{Shun Shimomura} 
%%%%%%%%%%%%%%%%%%%%%%%%%%%%%%%%%%%%%%%%%%%%%
%%%%%%%%%%%%%%%%%%%%%%%%%%%%%%%%%%%%%%%%%%%%%%
\address{Department of Mathematics, 
Keio University, 
3-14-1, Hiyoshi, Kohoku-ku,
Yokohama 223-8522 
Japan \quad
{\tt shimomur@math.keio.ac.jp}}
\date{}
%%%%%%%%%%%%%%%%%%%%%%%%%%%%%%%%%%%%%%%
%%%%%%%%%%%%%%%%%%%%%%%%%%%%%%%%%%%%%%%
\maketitle
\begin{abstract}
For the fourth Painlev\'e transcendents we derive elliptic asymptotic
representations, which were announced by late Professor Kapaev 
without proofs. Then we newly obtain related results including the 
correction function.
\end{abstract}
\hfill\thanks{To the memory of Professor Kapaev}
%\ams{34M55}
 \allowdisplaybreaks
%%%%%%%%%%%%%%%%%%%%%%%%%%%%%%%%%%%%%
%%%%%% section 1 %%%%%%%
\section{Introduction}\label{sc1}
%%%%%%%%%%%%%%%%%%%%%%%%%%%%%%%
%%%%%%%%%%%%%%%%%%%%%%%%%%%%%%%%%%
The fourth Painlev\'e equation 
%%%%%%%%%%%%%%%%%%%%%%%%%%%%%%%%
\begin{equation*}
\tag*{P$_\mathrm{IV}$}
y''=  \frac {(y')^2}{2y} + \frac 32 y^3 +4xy^2 +(-4\alpha+\beta+2x^2)y
-\frac{\beta^2}{2y}
\end{equation*}
%%%%%%%%%%%%%%%%%%%%%%%%
$(y'=dy/dx)$ with $\alpha, \beta\in \mathbb{C}$ governs the isomonodromy
deformation of the linear system
%%%%%%% (1.1) %%%%%
\begin{align}\label{1.1}
& \frac{d\Psi}{d\xi} =\Biggl( \Bigl(\frac {\xi^3}{2} +\xi(x+ \mathrm{u}
\mathrm{v}) +\frac{\alpha}{\xi} \Bigr)\sigma_3
+i \begin{pmatrix} 0 & \xi^2\mathrm{u}+ 2x\mathrm{u} +\mathrm{u}' \\
  \xi^2\mathrm{v}+ 2x\mathrm{v} -\mathrm{v}'  &  0   \end{pmatrix}\Biggr) \Psi,
\\
\notag
& \beta=\mathrm{u}'\mathrm{v}-\mathrm{u}\mathrm{v}' +2x\mathrm{u}\mathrm{v}
-(\mathrm{u}\mathrm{v})^2, \quad y=\mathrm{u}\mathrm{v}
\end{align}
by Kitaev \cite{Kitaev} (for another isomonodromy system see \cite{JM}).
Kapaev \cite{Kapaev-3} announced global asymptotic results on solutions of
P$_{\mathrm{IV}}$ including elliptic asymptotic representations along 
generic directions in the complex plane, giving notice of publishing the proofs
in \cite{FIKN} written by him in collaboration with Fokas, Its and Novokshenov. 
The monograph \cite{FIKN}, however, contains none of these
proofs, and Kapaev passed away without publishing the asymptotics on 
P$_{\mathrm{IV}}$ except \cite{I-Kapaev} for solutions on the real line
and \cite{Kapaev-4}. 
In Kapaev's announcement, \cite[Theorems 2 and 3]{Kapaev-3} are on
elliptic asymptotics described in terms of an elliptic integral as the 
inverse function;  
\cite[Theorems 4 and 5]{Kapaev-3} on trigonometric
asymptotics, and \cite[Theorems 6, 7, 8 and 9]{Kapaev-3} on truncated solutions.
Asymptotics of P$_{\mathrm{IV}}$ along Stokes rays are also 
studied by Kitaev \cite{Kitaev}, \cite{Kitaev-D}, \cite{Kitaev-0}. 
By isomonodromy technique, 
elliptic asymptotics of general solutions are known for P$_{\mathbf{I}}$ 
(\cite{Ka-Ki}, \cite{Kitaev-2}, \cite{Kitaev-3}), P$_{\mathbf{II}}$ 
(\cite{Novokshenov-1}, \cite{Novokshenov-2},
\cite{Kapaev-2}, \cite{Kapaev-1}, \cite{Kitaev-3}),
P$_{\mathbf{III}}(D_6)$ (\cite{Novokshenov-3}, \cite{Novokshenov-4}, 
\cite{SSS}), P$_{\mathbf{III}}(D_7)$ (\cite{SS}) 
and P$_{\mathbf{V}}$ (\cite{S}), and in each asymptotic formula 
the phase shift is expressed in terms of the monodromy data.  
For the complete P$_{\mathrm{IV}}$ 
the proofs of elliptic asymptotics have not been published, 
though P$_{\mathrm{IV}}$ with $\alpha=0$ was treated by 
Vereshchagin \cite{Vere} for $0<\arg x<\pi/4$. 
The present author believes that it is significant to present the 
proofs of elliptic representations for P$_{\mathrm{IV}}$ in
Kapaev's announcement \cite{Kapaev-3} for the following reasons: 
(i) in the study of a general solution of P$_{\mathrm{IV}}$,  
elliptic asymptotics are crucial information; 
(ii) as a practical matter, the process of deriving the elliptic 
representation needs some devices peculiar to P$_{\mathrm{IV}}$; 
and (iii) related important materials including the correction function,
which appear in the proofs, are not referred to in the announcement.
\par
In this paper we derive the elliptic asymptotics for 
the complete P$_{\mathrm{IV}}$ by the isomonodromy deformation method along 
the lines of the arguments in \cite{Kapaev-1}, \cite{Kitaev-2}, \cite{Kitaev-3}
with discussions on the Boutroux equations and on the justification of
the asymptotics as a solution of P$_{\mathrm{IV}}$. Then we newly obtain
related results including the correction function 
$B_{\phi}(t)$ given by \eqref{5.1}, which contains information 
on the asymptotics and is essential in the justification procedure as 
in \cite[Section 3]{Kitaev-3}.  
\par
The main results are stated in Section \ref{sc2}. Theorems \ref{thm2.1} and 
\ref{thm2.2} present elliptic representations of a general solution
of P$_{\mathrm{IV}}$, which correspond to the announced 
\cite[Theorem 2]{Kapaev-3}.
These results are also described by an alternative
elliptic expression as of Corollary \ref{cor2.3}, which has an advantage in 
treating in general sectors (cf.~Theorem \ref{thm2.4}). 
For elliptic expressions of Theorems \ref{thm2.1} and \ref{thm2.2} in
directions neighbouring the positive real axis,
degeneration to trigonometric asymptotics may be considered under certain 
suppositions, and is shown to be consistent with the result of 
\cite[Theorem 4]{Kapaev-3}. This fact supports the validity of signs 
in Theorems \ref{thm2.1} and \ref{thm2.2} contradicting those of
\cite[Theorem 2]{Kapaev-3} (cf. Remark \ref{rem2.4}).
Section \ref{sc3} summarises necessary facts on the isomonodromy linear system 
\eqref{3.1} and on its monodromy data consisting of Stokes coefficients.  
Section \ref{sc4} explains turning points, Stokes graphs and WKB solutions, 
which are necessary in the WKB analysis.  
In Section \ref{sc5} we solve a direct monodromy problem for system \eqref{3.1} 
by the WKB analysis to obtain the key relations consisting of monodromy data
and certain integrals (cf.~Propositions \ref{prop5.1} and \ref{prop5.2}).
Asymptotics of these key relations are examined in Section \ref{sc6} by the
use of the $\vartheta$-function. In Section \ref{sc7} from the formulas thus
obtained asymptotic forms of the main theorems are derived by solving an
inverse monodromy problem for the prescribed monodromy data. In this process
we make technical devices to  
find necessary special properties of the elliptic function related to
our case (Propositions \ref{prop7.4} and \ref{prop7.5}). 
The justification as a solution of
P$_{\mathrm{IV}}$ is performed along the lines of Kitaev \cite{Kitaev-3} with
\cite{Kitaev-1}. The final section is devoted to the proofs of necessary 
facts on the Boutroux equations summarised in Proposition \ref{prop8.15}, 
which determine $A_{\phi}$ parametrising
the related elliptic function. Furthermore we clarify local structure 
of Stokes curves near coalescing turning points, which  
is used in drawing Stokes graphs in Section \ref{sc4}. 
\par
Throughout this paper we use the following symbols:
\par 
(1) The coefficient $A(\varphi_0)$ defined by \cite[(20)]{Kapaev-3} is 
denoted by $e^{3i\phi}A_{\phi}$ $(\phi=\varphi_0)$;
\par
(2) $\sigma_1,$ $\sigma_2$, $\sigma_3$ denote the Pauli matrices
$$
\sigma_1=\begin{pmatrix} 0 & 1 \\ 1 & 0 \end{pmatrix}, \quad
\sigma_2=\begin{pmatrix} 0 & -i \\ i & 0 \end{pmatrix}, \quad
\sigma_3=\begin{pmatrix} 1 & 0  \\ 0 & -1 \end{pmatrix}; 
$$
\par
(3) for complex-valued functions $f$ and $g$, we write $f \ll g$ or $g\gg f$
if $f=O(|g|)$, and write $f \asymp g$ if $g \ll f \ll g$.
%%%%%%%%%%%%%%%%%%%%%%%%%
%%%%%%%%%%% Section 2 %%%%%%
\section{Results}\label{sc2}
%%%%%%%%%%%%%%%%%%%%%%%%%%%%
To state the results we explain the monodromy data \cite[Section 2]{Kapaev-3}, 
\cite[Section 2]{I-Kapaev}, and the Boutroux 
equations \cite[Section 3]{Kapaev-3}.
For $k\in \mathbb{Z}$ system \eqref{1.1} admits the matrix solutions
%%%%%%%%%%%%%%%%%%%%%%%%%%%%%%%%%%
%%%%%% (2.1) %%%%%%%%%%%%%%%%%%
\begin{equation}\label{2.1}
\Psi_k^{\infty}(\xi)=(I+O(\xi^{-1}))\exp((\tfrac 18\xi^4 +\tfrac 12 x\xi^2 +(\alpha
-\beta)\ln \xi)\sigma_3 )
\end{equation}
as $\xi \to \infty$ through the sector $|\arg\xi +\frac {\pi}8 -\frac {\pi}4 k|
<\frac{\pi}4,$ and
$$
\Psi^0(\xi)=T_0(\xi) \xi^{\alpha\sigma_3}(I+J_0\ln\xi)e^{\mathrm{int}(x)
\sigma_3},
\quad J_0=
 \begin{pmatrix} 0 & j_+ \\ j_-  & 0 \end{pmatrix}, 
\quad \mathrm{int}(x)=\int^x\mathrm{u}
\mathrm{v}\, dx 
$$
as $\xi \to 0$, where $T_0(\xi)$ is invertible around $\xi=0$, and
$j_+=0$ if $\alpha-\frac 12 \not=0,1,2,\ldots,$ $j_-=0$ if $\alpha-\frac 12
\not= -1,-2,-3,\ldots.$ The Stokes matrices  
$$
S_{2l-1}=\begin{pmatrix} 1 & s_{2l-1} \\ 0 & 1 \end{pmatrix}, \quad
S_{2l}=\begin{pmatrix} 1 & 0 \\ s_{2l}  & 1 \end{pmatrix} \quad (l\in \mathbb{Z}
) 
$$
are defined by $\Psi_{k+1}^{\infty}(\xi)=\Psi_{k}^{\infty}(\xi)S_k$, and 
satisfy $S_{k+4}=
e^{-i\pi(\alpha-\beta)\sigma_3}\sigma_3 S_k\sigma_3 e^{i\pi(\alpha-\beta)
\sigma_3},$ $s_{k+4}=-s_k e^{(-1)^k 2\pi i (\alpha-\beta)}.$ 
For the matrices
$M=e^{-\mathrm{int}(x)\sigma_3}(I+i\pi J_0)e^{\mathrm{int}(x)\sigma_3}
e^{i\pi\alpha \sigma_3}$ and $E$ such that $\sigma_3\Psi^0(e^{i\pi}\xi)
\sigma_3 =\Psi^0(\xi)M$ and $\Psi_1^{\infty}(\xi)=\Psi^0(\xi)E$,  
the semi-cyclic relation 
$$
S_1S_2S_3S_4=E^{-1}\sigma_3M^{-1}Ee^{i\pi(\alpha-\beta)\sigma_3}\sigma_3
$$
holds, and the traces of both sides lead to the surface of the monodromy data 
\begin{equation*}
\tag*{$\mathcal{M}_0(\alpha,\beta):$}
((1+s_{1}s_{2}) (1+s_{3}s_{4})+ s_{1}s_{4}) e^{-i\pi (\alpha-\beta)}
 -(1+s_{2}s_{3}) e^{i\pi(\alpha-\beta)} =-2i \sin \pi\alpha.
\end{equation*}
For each $m\in \mathbb{Z}$  
the semi-cyclic relation for $S_{j+m}$ $(1\le j\le 4)$ in 
Proposition \ref{prop3.1} yields 
\begin{align*}
\tag*{$\mathcal{M}_m(\alpha,\beta):$}
((1+s_{1+m}s_{2+m}) &(1+s_{3+m}s_{4+m})+ s_{1+m}s_{4+m}) e^{-i\pi (-1)^m(\alpha
-\beta)}
\\
& -(1+s_{2+m}s_{3+m}) e^{i\pi(-1)^m(\alpha-\beta)} =-2i (-1)^m\sin \pi\alpha.
\end{align*}
%% For each $m\in \mathbb{Z}$, $\mathcal{M}_m(\alpha,\beta)$ is thus related to 
%% $\mathcal{M}_0(\alpha,\beta)$ by these semi-cyclic relations.
Around a nonsingular point on $\mathcal{M}_m(\alpha,\beta)$ suitable three of 
$s_{j+m}$ $(1\le j\le 4)$ are independent. 
For any $c\in\mathbb{C}\setminus\{0\}$ the
gauge transformation $\Psi=c^{\sigma_3}\tilde{\Psi}$ induces the action
$$
[c]: \,\,\, (S_1,S_2,S_3, S_4) \mapsto (c^{-\sigma_3}S_1 c^{\sigma_3},
c^{-\sigma_3}S_2 c^{\sigma_3},c^{-\sigma_3}S_3 c^{\sigma_3},
c^{-\sigma_3}S_4 c^{\sigma_3})
$$
on $\mathcal{M}_0(\alpha,\beta)$ consistent with the isomonodromy structure
of \eqref{1.1}, and  
each solution of P$_{\mathrm{IV}}$ corresponds to an orbit, or equivalence class 
yielded by dividing $\mathcal{M}_0(\alpha,\beta)$ by $[c]$.
Thus an orbit passing through a point 
$(s_1,s_2,s_3,s_4)\in \mathcal{M}_0 (\alpha,\beta)$ parametrises 
a solution of P$_{\mathrm{IV}}$. Let us call it a {\it solution 
labelled by} $(s_1,s_2,s_3,s_4)$. 
(In \cite[Section 2]{Kapaev-3} the gauge
symmetry on $(\mathrm{u},\mathrm{v})$ is considered.)
\par
For $0<A<\frac{8}{27}$ let $z=0,$ $x_1,$ $x_3,$ $x_5$ with $x_5 <x_3<x_1 <0$
be the zeros of the polynomial $P_A(z)=z^4+4z^3+4z^2+4Az$ such that 
$x_1 \to 0,$ $x_3, x_5 \to -2$ as $A\to 0$ and that
$x_1,x_3 \to -\frac 23,$ $x_5 \to -\frac 83$ as $A\to \frac 8{27}$. 
Let $\Pi_{\pm}$ be two copies of 
$P^1(\mathbb{C})\setminus([x_5,x_3]\cup [x_1,0])$. 
The elliptic curve $w^2=P_A(z)$ is the two sheeted Riemann surface
$\Pi_{A,0}=\Pi_+ \cup \Pi_-$ 
glued along the cuts $[x_5,x_3]$, $[x_1,0]$.
As long as $(\phi,A) \in \mathbb{R} \times \mathcal{D}_0$ with
$\mathcal{D}_0=\mathbb{C}\setminus \{c\le 0\} \cup\{c\ge \tfrac{8}{27}\}$, 
the polynomial 
$z^4+4e^{i\phi}z^3 +4e^{2i\phi}z^2 +4e^{3i\phi} Az$ admits the roots 
$0,$ $z_j$ $(j=1,3,5)$ such that 
$\re e^{-i\phi}z_5 < \re e^{-i\phi}z_3 <\re e^{-i\phi} z_1 $ 
and that $z_j=x_j$ if $\phi=0,$ $0<A<\frac{8}{27}$
(cf. Corollary \ref{cor8.2}), and then the elliptic curve
$$
w^2=w(A,z)^2=z^4+4e^{i\phi}z^3 +4e^{2i\phi}z^2 +4e^{3i\phi}Az 
=z(z-z_1)(z-z_3)(z-z_5)
$$
may be considered to be the two sheeted Riemann surface $\Pi_{A,\phi}=\Pi_+ \cup
\Pi_-$ 
%%% with $\Pi_{\pm}=P^1(\mathbb{C})\setminus([0,z_1]\cup [z_5,z_3])$ 
that is a continuous modification of $\Pi_{A,0}$ with $\Pi_{\pm}$ glued 
along cuts $[z_1,0],$ $[z_5,z_3]$. 
Here the branch of $w=\sqrt{z(z-z_1)(z-z_3)(z-z_5)}=z^2\sqrt{1-z_1z^{-1}}
\sqrt{1-z_3z^{-1}}\sqrt{1-z_5z^{-1}}$ is such that $\sqrt{1-z_jz^{-1}}\to 1$
as $z\to \infty$ $(j=1,3,5)$ on the upper sheet $\Pi_+$.

As will be shown in Corollary \ref{cor8.16} with Remark \ref{rem8.3}, 
for each $\phi\in \mathbb{R}$, there exists
$A_{\phi} \in\mathbb{C}$ such that, for any cycle $\mathbf{c}$ on 
$\Pi_{A_{\phi},\phi}$
$$
\re \int_{\mathbf{c}} \frac{w(A_{\phi},z)}{z}dz =0
$$
and that $A_{\phi}$ has the following properties: 
\par
(1) for each $\phi\in \mathbb{R}$, $A_{\phi}$ is uniquely determined;
\par
(2) ${A}_{\phi+\pi/2}={A}_{\phi},$ ${A}_{-\phi}=\overline{{A}_{\phi}}$;
\par
(3) $A_0=\frac8{27},$ $A_{\pm \pi/4}= 0$ and 
$0\le \re A_{\phi} \le \tfrac 8{27};$
\par
(4) $A_{\phi}$ is continuous in $\phi \in \mathbb{R}$ and is smooth in
$\phi \in \mathbb{R}\setminus \{\pi k/4 \,|\, k\in \mathbb{Z} \}.$ 
\par\noindent
The elliptic curve $\Pi_{A_{\phi},\phi}$ degenerates 
if and only if $\phi=\pi k/4$ with $k\in \mathbb{Z}.$ 
%%%%%%%%%%%%%%%%%%%%%%%%%%%%%%%%%%%%%%%%%%%%%%%
\subsection{Solutions for $0<|\phi|<\pi/4$ in \cite[Theorem 2]{Kapaev-3}}
\label{ssc2.1}
%%%%%%%%%%%%%%%%%%%%%%%%%%%%%%%%%%%%%%%%
For $0<|\phi|<\pi/4$ and $A\in \mathcal{D}_0$, let
the primitive cycles $\mathbf{a}$ and $\mathbf{b}$ on $\Pi_{A,\phi}$
be as described on the upper sheet $\Pi_+$ in Figure \ref{cycles1}. 
(The cycles $\mathbf{a}$ and $\mathbf{b}$ are consistent with those defined in
\cite[Section 3]{Kapaev-3}.)
Then the Boutroux equations 
%%%% (2.2) %%%%%%%%%%%%%%%%%
\begin{equation}\label{2.2}
\re \int_{\mathbf{a}} \frac{w(A, z)}{z} dz=
\re \int_{\mathbf{b}} \frac{w(A, z)}{z} dz=0
\end{equation}
admit a unique solution $A=A_{\phi}$, which means \cite[Theorem 1]{Kapaev-3}.
For $0<|\phi|<\pi/4$ the periods of $\Pi_{A_{\phi},\phi}$ along $\mathbf{a}$ and
$\mathbf{b}$ are given by
$$
\Omega_{\mathbf{a}}=\Omega_{\mathbf{a}}^{\phi} =\int_{\mathbf{a}}\frac{dz}
{w(A_{\phi},z)}, 
\quad
\Omega_{\mathbf{b}}=\Omega_{\mathbf{b}}^{\phi} =\int_{\mathbf{b}}\frac{dz}
{w(A_{\phi},z)}, 
$$
which satisfy $\im \Omega_{\mathbf{b}}/\Omega_{\mathbf{a}}>0.$
%%%%%%%%%%%%%%%%%%%%%%%%%
%%%%%%%% Figure 2.1 %%%%%%%%%%%%%%%%%%%%%%%
{\small
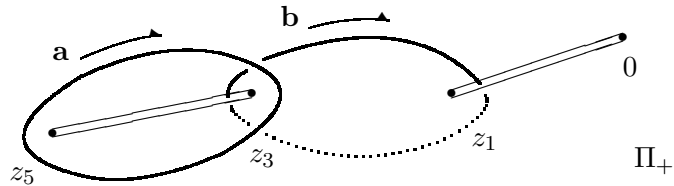
\begin{figure}[htb]
\begin{center}
\unitlength=0.75mm
%%%%%%%%%%%%%%%%%%%%%%%%%%%%%%%%%%
%%%%%%%%%%%%%%%%%%%%%%%%%%%%%%%%%%
%%%%%%%%%%%%%%%%%%
\begin{picture}(70,40)(-40,-19)
 \put(-57.5,-15.2){\makebox{$z_5$}}
 \put(-15.5,-12.2){\makebox{$z_3$}}
 \put(23.5,-9.2){\makebox{$z_1$}}
 \put(50,3.2){\makebox{$0$}}

%% \put(-29,11){\makebox{$\mathbf{a}$}}
%% \put(11,13){\makebox{$\mathbf{b}$}}
 \put(-50,6){\makebox{$\mathbf{a}$}}
 \put(-10,12){\makebox{$\mathbf{b}$}}

 \put(52,-13){\makebox{$\Pi_+$}}

\thinlines

 \qbezier(-45,6.2) (-37,9.8) (-31,10.3)
 \qbezier(-5,11.6) (3,13.8) (9,12.9)

%% \put(-45,6.2){\vector(-1,-1){0}}
%% \put(-5,11.6){\vector(-3,-2){0}}

 \put(-31,10.3){\vector(4,-1){0}}
 \put(9,12.9){\vector(3,-2){0}}
 \put(-50,-6.3){\line(5,1){35}}
 \put(-50,-7.7){\line(5,1){35}}
 \put(20,0.7){\line(3,1){30}}
 \put(20,-0.7){\line(3,1){30}}
\thicklines
 \put(-50,-7){\circle*{1.5}}
 \put(-15,0){\circle*{1.5}}
 \put(20,0){\circle*{1.5}}
 \put(50,10){\circle*{1.5}}
  \qbezier(-45,3) (-65,-9.5) (-45,-14.5)
  \qbezier(-20,-10.5) (0,1.5) (-20,7)
  \qbezier(-45,3) (-31, 9.5) (-20,7)
  \qbezier(-45,-14.5) (-34.5, -17) (-20,-10.5)

 \qbezier(-13, 6) (12.5,15.5) (25,1.8)
 \qbezier(-16, 4) (-19.5,2) (-19.2,-0.5)

 \qbezier[10](26.0,-0.3) (27.5,-3) (20,-7.5)
 \qbezier[25](-12.5,-7.3) (4,-14.8) (20,-7.5)
 \qbezier[7](-15.6,-5.4) (-19,-3) (-18.6,-2.4)

\end{picture}
\end{center}
\caption{Cycles $\mathbf{a},$ $\mathbf{b}$ on $\Pi_{A,\phi}=\Pi_+\cup\Pi_-$}
\label{cycles1}
\end{figure}
}
%%%%%%%%%%%%%%%%%%%%%%%%%%%%%%%%%%%%%%%%%%%%
\par
Let $\mathrm{P}(u;A)$ denote the elliptic function defined by
$$
\mathrm{P}_{u}^2 =\mathrm{P}^4 +4e^{i\phi}\mathrm{P}^3 +4e^{2i\phi}
\mathrm{P}^2 +4 e^{3i\phi}A\mathrm{P}, \quad \mathrm{P}(0;A)=0, \quad
\text{i.e.} \,\,\,  \int^{\mathrm{P}(u;A)}_0 \frac{dz}{w(A,z)}=u.
$$
Note that $\mathrm{P}(u;A_{\phi})$ does not degenerate as long as
$\phi\not=k\pi/4,$ $k\in\mathbb{Z}.$ Then \cite[Theorem 2]{Kapaev-3} with
$n=0,$ $m=0,1$ may be described as follows.
\par
Let $y=y(\mathbf{s},x)$ denote the solution of P$_{\mathrm{IV}}$ labelled by
the monodromy data $\mathbf{s}=(s_1,s_2,s_3,s_4)\in \mathcal{M}_0(\alpha,\beta).$
%%%%%%%%%%%%%%%%%%%%%%%%%%%%%%%%%%%%
%%%% Theorem 2.1 %%%%%%%%
\begin{thm}\label{thm2.1}
Suppose that $-\pi/4 <\phi<0$ and that $(1+s_1s_2)(1+s_2s_3)-1\not=0,$
$1+s_1s_2\not=0.$ Then 
\begin{align*}
& y(\mathbf{s},x)=e^{-i\phi}x \mathrm{P}(e^{i\phi}t+ \chi_{0-} +O(t^{-\delta});
A_{\phi}), \quad 2t=(e^{-i\phi}x)^2,
\\
& \chi_{0-}\equiv\frac{\Omega_{\mathbf{a}}}{2\pi i} \ln((1+s_1s_2)(1+s_2s_3)-1)
\\
&\phantom{----------}
-\frac{\Omega_{\mathbf{b}}}{2\pi i} \ln(1+s_1s_2)+ \frac{\Omega_{\mathbf{b}}}3
(\alpha-\beta) \quad \mod \Omega_{\mathbf{a}}\mathbb{Z} +\Omega_{\mathbf{b}}
\mathbb{Z}
\end{align*}
as $2t =(e^{-i\phi}x)^2 \to \infty$ through the cheese-like strip
$$
S_0(\phi,t_{\infty},\kappa_0,\delta_0)=\{x=e^{i\phi}\sqrt{2t}\,|\, 
\re t>t_{\infty}, \,\,\, |\im t|<\kappa_0 \} \setminus \bigcup_{\sigma\in
\mathcal{P}_{0-}} \{|x-\sigma|<\delta_0 \}
$$
with
$\mathcal{P}_{0-}=\{e^{i\phi}\sqrt{2t_{0-}} \,|\, e^{i\phi}t_{0-} +\chi_{0-} 
=\pm \tfrac 13
\Omega_{\mathbf{b}}+\Omega_{\mathbf{a}}\mathbb{Z} +\Omega_{\mathbf{b}}\mathbb
{Z} \}.$
Here $\delta$ is some positive number, $\kappa_0$ a given positive number,
$\delta_0$ a given small positive number, and $t_{\infty}=t_{\infty}(\kappa_0,
\delta_0)$ a sufficiently large positive number depending on $(\kappa_0,
\delta_0).$
\end{thm}
%%%%%%%%%%%%%%%%%%%%%%%%%%%%%%%%%%%%
%%%% Theorem 2.2 %%%%%%%%
\begin{thm}\label{thm2.2}
Suppose that $0<\phi<\pi/4 $ and that $(1+s_1s_2)(1+s_2s_3)-1\not=0,$
$1+s_2s_3\not=0.$ Then 
\begin{align*}
& y(\mathbf{s},x)=e^{-i\phi}x \mathrm{P}(e^{i\phi}t+ \chi_{0+} +O(t^{-\delta});
A_{\phi}), \quad 2t=(e^{-i\phi}x)^2,
\\
& \chi_{0+}\equiv\frac{\Omega_{\mathbf{a}}}{2\pi i} \ln ((1+s_1s_2)(1+s_2s_3)-1)
\\
&\phantom{----------}
+\frac{\Omega_{\mathbf{b}}}{2\pi i} \ln(1+s_2s_3)+ \frac{\Omega_{\mathbf{b}}}3
(\alpha-\beta) \quad \mod \Omega_{\mathbf{a}}\mathbb{Z} +\Omega_{\mathbf{b}}
\mathbb{Z}
\end{align*}
as $2t =(e^{-i\phi}x)^2 \to \infty$ through the cheese-like strip
$$
S_1(\phi,t_{\infty},\kappa_0,\delta_0)=\{x=e^{i\phi}\sqrt{2t}\,|\, 
\re t>t_{\infty}, \,\,\, |\im t|<\kappa_0 \} \setminus \bigcup_{\sigma\in
\mathcal{P}_{0+}} \{|x-\sigma|<\delta_0 \}
$$
with
$\mathcal{P}_{0+}=\{e^{i\phi}\sqrt{2t_{0+}} \,|\, e^{i\phi}t_{0+} +\chi_{0+} 
=\pm \tfrac 13
\Omega_{\mathbf{b}}+\Omega_{\mathbf{a}}\mathbb{Z} +\Omega_{\mathbf{b}}\mathbb
{Z} \}.$
\end{thm}
%%%%%%%%%%%%%%%%%%%%%%%%%%%%%%%
%%%% Remark 2.1 %%%%%%
\begin{rem}\label{rem2.1}
If $A_{\phi}\not=0, \tfrac 8{27}$, then $\mathrm{P}(\tau;A_{\phi})=
 e^{3i\phi}A_{\phi}(\wp(\tau;g_2,g_3)-\frac 13 e^{2i\phi})^{-1}$ with 
$g_2=-4e^{4i\phi}(A_{\phi}-\frac 13)$,
$g_3= e^{6i\phi}(-A_{\phi}^2+ \frac 43 A_{\phi} -\frac 8{27}),$
where $\wp(\tau;g_2,g_3)$ is the Weierstrass pe-function satisfying
$\wp_{\tau}^2=4\wp^3-g_2\wp -g_3.$ Furthermore, $\wp_{\tau}=-e^{3i\phi}A_{\phi}
\mathrm{P}^{-2} w(A_{\phi},\mathrm{P}).$
\end{rem}
%%%%%%%%%%%%%%%%%%%%%%%%%%
%%%%%%%%%%%%%%%%%%%%%%%%%%%%%%%%%%%%%%%%%%%%%%%
%%%%%%%%%% Remark 2.3 %%%%%%%%%%%%%%%%%%%%%%%
\begin{rem}\label{rem2.3}
The solutions $y(\mathbf{s},x)$ in Theorems \ref{thm2.1} and \ref{thm2.2},
and the correction function $B_{\phi}(t)$ in Proposition \ref{prop7.3} are
parametrised by $(\mathfrak{s}_{12},\mathfrak{s}_{23})=(s_1s_2, s_2s_3)$.
The variables $\mathfrak{s}_{12},$ $\mathfrak{s}_{23}$, and
$\mathfrak{s}_{34}=s_3s_4,$ $\mathfrak{s}_{41}=s_4s_1$ are invariant under the
action $[c]$ on $\mathcal{M}_0(\alpha,\beta)$, and hence the family of 
the orbits generated by $[c]$ on 
$\mathcal{M}_0(\alpha,\beta)$ may be identified with the two-dimensional 
surface:
\begin{align*}
\tag*{$\mathcal{M}^*_0(\alpha,\beta):$}    
&  \bigl( (1+\mathfrak{s}_{12})(1+\mathfrak{s}_{34})+      
\mathfrak{s}_{41} \bigr)e^{i\pi(\beta-\alpha)} 
-(1+ \mathfrak{s}_{23})e^{i\pi(\alpha-\beta)} +2i \sin \pi\alpha  =0,
\\
& \mathfrak{s}_{12}\mathfrak{s}_{34}-\mathfrak{s}_{23}\mathfrak{s}_{41}=0,
\end{align*}
whose points parametrise solutions of P$_{\mathrm{IV}}$ given in Theorems
\ref{thm2.1}, \ref{thm2.2}, and \cite[Theorems 2, 4, 5 and 6]{Kapaev-3} 
with $n=0$. Singular points on this surface is given by Proposition 
\ref{prop3.2}. 
It is easy to see that a point $(\mathfrak{s}_{12},\mathfrak{s}_{23},
\mathfrak{s}_{34}) \in \mathcal{M}^*_0(\alpha,\beta)$ satisfying the condition
of Theorem \ref{thm2.1} or \ref{thm2.2} is nonsingular.
\end{rem}
%%%%%%%%%%%%%%%%%%%%%%%%%%%%%%%%%%%%%%%%%%%%
%%%%%%%%%%%%%%%%%%%%%%%%%%%%%%%%%
%%%%%%%%%%%% Remark 2.4 %%%%%%%%%%
\begin{rem}\label{rem2.4}
An inconsistency appears between the signs of $\ln((1+s_1s_2)(1+s_2s_3)-1)$
in \cite[Theorem 2]{Kapaev-3} and 
in our Theorems \ref{thm2.1} and \ref{thm2.2}.
The agreements of the trigonometric asymptotics with
\cite[Theorem 4]{Kapaev-3} and \cite[Theorem 4.1]{I-Kapaev} discussed 
in Section \ref{ssc2.3} support the correctness of the sign $+$ 
in our theorems.
\end{rem}
%%%%%%%%%%%%%%%%%%%%%%%%%%%%%%%%%%%
\subsection{Alternative expression of solutions}\label{ssc2.2}
%%%%%%%%%%%%%%%%%%%%%%%%%%%%%%%%%%%%%%%%%%%%%%%
Let the elliptic curve $\Pi_{A_{\phi}}^*=\Pi^*_+ \cup\Pi_-^*$: 
$v^2=v(A_{\phi},\zeta)^2=e^{-4i\phi} w(A_{\phi},e^{i\phi}\zeta)^2$ be as defined
in Section \ref{sc8}, and the cycles $\mathbf{a}^*$ and $\mathbf{b}^*$ 
on $\Pi^*_{A_{\phi}}$ as drawn in Figure \ref{cycles2}.
These elliptic curve and cycles, which depend on $A_{\phi}$ only, are also
the images of $\Pi_{A_{\phi},\phi} =\Pi_+\cup \Pi_-$ and
$\mathbf{a}$ and $\mathbf{b}$ under the mapping $z=e^{i\phi}\zeta$. 
Then the Boutroux equation \eqref{2.2} is written in the form
%%%%%%% (2.3) %%%%%%%%%%%%%%%%%%%
\begin{equation}\label{2.3}
\re\, e^{2i\phi}\int_{\mathbf{a}_*} \frac{v(A_{\phi},\zeta)}{\zeta}d\zeta=
\re\, e^{2i\phi}\int_{\mathbf{b}_*} \frac{v(A_{\phi},\zeta)}{\zeta}d\zeta=0,
\end{equation}
in which 
$$
v(A_{\phi},\zeta):=e^{-2i\phi}w(A_{\phi},e^{i\phi}\zeta)
=\sqrt{\zeta^4+4\zeta^3 +4\zeta^2 +4A_{\phi}\zeta},
$$
and for each $\phi \in \mathbb{R},$ $A_{\phi}$ is a unique solution of
\eqref{2.3} as in Proposition \ref{prop8.15}.
Let $\mathfrak{P}(u)=\mathfrak{P}(u,A_{\phi})$ be the elliptic function defined
by
$$
\int^{\mathfrak{P}(u,A_{\phi})}_0 \frac{d\zeta}{v(A_{\phi},\zeta)} =u. 
$$
Then it is easy to see that $\mathrm{P}(e^{-i\phi}u;A_{\phi})
=e^{i\phi}\mathfrak{P}(u;A_{\phi})$, and that
$$
\Omega_{\mathbf{a}_*}^*=\int_{\mathbf{a}_*} \frac{d\zeta}{v(A_{\phi},\zeta)}
=e^{i\phi}\Omega_{\mathbf{a}}, \quad
\Omega_{\mathbf{b}_*}^*=\int_{\mathbf{b}_*} \frac{d\zeta}{v(A_{\phi},\zeta)}
=e^{i\phi}\Omega_{\mathbf{b}}
$$
are the periods of $\Pi_{A_{\phi}}^*$.
Then the solutions given above are also written as follows.
%%%%%%%%%%%%%%%%%%%%%%%%%%%%%%%%%%%%%%%%%%%
%%%%%%%% Corollary 2.3 %%%%%%%%%%%%%%
\begin{cor}\label{cor2.3}
Suppose that $0<|\phi|<\pi/4$, and that $(1+s_1s_2)(1+s_2s_3)-1\not=0$,
$\eta_{\phi}(\mathbf{s}) \not=0,$ where $\eta_{\phi}(\mathbf{s})
=1+s_1s_2$ if $-\pi/4<\phi<0$, and $= (1+s_2s_3)^{-1}$ if $0<\phi<\pi/4.$
Then
\begin{align*}
  &y(\mathbf{s}, x)=x \mathfrak{P}(\tfrac 12 x^2  +\chi_0^* +O(x^{-\delta});
A_{\phi}),
\\
 &\chi_0^* \equiv \frac{\Omega_{\mathbf{a}_*}^*}{2\pi i}\ln ((1+s_1s_2) (1+s_2
s_3)-1)
\\
&\phantom{----------}
 -\frac{\Omega_{\mathbf{b}_*}^*}{2\pi i}\ln \eta_{\phi}(\mathbf{s}) 
+\frac{\Omega_{\mathbf{b}_*}^*}{3}(\alpha-\beta)  \mod \Omega_{\mathbf{a}_*}
^* \mathbb{Z} + \Omega_{\mathbf{b}_*}^*\mathbb{Z}
\end{align*}
as $x\to \infty$ through the cheese-like strip
$$
S_0^*(\phi,t_{\infty},\kappa_0,\delta_0)=\{x=e^{i\phi}\sqrt{2t}\,|\, 
\re t>t_{\infty}, \,\,\, |\im t|<\kappa_0 \} \setminus \bigcup_{\sigma\in
\mathcal{P}_0^*} \{|x-\sigma|<\delta_0 \}
$$
with
$\mathcal{P}_0^*=\{e^{i\phi}\sqrt{2t_0^*} \,|\, e^{2i\phi}t^*_0 +\chi^*_0 
=\pm \tfrac 13
\Omega^*_{\mathbf{b}_*}+\Omega^*_{\mathbf{a}_*}\mathbb{Z} 
+\Omega^*_{\mathbf{b}_*}\mathbb{Z} \}.$
\end{cor}
%%%%%%%%%%%%%%%%%%%%%%%%%%%%
The dependence on $\phi$ of $v(A_{\phi},\zeta)$ is on $A_{\phi}$ only
such that $A_{\phi+\pi/2}=A_{\phi}$, and hence
so is that of $\Omega_{\mathbf{a}_*}^*$, 
$\Omega_{\mathbf{b}_*}^*$ and $\mathfrak{P}(u,A_{\phi})$. 
Solutions in general directions are given as follows.
%%%%%%%%%%%%%%%%%%%%%%%%%%%%%%
%%%%%%%% Theorem 2.4 %%%%%%%%%%%%%%
\begin{thm}\label{thm2.4}
For any $n \in \mathbf{Z}$, let $\mathbf{s}_n=(s_{1+n},s_{2+n},s_{3+n},s_{4+n})$.
Suppose that $0<|\phi-\pi n/2|<\pi/4$, and that $(1+s_{1+n}s_{2+n})(1+s_{2+n}
s_{3+n})-1 \not=0,$ $\eta_{\phi}(\mathbf{s}_{n}) \not=0,$ where 
$\eta_{\phi}(\mathbf{s}_{n})
=1+s_{1+n}s_{2+n}$ if $-\pi/4<\phi-\pi n/2<0$, and $=(1+s_{2+n}s_{3+n})^{-1}$ 
if $0<\phi-\pi n/2<\pi/4.$ Then $\mathrm{P}_{\mathrm{IV}}$ admits a solution
represented as follows$:$
\begin{align*}
 &y(\mathbf{s}_n, x)=x \mathfrak{P}(\tfrac 12 x^2  +\chi_0^n +O(x^{-\delta});
A_{\phi}),
\\
 &\chi_0^n \equiv \frac{\Omega_{\mathbf{a}_*}^*}{2\pi i}
\ln ((1+s_{1+n}s_{2+n}) (1+s_{2+n}s_{3+n})-1)
\\
& \phantom{----------}
 -\frac{\Omega_{\mathbf{b}_*}^*}{2\pi i}\ln \eta_{\phi}(\mathbf{s}_{n}) 
+\frac{\Omega_{\mathbf{b}_*}^*}{3}(\alpha-\beta)  \mod \Omega_{\mathbf{a}_*}
^* \mathbb{Z} + \Omega_{\mathbf{b}_*}^*\mathbb{Z}
\end{align*}
as $x\to \infty$ through the cheese-like strip
$$
S_0^n(\phi,t_{\infty},\kappa_0,\delta_0)=\{x=e^{i\phi}\sqrt{2t}\,|\, 
\re t>t_{\infty}, \,\,\, |\im t|<\kappa_0 \} \setminus \bigcup_{\sigma\in
\mathcal{P}_0^n} \{|x-\sigma|<\delta_0 \}
$$
with
$\mathcal{P}_0^n=\{e^{i\phi}\sqrt{2t_0^n} \,|\, e^{2i\phi}t^n_0 +\chi^n_0 
=\pm \tfrac 13
\Omega^*_{\mathbf{b}_*}+\Omega^*_{\mathbf{a}_*}\mathbb{Z} 
+\Omega^*_{\mathbf{b}_*}\mathbb{Z} \}.$
\end{thm}
%%%%%%%%%%%%%%%%%%%%%%%%%%%%%%%%%%%
%%% Remark 2.5 %%%%%%%%%%%%%%
\begin{rem}\label{rem2.5}
By the single-valuedness on $\mathbb{C}$, 
$y(\mathbf{s}_n,x)$ coincides with $y(\mathbf{s}_{n+4},x)$. 
Indeed $s_ks_{k+1}=s_{k+4}s_{k+5}$ for every $k\in\mathbb{Z}.$
\end{rem}
%%%%%%%%%%%%%%%%%%%%%%%%%%%%
%%%%%%%%%%%%%%%%%%%%%%%%%%%%%%%%%%%%%%%%%%%%%%%%%%%
\subsection{Observation related to trigonometric asymptotics}\label{ssc2.3}
%%%%%%%%%%%%%%%%%%%%%%%%%%%%%%%%%%%%%%%%%%%%%%%%%%%%%%%%
Let us observe trigonometric asymptotics as degeneration of our elliptic 
solutions. For elliptic expressions of solutions in Theorems \ref{thm2.1},
\ref{thm2.2} or Corollary \ref{cor2.3} we may calculate at least formal leading 
terms of the analytic continuations to strips lying along the positive 
real axis, which is expected to behave trigonometrically as in 
\cite[Theorem 4]{Kapaev-3}. 
Such problems were discussed in \cite[Section 4]{Kitaev-3} for the first 
and the second Painlev\'e transcendents. 
\par
Write the strip $S_0(\phi,t_{\infty},\kappa_0,\delta_0)$ in Theorem \ref{thm2.1}
in the form
\begin{align*}
 S_0(\phi,t_{\infty},\kappa_0,\delta_0) &=e^{i\phi}\sqrt 
{2 T(t_{\infty},\kappa_0))} \setminus P_{\phi} 
 =\sqrt{2 e^{2i\phi} T(t_{\infty},\kappa_0))} \setminus P_{\phi}, 
\\
T(t_{\infty},\kappa_0) &=\{t\,|\, \re t>t_{\infty},\,\, |\im t|<\kappa_0\},
\end{align*}
where
$ \sqrt{2T}=\{x=\sqrt{2t} \,|\, t\in T \}$ and $ P_{\phi}=
 \bigcup_{\sigma\in\mathcal{P}_{0-}}\{|x-\sigma|<\delta_0\}$.
Let us suppose that there exists a strip $\mathcal{S}_0$ such that 
$$
\mathcal{S}_0=\{t \,|\, \re t>t_{\infty}^0, \,\, -K_0 <\im t <-K_1 \}  
 \subset \bigcup_{-\phi_0<\phi<0} e^{2i\phi} T(t_{\infty},\kappa_0),
$$
where $K_0$, $K_1$ are positive constants and $\phi_0$ a small positive
constant. Then $\sqrt{2\mathcal{S}_0}$ is a strip lying along the line 
$\sqrt{2t^0_{\infty}}<x<+\infty$, and has the properties: for each $\phi$,
\par
(1) every $t\in \mathcal{S}_0\cap e^{2i\phi}T(t_{\infty},\kappa_0)$
fulfils $t \phi \asymp 1$ with implied constants independent of $\phi$ and $t$;
\par
(2) $\sqrt{2 (\mathcal{S}_0\cap e^{2i\phi}T(t_{\infty},\kappa_0))}
 \setminus P_{\phi}\subset S_0(\phi,t_{\infty},\kappa_0,\delta_0).$
\par\noindent
By the property (2) the expression of $y(\mathbf{s},x)$ in Theorem \ref{thm2.1}
is valid in each region $\mathcal{S}_0 \cap e^{2i\phi}T(t_{\infty},\kappa_0)$.
By Remark \ref{rem2.1} this is written in the form
%%%%%%%%%%%%%%%%%%% (2.4) %%%%%%%%%%%%%%%%%
\begin{equation}\label{2.4}
(2t)^{-1/2}y(\mathbf{s},x) =\frac{e^{3i\phi}A_{\phi}}{\wp(\tau;g_2,g_3)
-\tfrac 13 e^{2i\phi} }, \quad
\tau= e^{i\phi}t+\chi_{0-}+O(t^{-\delta}).
\end{equation}
Note that, by Corollary \ref{cor8.20} and the property (1), 
$$
\Omega_{\mathbf{a}}= e^{-i\phi}\Omega^*_{\mathbf{a}_*} =\tfrac 12 \sqrt{3}i
\ln t \,(1+o(1)), \quad
\Omega_{\mathbf{b}}= e^{-i\phi}\Omega^*_{\mathbf{b}_*} =- \sqrt{3}\pi (1+o(1)),
$$
and 
$\chi_{0-}=\tfrac 14 \sqrt{3}\pi^{-1} \ln((1+s_1s_2)(1+s_2s_3)-1) \ln t 
-i\tfrac 12 \sqrt{3} (\ln(1+s_1s_2) -\tfrac 23 \pi i (\alpha-\beta))$.
The periods of $\wp(\tau)$ are given by 
$(2\omega,2\omega')=(\sqrt{3}\pi(1 +o(1)), \tfrac 12 \sqrt{3}i\ln t \,(1+o(1))$
such that $\im (\omega'/\omega) =(2\pi)^{-1} \ln t\,(1+o(1)).$ 
In \eqref{2.4}, we have $\phi=O(t^{-1})$ 
and $A_{\phi}=\tfrac 8{27}+O(t^{-1})$ (cf.~Proposition \ref{prop8.19})
as $t \to \infty$, and then $\wp(u)$ degenerates to 
%%%%%%%%%(2.5) %%%%%%%%%%%%%%% (2.6) %%%%%%%%%%%%%
\begin{align}\label{2.5}
&\wp(u)= -\tfrac 13\hat{\omega}^{-2}+ \hat{\omega}^{-2} \sin^{-2}(\hat{\omega} u)
+O(h\cos^2(2\hat{\omega}u))
\quad \text{or}
\\
\label{2.6}
&\wp(u)=-\tfrac 13 \hat{\omega}^2 -8\hat{\omega}^2 h \cos (2\hat{\omega}(u-
\omega'))+O(h^2\cos^2(2\hat{\omega}(u-\omega')))
\end{align}
with $\hat{\omega}=\pi/(2\omega)$ and $h=e^{i\pi {\omega'/\omega}}$ 
\cite{HC}, \cite{WW}. In
$\mathcal{S}_0\cap (\bigcup_{-\phi_0<\phi<0} e^{2i\phi}T(t_{\infty},\kappa_0))$, 
from \eqref{2.5} we have the trigonometric expression
\begin{align*}
\frac{1}{(2t)^{-1/2}y(\mathbf{s},x)+\tfrac 23} &=
\frac 12 + \frac 12 (e^{i 2\tau/\sqrt{3}} +e^{-i 2\tau/\sqrt{3}}) 
(1+O(t^{-\varepsilon})),
\\
{2\tau}/{\sqrt{3}} =  {2t}/{\sqrt{3}} &+({2\pi})^{-1}
\ln((1+s_1s_2)(1+s_2s_3) -1)\ln t -\tfrac{2}3 \pi (\alpha-\beta) +O(1),
%% \\
%% |\im \ln((1+s_1s_2)& (1+s_2s_3)  -1) |<\pi/2,
\end{align*}
which agrees with \cite[Theorem 4, (30)]{Kapaev-3} up to constants.
From \eqref{2.6} we have
\begin{align*}
y(\mathbf{s},x) +\tfrac 23 &= 
 2 t^{-1/2} (e^{i 2\hat{\tau}/\sqrt{3}} +e^{-i 2\hat{\tau}/\sqrt{3}})
(1+O(t^{-\varepsilon})),
\\
{2\hat{\tau}}/{\sqrt{3}} &=  {2t}/{\sqrt{3}} +({2\pi})^{-1}
\ln(1-(1+s_1s_2)(1+s_2s_3)) \ln t -\tfrac{2}3 \pi (\alpha-\beta) +O(1),
%% \\
%% |\im  \ln(&1-(1+s_1s_2) (1+s_2s_3) ) |<\pi/2,
\end{align*}
which agrees with \cite[Theorem 4, (28)]{Kapaev-3} and also with
\cite[Theorem 1.1, (1.11)]{I-Kapaev} up to constants.
In the strip $0<K_0<\im t<K_1$ a similar argument is possible for the solution
in Theorem \ref{thm2.2}. The argument above, though not justified, suggests
information about degeneration to the trigonometric asymptotics.
%%%%%%%%%%%%%%%%%%%%%%%%%%%%%%%%%%%%%%%%%%%%%
%%%%%% Section 3 %%%%%%%%%%%%%%%%%
%%%%%%%%%%%%%%%%%%%%%%%%%%%%%%%%%%
\section{Basic facts}\label{sc3}
\subsection{Monodromy data}
The monodromy data and the monodromy manifold $\mathcal{M}^*_0(\alpha,\beta)$ 
described in Section \ref{sc2} have the following properties.
%%%%%%%%%%%%%%%%%%%%%%%
%%%% Proposition 3.1 %%%%%%%
\begin{prop}\label{prop3.1}
For each $m \in \mathbb{Z}$, 
\begin{align*}
&S_{1+m}S_{2+m}S_{3+m}S_{4+m}e^{-i\pi(\alpha-\beta)
\sigma_3}\sigma_3
\\
&\phantom{---} =
\begin{cases}
(ES_1\cdots S_m)^{-1} \sigma_3 M^{-1} (ES_1\cdots S_m) \quad &
\text{if $m\ge 1;$}
\\[0.4cm]
(ES_0^{-1}\cdots S_{m+1}^{-1})^{-1}\sigma_3 M^{-1}(ES_0^{-1}\cdots S_{m+1}^{-1})
\quad & \text{if $m\le -1.$} 
\end{cases}
\end{align*}
\end{prop} 
%%%%%%%%%%%%%%%%%
\begin{proof}
The semi-cyclic condition $S_1S_2S_3S_4=E^{-1}\sigma_3M^{-1}E e^{i\pi(\alpha-\beta)\sigma_3}\sigma_3$
with $S_{k+4}=e^{-i\pi(\alpha-\beta)\sigma_3}\sigma_3 S_k \sigma_3 e^{i\pi(\alpha
-\beta)\sigma_3}$ yields
\begin{align*}
S_2S_3S_4S_{1+4} =& S_1^{-1}E^{-1}\sigma_3 M^{-1} E e^{i\pi(\alpha-\beta)\sigma
_3 }\sigma_3e^{-i\pi(\alpha-\beta)\sigma_3}\sigma_3 S_1 \sigma_3 e^{i\pi(\alpha
-\beta)\sigma_3}
\\
 =& (ES_1)^{-1}\sigma_3 M^{-1} (E  S_1)e^{i\pi(\alpha-\beta)\sigma_3} \sigma_3, 
\\
S_0S_1S_2S_{3} =& S_0E^{-1}\sigma_3 M^{-1} E e^{i\pi(\alpha-\beta)\sigma
_3 }\sigma_3 S_4^{-1}
\\
 =& (ES_0^{-1})^{-1}\sigma_3 M^{-1} E S_0^{-1}S_0e^{i\pi(\alpha-\beta)\sigma_3}
\sigma_3 S_4^{-1}
\\
 =& (ES_0^{-1})^{-1}\sigma_3 M^{-1} (E S_0^{-1})e^{i\pi(\alpha-\beta)\sigma_3}
\sigma_3. 
\end{align*}
Repeating this procedure, we obtain the proposition.
\end{proof}
%%%%%%%%%%%%%%%%%%%%%%%%%%%%%%%%%%%%
%%%%%%%%% Proposition 3.2 %%%%%%%%%%
\begin{prop}\label{prop3.2}
The surface $\mathcal{M}^*_0(\alpha,\beta)$ has a singular point $\mathbf{s}
_{\mathrm{sing}}$ if and only if
$\alpha -\tfrac 12 \in \mathbb{Z}$, and then $\mathbf{s}_{\mathrm{sing}}=
(\mathfrak{s}_{12}, \mathfrak{s}_{23}, \mathfrak{s}_{34}, \mathfrak{s}_{41})
=(e^{-i\pi\beta}-1, e^{i\pi\beta}-1,e^{-i\pi\beta}-1, e^{-i\beta\pi}
-e^{-2i\beta\pi}).$
\end{prop}
%%%%%%%%%%%%%%%%%%%%%%%%%%
\begin{proof}
Write the surface in the form
$$
f=e^{i\pi(\beta-\alpha)}(xz+u-1)- e^{i\pi(\alpha-\beta)}y
+2i\sin \pi\alpha=0,\quad (x-1)(z-1)=(y-1)(u-1)
$$
with $(x,y,z,u)=(1+\mathfrak{s}_{12},1+\mathfrak{s}_{23},1+\mathfrak{s}_{34},
1+\mathfrak{s}_{41}),$
and examine when $f_x,$ $f_y,$ $f_z$, $f_u$ and $f$ have a common zero.
If, say $y-1\not=0$, by using $u-1=(x-1)(z-1)/(y-1)$, we have $\alpha-\tfrac 12
\in \mathbb{Z}$ and $x=z=1/y=e^{-i\pi\beta}.$
\end{proof}
%%%%%%%%%%%%%%%%%%%%%%%%%%%%%%%%%%%%%%%%%%%%%
%%%% Remark 3.1 %%%%%%%%%%%%%
\begin{rem}\label{rem3.1}
Equation P$_{\mathrm{IV}}$ admits a one-parameter family of classical solutions 
(respectively, a rational solution) if and only if 
$\{\alpha-\tfrac 12,\, \tfrac 12\beta,\, \alpha-\tfrac 12\beta-\tfrac 12 \}
\cap \mathbb{Z} \not=\emptyset$ 
(respectively, $\{(\alpha-\tfrac 12, \tfrac 12 \beta ),
(\alpha\pm \tfrac 16,2\alpha-\tfrac 12 \beta) \}\cap \mathbb{Z}^2\not=
\emptyset$) \cite{Gromak-1}, \cite{Gromak-2}, \cite[Chap.~6]{GLS},
\cite{Murata}, \cite{UW}.  
\end{rem}
%%%%%%%%%%%%%%%%%%%%%%%%%
\subsection{Isomonodromy linear system}
Let us transform system \eqref{1.1} into a form suitable to our calculation.
By 
$$
\Psi=\mathrm{u}^{(1/2)\sigma_3} x^{-(1/4)\sigma_3}\Phi, \quad
\mathrm{u}\mathrm{v}=y, \quad  
 \frac{\mathrm{u}_x}{\mathrm{u}}= x \mathrm{z},
$$
system \eqref{1.1} is changed into
\begin{align*}
\frac{d\Phi}{d\xi}=& \Biggl(\Bigl(\frac {\xi^3}2  +(x+y)\xi + \frac{\alpha}{\xi}
\Bigr)\sigma_3 + ix^{1/2} \begin{pmatrix}  0 & \xi^2 +x\mathrm{z} +2x \\
         x^{-1}(y\xi^2 -xy\mathrm{z} +y^2 +\beta ) & 0 \end{pmatrix}
\Biggr)\Phi,
\\
 y' =& 2x y\mathrm{z} + 2xy -y^2- \beta.
\end{align*}
The change of variables
$$
\tau=x^2, \quad y=x\eta, \quad \xi=x^{1/2}\tilde{\xi}
$$
takes this system to
\begin{align*}
\frac{d\Phi}{d\tilde{\xi}}=&  \tau \Biggl( \Bigl(\frac{\tilde{\xi}^3}2+(1+\eta)
\tilde{\xi} +\frac{\alpha \tau^{-1}}{\tilde{\xi}}\Bigr)\sigma_3
+i\begin{pmatrix} 0 & \tilde{\xi}^2 +\mathrm{z} +2 \\
\eta \tilde{\xi}^2 -\eta\mathrm{z} + \eta^2 +\beta\tau^{-1} & 0
\end{pmatrix}
\Biggr)\Phi,
\\
\tau \frac{d\eta}{d\tau} =&\tau\eta(\mathrm{z}+1)-\frac 12 ({\tau\eta^2} 
+\eta+\beta).
\end{align*} 
The further substitution
$$
\tau=2e^{2i\phi} t, \quad \eta=e^{-i\phi}\psi, \quad  e^{i\phi}\mathrm{z}
=\mathfrak{z}, \quad \lambda=e^{i\phi/2}\tilde{\xi}
$$
leads to
%%%%%%%%%%%%%%% (3.1), (3.2) %%%%%%%%
\begin{align}\label{3.1}
&\frac{d\Phi}{d\lambda}  = t \mathcal{B}(t,\lambda)\Phi, \qquad
\mathcal{B}(t,\lambda)=b_3\sigma_3 +b_1\sigma_1 +b_2\sigma_2, 
\\
\label{3.2}
 &2(\mathfrak{z}+  e^{i\phi})=e^{-i\phi}\frac{\psi_t}{\psi} +\psi 
+\tfrac 12 (e^{-i\phi} + \beta\psi^{-1})t^{-1}
\end{align}
with
\begin{align*}
&b_1= ie^{-i\phi/2}\left((\psi+e^{i\phi})\lambda^2-(\psi-e^{i\phi})\mathfrak{z}
+\psi^2 +2 e^{2i\phi} +\tfrac 12 \beta t^{-1}\right),
\\
&b_2= e^{-i\phi/2}\left((\psi-e^{i\phi})\lambda^2-(\psi+e^{i\phi})\mathfrak{z}
+\psi^2 -2 e^{2i\phi} +\tfrac 12 \beta t^{-1}\right),
\\
&b_3= \lambda^3+2(e^{i\phi}+\psi)\lambda
+\alpha t^{-1}\lambda^{-1} 
\end{align*}
In \eqref{3.2} as a resulting equation, 
$\psi_t$ denotes the derivative $(d/dt)\psi.$
Let us now change the meaning of $\psi_t$ in such a way that 
$\psi_t$ is an arbitrary function not 
necessarily the derivative, and in what follows suppose that
system \eqref{3.1} is equipped with $\mathfrak{z}$ containing 
such $\psi_t.$ 
Then the isomonodromy property of \eqref{1.1} is converted to that
of \eqref{3.1}.
%%%%%%%%%%%%%%%%%%%%%%%%%%%%%%%%
%%%%%%%% Proposition 3.3 %%%%%%%%%%%%%
\begin{prop}\label{prop3.3}
The monodromy data of \eqref{3.1} is invariant under a small change of $t$
if and only if $\psi_t=(d/dt)\psi$ holds in \eqref{3.2} and $y(x)= e^{-i\phi}x\psi$ 
with $2e^{2i\phi}t=x^2$ solves $\mathrm{P}_{\mathrm{IV}}$.
\end{prop}
%%%%%%%%%%%%%%%%%%%%%%%%%%%%%%%%%%%%%%%%%
For $k\in \mathbb{Z}$ system \eqref{3.1} admits canonical solutions
%%%%%%%%%%%%%%%%%%%%%%%%%%%
%%%%%% (3.3) %%%%%%
\begin{equation}\label{3.3}
\Phi_k^{\infty}(\lambda)=(I+O(\lambda^{-1})) \exp((\tfrac 14 t\lambda^4
+e^{i\phi}t\lambda^2 +(\alpha-\beta)\ln \lambda)\sigma_3)
\end{equation}
as $\lambda \to \infty$ through the sector $|\arg(t^{1/4}\lambda)+\tfrac{\pi}
8 -\tfrac{\pi}4k|<\tfrac{\pi}4$. The Stokes matrices are defined by
$\Phi_{k+1}^{\infty}(\lambda)=\Phi_k^{\infty}(\lambda)S_k^*$.
Recalling the Stokes matrices $S_k$ with respect to $\Psi^{\infty}_k(\xi)$
solving \eqref{1.1}, we have the following relation.
%%%%%%%%%%%%%%%%%%%%%%%%%%%%%%%%%%%%%
%%%%% Proposition 3.4 %%%%%%%
\begin{prop}\label{prop3.4}
For every $k\in \mathbb{Z}$, 
$$
S_k=\mathrm{u}^{(1/2)\sigma_3}x^{-(1/4)\sigma_3}
(e^{-i\phi/2}x^{1/2})^{-(\alpha-\beta)\sigma_3} S_k^*
(e^{-i\phi/2}x^{1/2})^{(\alpha-\beta)\sigma_3}\mathrm{u}^{-(1/2)\sigma_3}
x^{(1/4)\sigma_3}.
$$
\end{prop}
%%%%%%%%%%%%%%%%%%%%%%%%%%
\begin{proof}
The relation $\Psi^{\infty}_k(\xi)=\mathrm{u}^{(1/2)\sigma_3}x^{-(1/4)\sigma_3}
\Phi^{\infty}_k(\lambda)(e^{-i\phi/2}x^{1/2})^{(\alpha-\beta)\sigma_3}
\mathrm{u}^{-(1/2)\sigma_3}x^{(1/4)\sigma_3}$ yields the conclusion.
\end{proof}
The Stokes coefficients $s^*_k$ for $\Phi^{\infty}_k(\lambda)$ are given by
$$
S_{2l-1}^*=\begin{pmatrix} 1 & s^*_{2l-1} \\  0 & 1 \end{pmatrix}, \quad
S_{2l}^*=\begin{pmatrix} 1 & 0 \\ s^*_{2l}  & 1 \end{pmatrix} \quad
(l\in \mathbb{Z}). 
$$
%%%%%%%%%%%%%%%%%%%%%%%%%%
%%%% Corollary 3.5 %%%%%%%
\begin{cor}\label{cor3.5}
For any $k, j\in \mathbb{Z}$, $s_k^*s_{k+2j+1}^*=s_k s_{k+2j+1}$.
\end{cor}
%%%%%%%%%%%%%%%%%%%%%%%%%%%%%%%%%%%%%
%%%%% Section 4 %%%%%%%%%%%%%%%%%%
\section{Turning points, Stokes graph, WKB analysis}\label{sc4}
%%%%%%%%%%%%%%%%%%%%%%%%%%%%%%%%%%%%%%%
For system \eqref{3.1} we will treat the direct monodromy problem by 
WKB analysis. 
Let us start with the characteristic roots $\pm
\mu(t,\lambda)$ of $\mathcal{B}(t,\lambda)$ constituting the essential
part of the WKB solution (cf.~Proposition \ref{prop4.2}). 
By $\mu(t,\lambda)^2=b_1^2+b_2^2 +b_3^2$, we have
%%%%%%%%%% (4.1) %%%%%%%%%%%
\begin{equation}\label{4.1}
\mu(t,\lambda)^2=\lambda^6+4e^{i\phi}\lambda^4+(4e^{2i\phi} +2(\alpha-\beta)
t^{-1})\lambda^2+4 e^{3i\phi} a_{\phi}+\alpha^2t^{-2}\lambda^{-2}
\end{equation}
with 
%%%%%%%%% (4.2) %%%%%%%%%%
\begin{align}\label{4.2}
4 e^{3i\phi} a_{\phi}= & 4 e^{3i\phi}a_{\phi}(t)= 
e^{-2i\phi}{(\psi_t)^2}{\psi^{-1}} -(\psi+2e^{i\phi})^2\psi
\\   
\notag
& +(e^{-2i\phi}\psi_t +(4\alpha-\beta)\psi +2(2\alpha-\beta)e^{i\phi})t^{-1}
+\tfrac 14 (e^{-2i\phi}\psi-\beta^2 \psi^{-1})t^{-2} .
\end{align}
To draw Stokes graphs it is necessary to know the location of the turning
points. 
Now we note the following facts on the solution $A_{\phi}$ of the Boutroux
equation \eqref{2.2} for $|\phi|<\pi/4$ (Proposition \ref{prop8.15}):
\par
(i) $A_{0}=\frac{8}{27},$ and then $w(A_0,z)=z(z +\frac 23)^2(z+\frac 8{3});$
\par
(ii) $A_{\pm \pi/4}=0,$ and then $w(A_{\pm \pi/4},z)=z^2(z+2e^{\pm i\pi/4})^2;$
\par
(iii) for $0<|\phi|<\pi/4,$ $0<\re A_{\phi} <\frac 8{27}$ and 
$w(A_{\phi},z)$ does not degenerate.
\par\noindent
Then by Corollary \ref{cor8.2}, for $|\phi|<\pi/4$, the zeros of $w(A_{\phi},z)$
may be numbered in such a way that $\re e^{-i\phi}z_5 \le \re e^{-i\phi}z_3 \le 
\re e^{-i\phi}z_1 $ and that $(z_1,z_3,z_5)\to (-\frac 23, -\frac 23,
-\frac 83)$ as $\phi \to 0$, and $\to (0, -2e^{\pm i\pi/4}, -2e^{\pm i\pi/4})$
as $\phi \to \pm \pi/4$ (cf. Section \ref{sc2}). 
\par
Our WKB analysis is carried out under the supposition $a_{\phi}(t)= A_{\phi}
+O(t^{-1})$ as $t\to\infty$, that is, \eqref{5.1}. The characteristic root 
satisfies $\lambda\mu(t,\lambda) \to w(A_{\phi},\lambda^2)$ as $t\to \infty$.  
Let $\lambda_j$ $(1\le j \le 6)$ be turning points of $\mu(t,\lambda)$ such
that  
\begin{align*}
&\lambda_1 \to z_1^{1/2}, \,\,\, \lambda_3\to z_3^{1/2}, \,\,\, \lambda_5
\to z_5^{1/2}\,\,\, \text{as}\,\, t \to \infty,  
\\
&\lambda_2=-\lambda_1, \,\,\, \lambda_4=-\lambda_3, \,\,\, \lambda_6
=-\lambda_5,
\\
& |\arg \lambda_j -\tfrac {\pi}2 |< \tfrac {\pi}4 \,\,\, (j=1,3,5) \,\,\,
\text{for sufficiently large $t$}.
\end{align*}
\par
The algebraic function $\mu(t,\lambda)$ is written in the form
\begin{align*}
\mu(t,\lambda)&= \sqrt{\smash{ 
\lambda^6}+ 4e^{i\phi}\lambda^4 +(4e^{2i\phi}+2(\alpha-\beta)t^{-1})\lambda^2
+4e^{3i\phi}\smash{ a_{\phi}}+ \alpha^2t^{-2}\lambda^{-2}}
\\
&= \lambda^{-1}\sqrt{(\lambda^2-\lambda_1^2)(\lambda^2-\lambda_3^2)
(\lambda^2-\lambda_5^2)(\lambda^2-\lambda_0^2)}
\\
&=\lambda^{-1}\sqrt{(\lambda-\lambda_1)(\lambda-\lambda_2)
(\lambda-\lambda_3)(\lambda-\lambda_4)(\lambda-\lambda_5)(\lambda-\lambda_6)
(\lambda-\lambda_{0})(\lambda+\lambda_{0})},
\\
& \lambda_0=O(t^{-1}),
\end{align*}
which is considered on the two sheeted
Riemann surface $\mathcal{R}_{t}=\mathcal{R}^+_{t} \cup \mathcal{R}^-_{t}$ glued along the
cuts $[\lambda_5,\lambda_3]$, $[\lambda_1,\lambda_0],$ $[-\lambda_0,\lambda_2]$,
$[\lambda_4,\lambda_6]$
with $\mathcal{R}^{\pm}_{t}=\mathbb{C}\setminus ( 
[\lambda_5,\lambda_3]\cup[\lambda_1,\lambda_0]\cup[-\lambda_0,\lambda_2]
\cup[\lambda_4,\lambda_6])$.
The branches of $\mu(\infty,\lambda)$ with $a_{\phi}(\infty)=A_{\phi}$ (by
\eqref{5.1}) is chosen in such a way that
$\mu(\infty,\lambda)=\lambda^3(1+O(\lambda^{-2})$ as $\lambda \to\infty$ on
the upper sheet $\mathcal{R}^+_{t}.$
For $0<|\phi|<\pi/4$, $w(A_{\phi},z)$ does not degenerate and 
neither does $\mathcal{R}_t$.
Then each turning point satisfies $\lambda_j(t)-\lambda_j(\infty)=O(t^{-1})$
as $t\to \infty.$
\par
In what follows we treat a Stokes graph with $t=\infty$, and the limit
turning point $\lambda_j(\infty)$ is simply denoted by $\lambda_j$.
By $z=\lambda^2$ the algebraic function $w(A_{\phi},z)$ on 
$\Pi_{A_{\phi},\phi}$ is mapped to 
$\lambda \mu(\infty,\lambda)$ on $\mathcal{R}_{\infty}$, in which
\begin{align*}
\mu(\infty,\lambda)=&\sqrt{\lambda^6+4e^{i\phi}\lambda^4 +4e^{2i\phi}
\lambda^2 +4 e^{3i\phi} A_{\phi}}
=\sqrt{(\lambda^2-\lambda_1^2)(\lambda^2-\lambda_3^2)(\lambda^2-\lambda_5^2)}
\\
=&\sqrt{(\lambda^2-z_1)(\lambda^2-z_3)(\lambda^2-z_5)}.
\end{align*}
The Stokes graph on $\mathcal{R}_{\infty}$ consists of vertices 
and Stokes curves, where the Stokes curve is defined by 
$\re \int^{\lambda}_{\lambda_j} \mu(\infty,\lambda)d\lambda=0,$ and the
vertices are turning points and singular points.
In our case the turning points and the Stokes curves have the
following properties:
\par
(i) if $\phi=0$, then $\lambda_1=\lambda_3= \tfrac 13\sqrt{6}i,$ 
$\lambda_5
= \frac 23\sqrt{6} i,$ that is, $\lambda_{1,3}=\tfrac 13\sqrt{6}i$ is a 
double turning point, and
if $\phi=\pm \pi/4$, then $\lambda_1=0,$ $\lambda_3=\lambda_5= 
\sqrt{2}i e^{\pm i \pi/8}$, that is, $\lambda_{3,5}=\sqrt{2}i e^{\pm i \pi/8}$ 
is a double turning point;
\par
(ii) if $\phi$ is close to $0$, then the double turning point $\lambda_{1,3}$
resolves into 
\begin{align*}
&\lambda_{2\pm 1} =\tfrac 13 \sqrt{6}i \pm \phi_* e^{i3\pi/4} +O(\phi_*^2) 
\quad \text{for $\phi >0$},
\\
&\lambda_{2\pm 1} =\tfrac 13 \sqrt{6}i \pm \phi_* e^{i\pi/4}+O(\phi_*^2)
\quad \text{for $\phi <0$},
\end{align*}
where
$\phi_*=\phi_*(\phi)$ is such that $\phi_* \ge 0$ and $\phi_* =o(1)$ as
$\phi \to 0$ (cf.~Propositions \ref{prop8.17} and \ref{prop8.18}); 
\par 
(iii) by the Boutroux equations (2.2) with $z=\lambda^2$,
$$
\re \int_{\lambda_1}^{\lambda_3} \mu(\infty,\lambda) d\lambda=0, \quad
\re \int_{\lambda_3}^{\lambda_5} \mu(\infty,\lambda) d\lambda=0,
$$
implying the existence of Stokes curves joining $\lambda_1$ to $\lambda_3$, 
and $\lambda_3$ to $\lambda_5$;
\par
(iv) the Stokes curves tending to $\infty$ are asymptotic to the rays
$\arg \lambda=\frac{\pi}8+ \frac{\pi}4k$ $(1\le k\le 8)$.
\par
Taking these facts into account, we may draw the limit Stokes graphs 
for $|\phi|<\pi/4$ as in Figure \ref{stokes}, in which the cuts 
$[\lambda_5,\lambda_3]$,
$[\lambda_1,\lambda_2]$, $[\lambda_4,\lambda_6]$ are omitted. 
%%%%%%%%%%%%%%%%%%%%%%%%%%%%%%%%%%%%%%%%%%%%%%%%%%%%%
%%%%%%%%%%%%%%%%%%%%%%%%%%%%%%%%%%%%%%%%%%%%%%%%%%%%
%%%%%%%%%%%% Figure 4.1 %%%%%%%%%%%%%%%%%%%%
%%%%%%%%%%%%%%%%%%%%%%%%%%%%%%%%%%%%%%%%
%%%%%%%%%%%%%%%%%%%%%%%%%%%%%%%%%%%%%%%%
{\small
\begin{figure}[htb]
\begin{center}
\unitlength=0.65mm
%%%%%%%%%%%%%%%%%%%%%%%%%%%%%%%%%%
\begin{picture}(80,90)(-40,-45)

 \put(0,32){\makebox{$\lambda_5$}}
 \put(-6,-36){\makebox{$\lambda_6$}}
 \put(5,12){\makebox{$\lambda_{3}$}}
 \put(-10,13){\makebox{$\lambda_{1}$}}
 \put(-9,-15){\makebox{$\lambda_{4}$}}
 \put(4,-17){\makebox{$\lambda_{2}$}}
 \put(3,0){\makebox{$0$}}

 \put(0,0){\circle{2.3}}

 \put(5,28){\circle*{2.3}}
 \put(-5,-28){\circle*{2.3}}

 \put(5,19.5){\circle*{2.3}}
 \put(-5,-19.5){\circle*{2.3}}

 \put(-3,10){\circle*{2.3}}
 \put(3,-10){\circle*{2.3}}

\thicklines
 \qbezier (5,19.5) (16,14) (33,19)
 \qbezier (-5,-19.5) (-16,-14) (-33,-19)

 \qbezier (5,19.5) (6.2,25) (5,28)
 \qbezier (5,19.5) (-2,15) (-3,10)

 \qbezier (-5,-19.5) (-6.2,-25) (-5,-28)
 \qbezier (-5,-19.5) (2,-15) (3,-10)

 \qbezier (-3,10) (-12,7) (-33,19)
 \qbezier (-3,10) (-1,7) (0,0)

 \qbezier (3,-10) (12,-7) (33,-19)
 \qbezier (3,-10) (1,-7) (0,0)

 \qbezier (5,28) (8,27) (14, 40)
 \qbezier (5,28) (-5,28) (-13, 40)

 \qbezier (-5,-28) (-8,-27) (-14, -40)
 \qbezier (-5,-28) (5,-28) (13, -40)

\put(-10,-53){\makebox{$-\pi/4<\phi<0$}}
\end{picture}
%%%%%%%%%%%%%%%%%%%%%%%%%%%%%%%%%%
\qquad\qquad
\begin{picture}(80,80)(-40,-45)

 \put(-6,32){\makebox{$\lambda_5$}}
 \put(3,-36){\makebox{$\lambda_6$}}
 \put(-9,12){\makebox{$\lambda_{3}$}}
 \put(5,13){\makebox{$\lambda_{1}$}}
 \put(3,-15){\makebox{$\lambda_{4}$}}
 \put(-9,-16){\makebox{$\lambda_{2}$}}
 \put(-6,0){\makebox{$0$}}

 \put(0,0){\circle{2.3}}

 \put(-5,28){\circle*{2.3}}
 \put(5,-28){\circle*{2.3}}

 \put(-5,19.5){\circle*{2.3}}
 \put(5,-19.5){\circle*{2.3}}

 \put(3,10){\circle*{2.3}}
 \put(-3,-10){\circle*{2.3}}

\thicklines
 \qbezier (-5,19.5) (-16,14) (-33,19)
 \qbezier (5,-19.5) (16,-14) (33,-19)

 \qbezier (-5,19.5) (-6.2,25) (-5,28)
 \qbezier (-5,19.5) (2,15) (3,10)

 \qbezier (5,-19.5) (6.2,-25) (5,-28)
 \qbezier (5,-19.5) (-2,-15) (-3,-10)

 \qbezier (3,10) (12,7) (33,19)
 \qbezier (3,10) (1,7) (0,0)

 \qbezier (-3,-10) (-12,-7) (-33,-19)
 \qbezier (-3,-10) (-1,-7) (0,0)

 \qbezier (-5,28) (-8,27) (-14, 40)
 \qbezier (-5,28) (5,28) (13, 40)

 \qbezier (5,-28) (8,-27) (14, -40)
 \qbezier (5,-28) (-5,-28) (-13, -40)

\put(-10,-53){\makebox{$0<\phi<\pi/4$}}
\end{picture}
%%%%%%%%%%%%%%%%%%%%%%%%%%%%%%
%%%%%%%%%%%%%%%%%%%%%%%%%%%%%%
\vskip1.6cm
%%%%%%%%%%%%%%%%%%%%%%%%%%%%%%%
\unitlength=0.58mm
\begin{picture}(80,80)(-40,-45)

 \put(14,26){\makebox{$\lambda_{3,5}$}}
 \put(4,0){\makebox{$\lambda_{1,2}$}}
 \put(-6,-24){\makebox{$\lambda_{4,6}$}}
 \put(-3,3){\makebox{$0$}}

\put(9.94,24){\circle*{2.3}}
\put(-9.94,-24){\circle*{2.3}}

\put(0,0){\circle*{2.3}}

\thicklines
 \qbezier (9.94,24) (17,20.5) (33,21)
 \qbezier (-9,40) (-2,28) (9.94,24)
 \qbezier (-9.94,-24) (-17,-20.5) (-33,-21)
 \qbezier (9,-40) (2,-28) (-9.94,-24)
 \qbezier (33,-13.67) (0,0) (-33,13.67)

 \qbezier (16.57,40) (0,0) (-16.57, -40)

\put(-10,-53){\makebox{$\phi=-\pi/4$}}
\end{picture}
\quad
%%%%%%%%%%%%%%%%%%%%%%%%%%%%%%%
\begin{picture}(80,80)(-40,-45)

 \put(-2,34){\makebox{$\lambda_5$}}
 \put(-2,-38){\makebox{$\lambda_6$}}
 \put(2,18){\makebox{$\lambda_{1,3}$}}
 \put(-12,-12){\makebox{$\lambda_{2,4}$}}
 \put(2,0){\makebox{$0$}}

 \put(0,0){\circle{2.3}}
% \put(40,0){\circle{1}}
% \put(-40,0){\circle{1}}
% \put(40,16.57){\circle{2}}
% \put(-40,16.57){\circle{2}}
% \put(16.57,40){\circle{2}}
% \put(-16.57,40){\circle{2}}

\put(0,30){\circle*{2.3}}
\put(0,-30){\circle*{2.3}}

\put(0,15){\circle*{2.3}}
\put(0,-15){\circle*{2.3}}

\thicklines
 \qbezier (0,15) (16,15) (32,21)
 \qbezier (0,15) (-16,15) (-32,21)
 \qbezier (0,-15) (16,-15) (32,-21)
 \qbezier (0,-15) (-16,-15) (-32,-21)

 \qbezier (0,30) (0,0) (0, -30)
 \qbezier (0,30) (7,30) (13, 40)
 \qbezier (0,30) (-7,30) (-13, 40)
 \qbezier (0,-30) (7,-30) (13, -40)
 \qbezier (0,-30) (-7,-30) (-13, -40)

\put(-10,-53){\makebox{$\phi=0$}}
\end{picture}
%%%%%%%%%%%%%%%%%%%%%%%%%%%%%%
\quad
%%%%%%%%%%%%%%%%%%%%%%%%%%%%%%%
\begin{picture}(80,80)(-40,-45)

 \put(-24,25){\makebox{$\lambda_{3,5}$}}
 \put(-14,1){\makebox{$\lambda_{1,2}$}}
 \put(-5,-23){\makebox{$\lambda_{4,6}$}}
 \put(1,3){\makebox{$0$}}

\put(-9.94,24){\circle*{2.3}}
\put(9.94,-24){\circle*{2.3}}

\put(0,0){\circle*{2.3}}

\thicklines
 \qbezier (-9.94,24) (-17,20.5) (-33,21)
 \qbezier (9,40) (2,28) (-9.94,24)
 \qbezier (9.94,-24) (17,-20.5) (33,-21)
 \qbezier (-9,-40) (-2,-28) (9.94,-24)
 \qbezier (-33,-13.67) (0,0) (33,13.67)

 \qbezier (-16.57,40) (0,0) (16.57, -40)

\put(-10,-53){\makebox{$\phi=\pi/4$}}

\end{picture}
%%%%%%%%%%%%%%%%%%%%%%%%%%%%%%%
\vskip0.7cm
%%%%%%%%%%%%%%%%%%%%%%%%%%%%%%%%%%
%%%%%%%%%%%%%%%%%%%%%%%%%%%%%%%%%
\end{center}
\caption{Limit Stokes graphs on $\mathcal{R}_{\infty}^+$}
\label{stokes}
\end{figure}
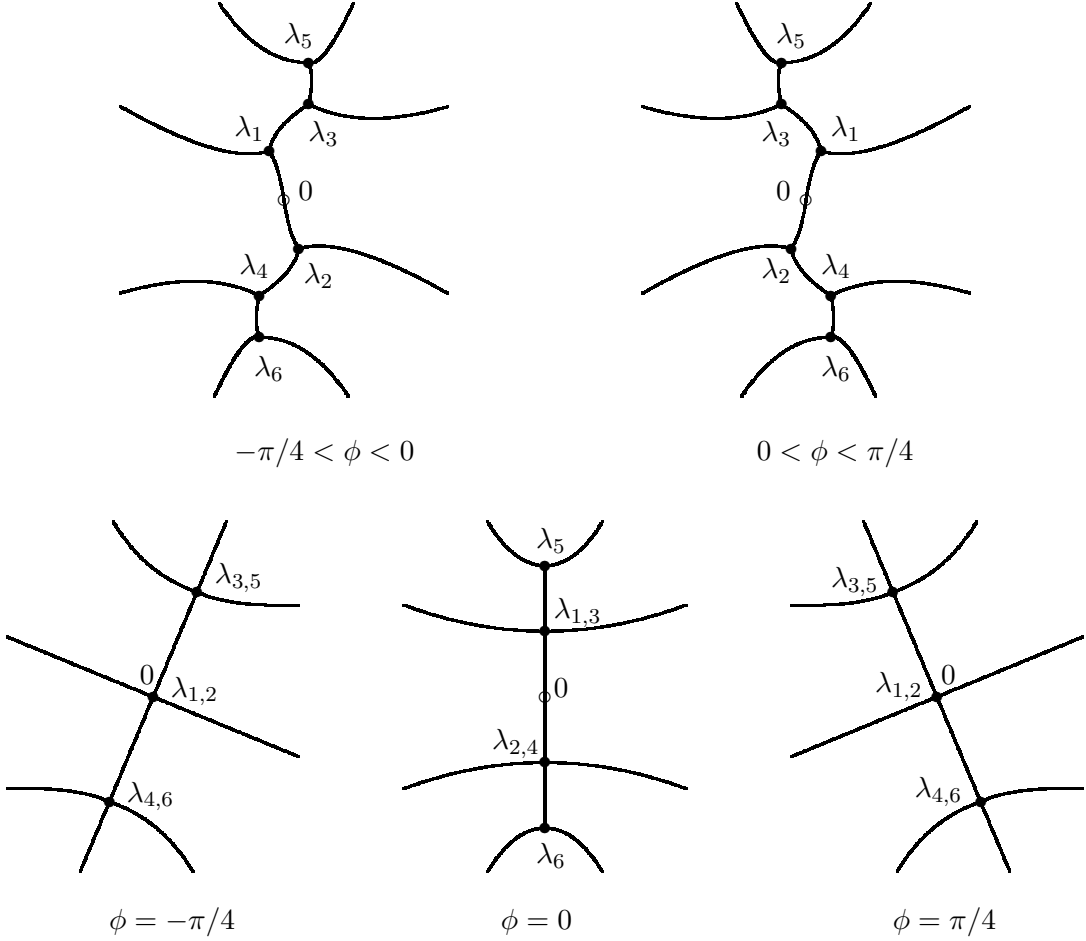
}
%%%%%%%%%%%%%%%%%%%%%%%%%%%%%%%%%%%%%%%%%%%%%%%%%%%
%%%%%%%%%%%%%%%%%%%%%%%%%%%%%%%%%%%%%%%%%%%%%%%%
An unbounded domain $D \subset \mathcal{R}_{\infty}$ is said to be canonical
if, for every $\lambda \in D$, there exist contours ${C}_{\pm}(\lambda) \subset
D$ terminating in $\lambda$ such that
$$
\re \int_{\lambda^-_0}^{\lambda}\mu(\lambda)d\lambda \to -\infty,  \quad
\re \int_{\lambda^+_0}^{\lambda}\mu(\lambda)d\lambda \to +\infty
$$
as $\lambda_0^- \to \infty$ along $C_-(\lambda)$ and as $\lambda_0^+ \to \infty$
along $C_+(\lambda)$, respectively (see \cite{F}, \cite{FIKN}). The interior 
of a canonical domain contains exactly one Stokes curve, and the boundary 
consists of Stokes curves. System \eqref{3.1} admits the following WKB solution 
(\cite{F}, \cite[Theorem 7.2]{FIKN}, \cite[Proposition 3.8]{S}). 
%%%%%%%%%%%%%%%%%%%%%%%%%%%%%%%%%%%%%%%%%%%%
%%%%% Proposition 4.2 %%%%%%%%%%%%%
\begin{prop}\label{prop4.2}
In a canonical domain system \eqref{3.1} with $\mathcal{B}(t,\lambda)
=b_1\sigma_1+b_2\sigma_2 +b_3\sigma_3$ admits an asymptotic solution 
expressed as
$$
\Psi_{\mathrm{WKB}}(\lambda)=T(I+O(t^{-\delta})) \exp\Bigl(\int^{\lambda}
_{\lambda_*} \Lambda(\tau) d\tau \Bigr), \quad
T=\begin{pmatrix}  1 & \frac{b_3-\mu}{b_1+ib_2} \\
               \frac{\mu-b_3}{b_1-ib_2}  &  1
\end{pmatrix},
$$
as long as $|b_1 \pm ib_2|^{-1} \ll 1,$ $|\psi|+|\mathfrak{z}| \ll 1,$
$|\lambda-\lambda_j| \gg t^{-(2/3)(1 -\delta)}$ $(\lambda_j:$ a simple 
turning point$)$ with given implied constants. 
Here $\delta$ is an arbitrary number such that $0<\delta<1$, 
$\lambda_* $ is a fixed base point, and
\begin{align*}
 &\Lambda(\lambda)=t \mu(t,\lambda)\sigma_3 -\mathrm{diag} T^{-1}T_{\lambda},
\\
 &\mathrm{diag}T^{-1}T_{\lambda} 
=\frac 14 \Bigl(1-\frac{b_3}{\mu}\Bigr) \frac{\partial}{\partial\lambda}
\ln \frac{b_1+ib_2}{b_1-ib_2} \sigma_3 +\frac 12\frac{\partial}{\partial\lambda}
\ln\frac{\mu}{\mu+b_3} I.
\end{align*}
\end{prop}
%%%%%%%%%%%%%%%%%%%%%%%%%%%%
%%%%%%%% Remark 4.1 %%%%%%%%
%%%%%%%%%%%%%%%%%%%%%%%%%%%%%
\begin{rem}\label{rem4.1}
In the WKB solution we write $\Lambda(\lambda)$ in the component-wise form
$\Lambda(\lambda)=\Lambda_3(\lambda)+\Lambda_I(\lambda)$, $\Lambda_3(\lambda)
 \in \mathbb{C}\sigma_3,$ $\Lambda_I(\lambda) \in \mathbb{C}I$ with
$$
\Lambda_3(\lambda)=t\mu(t,\lambda)\sigma_3-\mathrm{diag}T^{-1}T_{\lambda}|
_{\sigma_3}\sigma_3, 
\quad \Lambda_I(\lambda)= -\mathrm{diag}T^{-1}T_{\lambda}|_I I.
$$
\end{rem}
%%%%%%%%%%%%%%%%%%%%%%%
The WKB solution fails in expressing asymptotics in a neighbourhood of a
turning point. Around a simple turning point equation \eqref{3.1} is reduced
to the system
%%%%%% (4.3) %%%%%%%%%%
\begin{equation}\label{4.3}
\frac{dW}{d\zeta}=\begin{pmatrix} 0 & 1 \\ \zeta & 0  \end{pmatrix} W
\end{equation}
having solutions ${}^{T}(\mathrm{Ai}(\zeta),\mathrm{Ai}_{\zeta}(\zeta)),$
${}^{T}(\mathrm{Bi}(\zeta),\mathrm{Bi}_{\zeta}(\zeta))$ \cite[Theorem 7.3]
{FIKN}, \cite[Proposition 3.9]{S}, where $\mathrm{Ai}
(\zeta)$ is the Airy function and $\mathrm{Bi}(\zeta)=e^{-\pi i/6}
\mathrm{Ai}(e^{-2\pi i/3}\zeta)$ \cite{AS}, \cite{HTF}. 
%%%%%%%%%%%%%%%%%%%%%%%%%%%%%%%%%%%%%%%%%%%%%%%%%%%%%%%%%%%
%%%%%% Proposition 4.3 %%%%%%%%%%%%%%%%%%%%%
\begin{prop}\label{prop4.3}
Let $\lambda_j$ be a simple turning point, and write $c_k=b_k(\lambda_j),$
$c'_k=(b_k)_{\lambda}(\lambda_j)$ $(k=1,2,3)$ and $\kappa_c=c_1c_1'+c_2c_2'
+c_3c_3'.$ Suppose that $c_k,$ $c_k'$ are bounded and $ c_1\pm ic_2\not=0.$
Then \eqref{3.1} admits a matrix solution of the form
$$
\Phi_j(\lambda)=T_j(I+O(t^{-\delta'})) \begin{pmatrix} 1 & 0 \\ 0 & \hat{t}^{-1}
\end{pmatrix} W(\zeta), 
\quad T_j=\begin{pmatrix} 1  & -\frac{c_3}{c_1+ic_2} \\
        -\frac{c_3}{c_1-ic_2} & 1   \end{pmatrix},
$$
in which $\hat{t}=2(2\kappa_c)^{-1/3}(c_1-i c_2)(t/4)^{1/3}$ and
$\lambda-\lambda_j=(2\kappa_c)^{-1/3} (t/4)^{-2/3}(\zeta+\zeta_0)$ with
$|\zeta_0| \ll t^{-1/3},$ as long as $|\zeta| \ll t^{(2/3-\delta')/3},$ that is,
$|\lambda-\lambda_j| \ll t^{-2/3 +(2/3-\delta')/3}.$ Here $\delta'$ is a given
number such that $0<\delta'<2/3,$ and $W(\zeta)$ solves \eqref{4.3}, which
admits canonical solutions $W_{\nu}(\zeta)$ $(\nu\in \mathbb{Z})$ such that
$$
W_{\nu}(\zeta)=\zeta^{-(1/4)\sigma_3}(\sigma_3+\sigma_1)(I+O(\zeta^{-3/2}))
\exp(\tfrac 23 \zeta^{3/2} \sigma_3)
$$
as $\zeta \to \infty$ through the sector $|\arg \zeta-(2\nu-1)\frac{\pi}3|
<\frac {2\pi}3,$ and that $W_{\nu+1}(\zeta)=W_{\nu}(\zeta)G_{\nu}$ with
$$
G_1=\begin{pmatrix} 1 & -i \\  0 & 1 \end{pmatrix}, \quad
G_2=\begin{pmatrix} 1 & 0 \\  -i & 1 \end{pmatrix}, \quad
G_{\nu+1}=\sigma_1 G_{\nu} \sigma_1.
$$
\end{prop}
%%%%%%%%%%%%%%%%%%%%%%%%%%%%%%%%%%%%%%%%
%%%%%%%%% Remark 4.2 %%%%%%%%%%
\begin{rem}\label{rem4.2}
The matrix solutions
$\Psi_{\mathrm{WKB}}(\lambda)$ and $\Phi_j(\lambda)$ in Propositions 
\ref{prop4.2} and \ref{prop4.3} are simultaneously valid in the 
annulus $\mathcal{A}_{j}:$ $t^{-2/3 +\varepsilon} \ll |\lambda
-\lambda_j| \ll t^{-2/3+2\varepsilon}$ with, say, $\varepsilon
=\frac 23 \delta=\frac 16(\frac 23 -\delta')=\frac 1{10}$, in which the
matching is possible.
\end{rem}
%%%%%%%%%%%%%%%%%%%%%%%%%%%%%%%%%%%%%%%%%%%%
%%%%%%%%%%%%%%%%%%%%% Section 5 %%%%%%%%%
\section{Direct monodromy problem}\label{sc5}
By WKB analysis we calculate Stokes matrices of linear system \eqref{3.1}
as a solution of the direct monodromy problem.
Recall $a_{\phi}(t)$ given by \eqref{4.2} and a unique solution $A_{\phi}$
of the Boutroux equations \eqref{2.2} for $0<|\phi|<\pi/4$. In our 
calculation suppose that $\psi$ and $\psi_t$ are arbitrary functions and that
%%%%%%%% (5.1) %%%%%%% 
\begin{equation}\label{5.1}
a_{\phi}(t)=A_{\phi} +\frac{B_{\phi}(t)}t, \quad B_{\phi}(t) \ll 1
\end{equation}
as $t \to \infty$ in the strip 
$$
 S_{\phi}(t'_{\infty},\kappa_1,\delta_1)=\{ t\,|\, \re t>t'_{\infty},\,\,\,
|\im t|<\kappa_1,\,\,\, |\psi(t)|+|\psi(t)|^{-1}+|\psi_t(t)|<\delta_1^{-1} \},
$$
where $\kappa_1$ is a given number, $\delta_1$ a given small number and
$t_{\infty}'$ a sufficiently large number. 
%%%%%%%%%%%%%%%%%%%%%%%%%%%%%%
%%%%%%%%%%%%%%%%%%%%%%%%%%%%%%
%%%%%%%%%% Figure 5.1 %%%%%%%%%%%%
%%%%%%%%%%%%%%%%%%%%%%%%%%%%%%
%%%%%%%%%%%%%%%%%%%%%%%%%%%%%%%
{\small
\begin{figure}[htb]
\begin{center}
\unitlength=0.90mm
%%%%%%%%%%%%%%%%%%%%%%%%%%%%%%%%%%
\begin{picture}(80,60)(-40,-15)
 \put(20,11){\makebox{$\mathbf{c}^3_{\infty_1}$}}
 \put(10,23){\makebox{$\mathbf{c}^5_{3}$}}
 \put(-11,27){\makebox{$\mathbf{c}_5^{\infty_3}$}}

% \put(32,21){\makebox{${}_{e^{\pi i/8}\infty}$}}

 \put(1,35){\makebox{$\lambda_5$}}
 \put(4,14){\makebox{$\lambda_{3}$}}
 \put(-8,12){\makebox{$\lambda_{1}$}}

 \put(3,0){\makebox{$0$}}
 \put(0,0){\circle{1.5}}

 \put(5,29){\circle*{1.5}}

 \put(5,19.5){\circle*{1.5}}

 \put(-3,10){\circle*{1.5}}
 \put(7,22){\vector(-1,0){0}}
 \put(8.6,29.5){\vector(-1,3){0}}
 \put(2.2,31.1){\vector(-1,-2){0}}
\thinlines
 \qbezier (7.1,31.7) (4.5,34) (2.2, 31.1)
 \qbezier (9,19) (9,22) (7, 22)
 \qbezier (7,26) (9.4,27) (8.6, 29.5)

 \qbezier (5,29) (9,29) (14, 40)
 \qbezier (-3,10) (-1,7) (0,0)
 \qbezier (-3.3,10) (-4.8,8) (-3.9,3)
 \qbezier (-3.2,7.5) (-3.7,7) (-3.1,3)

 \qbezier (-3,10) (-12,7) (-33,19)
 \qbezier (5,19.5) (-2,15) (-3,10)

%% \qbezier (4.1,19.9) (1.9,25) (4.1,28.4)
%% \qbezier (4.4,21.8) (3.4,25) (4.4,26.7)
 \qbezier (4.0,19.9) (1.8,25) (4.0,28.4)
 \qbezier (4.15,21.8) (3.2,25) (4.15,26.7)

\thicklines
 \qbezier (5,19.5) (16,14) (33,19)
 \qbezier (5,19.3) (16,13.8) (33,18.8)

 \qbezier (5,19.5) (6.3,25) (5,29)
 \qbezier (5.3,19.5) (6.6,25) (5.3,29)

 \qbezier (5,29) (-5,30) (-13, 40)
 \qbezier (5,29.2) (-5,30.2) (-13, 40.2)

 \put(-10,-13){\makebox{$\mathfrak{c}^{5\pi/8}_{\pi/8}=\mathbf{c}_{\infty_1}^3
\cup \mathbf{c}_3^5 \cup \mathbf{c}_5^{\infty_3}$}}
\end{picture}
%%%%%%%%%%%%%%%%%%%%%%%%%%%%%%%%%%
%%%%%%%%%%%%%%%%%%%%%%%%%%%%%%%
\qquad
\begin{picture}(80,60)(-40,-15)
 \put(20,11){\makebox{$\mathbf{c}_{\infty_1}^3$}}
 \put(-4,18.5){\makebox{$\mathbf{c}_{3}^1$}}
 \put(-25,6){\makebox{$\mathbf{c}^{\infty_4}_1$}}

 \put(1,32){\makebox{$\lambda_5$}}
 \put(8,20){\makebox{$\lambda_{3}$}}
 \put(-0.5,8){\makebox{$\lambda_{1}$}}

 \put(3,0){\makebox{$0$}}
 \put(0,0){\circle{1.5}}

 \put(5,29){\circle*{1.5}}

 \put(5,19.5){\circle*{1.5}}

 \put(-3,10){\circle*{1.5}}
\thinlines
 \qbezier (5,29) (9,29) (14, 40)
 \qbezier (5,29) (-5,30) (-13, 40)
 \qbezier (5.5,19.5) (6.7,25) (5.5,28)
 \qbezier (-3,10) (-1,7) (0,0)
 \qbezier (-3.3,10) (-4.8,8) (-3.9,3)
 \qbezier (-3.2,7.5) (-3.7,7) (-3.1,3)

 \qbezier (4.0,19.9) (1.8,25) (4.0,28.4)
 \qbezier (4.15,21.8) (3.2,25) (4.15,26.7)

 \qbezier (7.9,17) (5.7,14.5) (2.9,16.5)

 \qbezier (-2.5,13.5) (-5.7,14.5) (-6.9,11.0)
\put(2.9,16.5){\vector(-2,3){0}}
\put(-6.9,11.0){\vector(0,-1){0}}

\thicklines
 \qbezier (5,19.5) (16,14) (33,19)
 \qbezier (5,19.7) (16,14.2) (33,19.2)

 \qbezier (5,19.5) (-2,15) (-3,10)
 \qbezier (5.1,19.5) (-1.9,15.0) (-2.9,10.0)
 \qbezier (5.0,19.6) (-1.9,15.1) (-2.9,10.1)

 \qbezier (-3,10) (-12,7) (-33,19)
 \qbezier (-3,9.8) (-12,6.8) (-33,18.8)

%% \put(40,16.5){\circle{1}}
%% \put(16.5,40){\circle{1}}
%% \put(-40,16.5){\circle{1}}
%% \put(-16.5,40){\circle{1}}

 \put(-10,-13){\makebox{$\mathfrak{c}^{7\pi/8}_{\pi/8}=\mathbf{c}_{\infty_1}^3
\cup \mathbf{c}_3^1 \cup \mathbf{c}_1^{\infty_4}$}}
\end{picture}
%%%%%%%%%%%%%%%%%%%%%%%%%%%%%%%%%%
%%%%%%%%%%%%%%%%%%%%%%%%%%%%%%%%%
\end{center}
\caption{Fragments of the Stokes graph for $-\pi/4<\phi<0$}
\label{fragment1}
\end{figure}
}
%%%%%%%%%%%%%%%%%%%%%%%%%%%%%%%%
%%%%%%%%%%%%%%%%%%%%%%%%%%%%%%
\par
Let $-\pi/4<\phi <0.$ Let us calculate the analytic continuations of the
matrix solution $\Phi_1^{\infty}(\lambda)$ given by \eqref{3.3} in the sector
$|\arg(t\lambda)-\tfrac{\pi}8|<\tfrac{\pi}4$ along the fragments of the 
Stokes graph
$$
\mathfrak{c}_{\pi/8}^{5\pi/8}
=\mathbf{c}_{\infty_1}^3 \cup \mathbf{c}_{3}^{5} \cup \mathbf{c}_{5}^{\infty_3}
\,\,\, \text{and} \,\,\,
\mathfrak{c}_{\pi/8}^{7\pi/8}=\mathbf{c}_{\infty_1}^3
 \cup \mathbf{c}_{3}^{1} \cup \mathbf{c}_{1}^{\infty_4}
$$ 
with
$\mathbf{c}_{\infty_1}^3=(e^{i\pi/8}\infty, \lambda_3)^{\sim},$
$\mathbf{c}_{3}^5=(\lambda_3, \lambda_5)^{\sim},$
$\mathbf{c}_5^{\infty_3}=(\lambda_5, e^{5\pi/8})^{\sim},$
$\mathbf{c}_3^{1}=(\lambda_3, \lambda_1)^{\sim},$
$\mathbf{c}_1^{\infty_4}=(\lambda_1, e^{7\pi/8})^{\sim},$ where $(p,q)^{\sim}$
denotes a curve joining $p$ to $q$, and $\mathbf{c}_3^5$ lies on the
right shore of the cut $[\lambda_3,\lambda_5]$ (cf.~Figure \ref{fragment1}). 
\par
The Stokes matrix $S^*_1S^*_2= 
\Phi_1^{\infty}(\lambda)^{-1}\Phi_3^{\infty}(\lambda)$ follows from 
the analytic continuation of $\Phi_1^{\infty}(\lambda)$ along $\mathfrak{c}
^{5\pi/8}_{\pi/8}$, which is calculated by the matching procedure as below. 
In the steps {\bf (2)}, {\bf (3)}, {\bf (6)}, {\bf (7)}, analytic continuations
are considered in the annulus $\mathcal{A}_3$ and $\mathcal{A}_5$, which may 
be given by $\delta= \tfrac{3}{20},$ $\delta'=\tfrac{1}{15}$ as in
Remark \ref{rem4.2}. Thus, in what follows we set $\delta=\tfrac{3}{20}.$ 
\par
{\bf (1)} For a WKB solution $\Psi^{(\infty)}_3(\lambda)$ along 
$\mathbf{c}_{\infty_1}^{3}$ with a base point
$\tilde{\lambda}_3$ such that $\tilde{\lambda}_3 -\lambda_3\asymp t^{-1}$, 
set $\Phi_1^{\infty}(\lambda)=\Psi^{(\infty)}_3(\lambda)\Gamma_{\infty,3}$.
Then
\begin{align*}
\Gamma_{\infty,3} =& \Psi^{(\infty)}_3(\lambda)^{-1}\Phi_1^{\infty}(\lambda)
\\
=&C_3(\tilde{\lambda_3})c_I(\tilde{\lambda_3})(I+O(t^{-\delta}))
\\
&\times \exp \biggl( -\lim_{\substack{\lambda \to \infty \\[0.05cm]
\lambda \in \mathbf{c}_{\infty_1}^3}} \biggl(\int^{\lambda}_{\lambda_3}
\Lambda_3(\tau)d\tau -(\tfrac 14 t\lambda^4+ e^{i\phi}t\lambda^2 +(\alpha-
\beta)\ln \lambda)\sigma_3 \biggr)\biggr),
\end{align*}
where $C_3(\tilde{\lambda}_3)=\exp(\int^{\tilde{\lambda}_3}_{\lambda_3}
\Lambda_3(\tau)d\tau ),$ $c_I(\tilde{\lambda}_3)=\exp(-\int^{\infty}_{\tilde
{\lambda}_3}\Lambda_I(\tau) d\tau).$
\par
{\bf (2)} For a canonical solution $\Phi_3^+(\lambda)$ along 
$\mathbf{c}^3_{\infty_1} \cap\mathcal{A}_{3}$ 
(cf. Remark \ref{rem4.2}), set $\Psi^{(\infty)}_3(\lambda) =\Phi_3^+(\lambda)
\Gamma^+_3 $. Then
$$
\Gamma^+_3=\Phi_3^+(\lambda)^{-1}\Psi^{(\infty)}_3(\lambda)
=(\tilde{\zeta}_3)^{1/4}(I+O(t^{-\delta}))C_3(\tilde{\lambda}_3)^{-1}
\begin{pmatrix}  1 & 0  \\  0 & -\frac {c_1-ic_2}{c_3} \end{pmatrix},
$$
where $c_k=b_k(\lambda_3)$ and $\tilde{\zeta}_3 \asymp \tilde{\lambda}_3-
\lambda_3$ is suitably chosen.
\par
{\bf (3)} For a canonical solution $\Phi_3^-(\lambda)$ along $\mathbf{c}^5_{3} 
\cap\mathcal{A}_{3}$, set $\Phi_3^+(\lambda)=\Phi_3^-(\lambda)\Gamma
_{(3)}$. Then
$$
\Gamma_{(3)}=\Phi_3^{-}(\lambda)^{-1}\Phi_3^+(\lambda)=G_1^{-1}=
\begin{pmatrix} 1 & i \\ 0 & 1 \end{pmatrix}.
$$
\par
{\bf (4)} For a WKB solution $\Psi_3^{(5)}(\lambda)$ along $\mathbf{c}_3^5$
with a base point $\lambda_3'$ such that $\lambda_3'-\lambda_3 \asymp t^{-1}$,
set $\Phi_3^-(\lambda)=\Psi_3^{(5)}(\lambda) \Gamma_3^-$. Then
$$
\Gamma_3^-=\Psi_3^{(5)}(\lambda)^{-1}\Phi_3^-(\lambda)=(\zeta_3')^{-1/4}
(I+O(t^{-\delta}))C'_3(\lambda_3')
 \begin{pmatrix}  1 & 0 \\  0 & -\frac{c_3}{c_1-i c_2} \end{pmatrix},
$$
where $C_3'(\lambda_3')=\exp(\int^{\lambda_3'}_{\lambda_3}\Lambda_3(\tau)d\tau)$
and $\zeta_3' \asymp \lambda_3'-\lambda_3$ is suitably chosen.
\par
{\bf (5)} For a WKB solution $\Psi_5^{(3)}(\lambda)$ along $\mathbf{c}_3^5$
with a base point $\lambda_5'$ such that $\lambda_5'-\lambda_5 \asymp t^{-1}$,
set $\Psi_3^{(5)}(\lambda)=\Psi_5^{(3)}(\lambda)\Gamma_{3,5}$. Then
\begin{align*}
\Gamma_{3,5}=&\Psi_5^{(3)}(\lambda)^{-1}\Psi_3^{(5)}(\lambda)
\\
=&C_3'(\lambda'_3)^{-1}C_3'(\lambda'_5)c_I(\lambda_3',\lambda_5')
(I+O(t^{-\delta}))\exp\biggl(\int_{\lambda_3}^{\lambda_5}\Lambda_3(\tau)
d\tau\biggr),
\end{align*}
where $C_3(\lambda_5')=\exp(\int^{\lambda_5'}_{\lambda_5}\Lambda_3(\tau)d\tau)$
and $c_I(\lambda'_3,\lambda'_5)=\exp(\int^{\lambda'_5}_{\lambda'_3}
\Lambda_I(\tau)d\tau).$
\par
{\bf (6)} For a canonical solution $\Phi_5^+(\lambda)$ along $\mathbf{c}_3^5
\cap \mathcal{A}_5$, set $\Psi_5^{(3)}(\lambda)=\Phi_5^+(\lambda)\Gamma_5^+.$
Then
$$
\Gamma_5^+=\Phi_5^+(\lambda)^{-1}\Psi_5^{(3)}(\lambda)=(\zeta_5')^{1/4}
(I+O(t^{-\delta}))C_3(\lambda'_5)^{-1}
 \begin{pmatrix}  1 & 0 \\ 0 & -\frac{d_1-id_2}{d_3}  \end{pmatrix},
$$
where $d_k=b_k(\lambda_5)$ and $\zeta'_5 \asymp \lambda'_5-\lambda_5$ 
is suitably chosen.
\par
{\bf (7)} For a canonical solution $\Phi_5^-(\lambda)$ along $\mathbf{c}_5
^{\infty_3}\cap \mathcal{A}_5$, set $\Phi_5^+(\lambda)=\Phi_5^-(\lambda)
\Gamma_{(5)}$. Then
$$
\Gamma_{(5)}=\Phi_5^-(\lambda)^{-1}\Phi_5^+(\lambda)=(G_1G_2)^{-1}
=\begin{pmatrix}  1 & i \\ i & 0 \end{pmatrix}.
$$
\par
{\bf (8)} For a WKB solution $\Psi_5^{(\infty)}(\lambda)$ along $\mathbf{c}_5
^{\infty_3}$ with a base point $\tilde{\lambda}_5$ such that $\tilde{\lambda}_5
-\lambda_5 \asymp t^{-1},$ set $\Phi_5^-(\lambda)=\Psi_5^{(\infty)}(\lambda)
\Gamma_5^-$. Then
$$
\Gamma_5^-=\Psi_5^{(\infty)}(\lambda)^{-1}\Phi_5^-(\lambda)
=(\tilde{\zeta}_5)^{-1/4}(I+O(t^{-\delta})) C_3(\tilde{\lambda}_5)
\begin{pmatrix}  1 & 0 \\ 0 & -\frac{d_3}{d_1-i d_2} \end{pmatrix},
$$
where $C_3(\tilde{\lambda}_5)=\exp(\int^{\tilde{\lambda}_5}_{\lambda_5}
\Lambda_3(\tau)d\tau )$ and $\tilde{\zeta}_5 \asymp \tilde{\lambda}_5-\lambda_5$
is suitably chosen.
\par
{\bf (9)} For $\Phi_3^{\infty}(\lambda)$ solving \eqref{3.1} in the sector
$|\arg\lambda-\frac{5\pi}8|<\frac{\pi}4$, set $\Psi_5^{(\infty)}(\lambda)
=\Phi_3^{\infty}(\lambda)\Gamma_{5,\infty}.$ Then 
\begin{align*} 
\Gamma_{5,\infty}=& \Phi_3^{\infty}(\lambda)^{-1}\Psi_5^{(\infty)}(\lambda)
\\
=& C_3(\tilde{\lambda}_5)^{-1}c_I(\tilde{\lambda}_5)^{-1}(I+O(t^{-\delta}))
\\
& \times \exp\biggl(\lim_{\substack{\lambda\to\infty \\[0.05cm] \lambda \in
\mathbf{c}_5^{\infty_3}}} \biggl(\int_{\lambda_5}^{\lambda} \Lambda_3(\tau)
d\tau -(\tfrac 14t\lambda^4+e^{i\phi}t\lambda^2+(\alpha-\beta)\ln\lambda)\sigma
_3\biggr)\biggr),
\end{align*}
where $c_I(\tilde{\lambda}_5)=\exp(-\int^{\infty}_{\tilde{\lambda}_5}\Lambda_I
(\tau)d\tau).$
\par
Product of the matrices above along $\mathfrak{c}^{5\pi/8}_{\pi/8}$ 
yields the Stokes matrix
\begin{align*}
(S_1^*S_2^*)^{-1} =& \begin{pmatrix}  1 & -s_1^* \\ -s_2^* & 1+s_1^*s_2^*
\end{pmatrix}
\\
 =& \Phi_3^{\infty}(\lambda)^{-1}\Phi_1^{\infty}(\lambda)
\\
=&\Gamma_{5,\infty} \Gamma_5^- \Gamma_{(5)} \Gamma_5^+ \Gamma_{3,5}
\Gamma_3^- \Gamma_{(3)}\Gamma_3^+ \Gamma_{\infty,3}
\\
=&\epsilon_1 e^{J_{5}^{\infty 3}\sigma_3}
\begin{pmatrix} 1 & 0 \\ 0 & -d_0^{-1} \end{pmatrix}
\begin{pmatrix} 1 & i \\ i & 0 \end{pmatrix} 
\begin{pmatrix} 1 & 0 \\ 0 & -d_0 \end{pmatrix}
\\
& \times e^{J_{3,5}\sigma_3}
\begin{pmatrix} 1 & 0 \\ 0 & -c_0^{-1} \end{pmatrix}
\begin{pmatrix} 1 & i \\ 0 & 1 \end{pmatrix} 
\begin{pmatrix} 1 & 0 \\ 0 & -c_0 \end{pmatrix} e^{-J_3^{\infty 1}\sigma_3}
\\
=& \epsilon_1 \begin{pmatrix}
e^{J_{3,5}+J_5^{\infty 3} -J_3^{\infty 1}} & -i(c_0e^{J_{3,5}}+d_0e^{-J_{3,5}})
e^{J_5^{\infty 3}+J_3^{\infty 1}}  \\
-i d_0^{-1} e^{J_{3,5}-J_5^{\infty 3}-J_3^{\infty 1}}  & 
- c_0d_0^{-1} e^{J_{3,5}-J_5^{\infty 3}+J_3^{\infty 1}}   
\end{pmatrix}
\end{align*}
(up to the multiplier $1+O(t^{-\delta})$ to each entry),
in which $\epsilon_1^2=1,$ 
%%%%%%%%%%%% (5.2) %%%%%%%%%%%%%
\begin{align}\notag
& c_0=(c_1-ic_2)/c_3, \quad d_0=(d_1-id_2)/d_3, \quad c_k=b_k(\lambda_3),\,\,\,
d_k=b_k(\lambda_5), 
\\
\label{5.2}
& J_{3,5} \sigma_3 =\int^{\lambda_5}_{\lambda_3} \Lambda_3(\tau)d\tau, \quad
\text{$\mathbf{c}_3^5=(\lambda_3,\lambda_5)^{\sim}$: on the right shore of
the cut $[\lambda_3,\lambda_5]$},
\\
\notag
& J_5^{\infty 3}\sigma_3
=\lim_{\substack{\lambda\to\infty \\[0.05cm] \lambda \in
\mathbf{c}_5^{\infty_3}}} \biggl(\int_{\lambda_5}^{\lambda} \Lambda_3(\tau)
d\tau -(\tfrac 14t\lambda^4+e^{i\phi}t\lambda^2+(\alpha-\beta)\ln\lambda)\sigma
_3\biggr),
\\
\notag
& J_3^{\infty 1}\sigma_3
=\lim_{\substack{\lambda\to\infty \\[0.05cm] \lambda \in
\mathbf{c}^3_{\infty_1}}} \biggl(\int_{\lambda_3}^{\lambda} \Lambda_3(\tau)
d\tau -(\tfrac 14t\lambda^4+e^{i\phi}t\lambda^2+(\alpha-\beta)\ln\lambda)\sigma
_3\biggr).
\end{align}            
Then the diagonal entries of $(S_1^*S_2^*)^{-1}$ give 
$1=\epsilon_1
e^{J_{3,5}+J_5^{\infty 3}-J_3^{\infty 1}}(1+O(t^{-\delta})),$ 
$1+s_1^*s_2^*=-\epsilon_1 c_0 d_0^{-1}
e^{J_{3,5}-J_5^{\infty 3}+J_3^{\infty 1}}(1+O(t^{-\delta})),$ 
which imply
%%%%%%%%%%%%%% (5.3) %%%%%%%%%%%%%%%
\begin{equation}\label{5.3}
 1+s_1^* s_2^*= -c_0d_0^{-1} e^{2J_{3,5}}(1+O(t^{-\delta})). 
\end{equation} 
Similarly the analytic continuation of $\Phi_1^{\infty}(\lambda)$ along
$\mathfrak{c}_{\pi/8}^{7\pi/8}$ yields
\begin{align*}
(S^*_1S^*_2S^*_3)^{-1}=&\begin{pmatrix}
 1+s_2^*s_3^*  & -s_1^*-s_3^*-s_1^*s_2^*s_3^* \\ -s_2^*  & 1+s_1^*s_2^*
\end{pmatrix}  
\\
=& \Phi_4^{\infty}(\lambda)^{-1} \Phi_1^{\infty}(\lambda)
\\
=&\epsilon_2 e^{J_1^{\infty 4}\sigma_3} 
\begin{pmatrix} 1 & 0 \\ 0 & -e_0^{-1} \end{pmatrix}
\begin{pmatrix} 1 & i \\ 0 & 1 \end{pmatrix}
\begin{pmatrix} 1 & 0 \\ 0 & -e_0 \end{pmatrix}
\\
&\times e^{J_{3,1} \sigma_3}
\begin{pmatrix} 1 & 0 \\ 0 & -c_0^{-1} \end{pmatrix}
\begin{pmatrix} 1 & 0 \\ -i & 1 \end{pmatrix}
\begin{pmatrix} 1 & 0 \\ 0 & -c_0 \end{pmatrix} e^{-J_3^{\infty 1}\sigma_3}
\\
=& \epsilon_2 \begin{pmatrix} (e^{J_{3,1}}+c_0^{-1}e_0 e^{-J_{3,1}}) 
e^{J_1^{\infty 4}-J_3^{\infty 1}} &
 -i e_0 e^{-J_{3,1}+J_1^{\infty 4}+J_3^{\infty 1}} \\ 
i c_0^{-1} e^{-J_{3,1}-J_1^{\infty 4}-J_3^{\infty 1}} & 
 e^{-J_{3,1}-J_1^{\infty 4}+J_3^{\infty 1}}  \end{pmatrix}, 
\end{align*}
in which $\epsilon_2^2=1,$
\begin{align}\notag
 &e_0=(e_1-ie_2)/e_3, \quad e_k=b_k(\lambda_1),  
\\
%%%%%%%%%%%%%%% (5.4) %%%%%%%%%%%%%
\label{5.4}
 &J_{3,1} \sigma_3= \int^{\lambda_1}_{\lambda_3} \Lambda_3(\tau)d\tau,
\\
\notag
& J_1^{\infty 4}\sigma_3
=\lim_{\substack{\lambda\to\infty \\[0.05cm] \lambda \in
\mathbf{c}_1^{\infty_4}}} \biggl(\int_{\lambda_1}^{\lambda} \Lambda_3(\tau)
d\tau -(\tfrac 14t\lambda^4+e^{i\phi}t\lambda^2+(\alpha-\beta)\ln\lambda)\sigma
_3\biggr).
\end{align}
Observing the off-diagonal entries, we have 
$(1+s_1^*s_2^*)(1+s_2^*s_3^*)-1=c_0^{-1}e_0 e^{-2J_{3,1}}(1+O(t^{-\delta}))$.
Combining this with \eqref{5.3} and using Corollary \ref{cor3.5}, we have the
following.
%%%%%%%%%%%%%%%%%%%%%%%%%%%%%%%%%%%%%%%%
%%%%%%%%%% Proposition 5.1 %%%%%%%%%%%%%
\begin{prop}\label{prop5.1}
Suppose that $-\pi/4<\phi <0.$ Then
\begin{align*}
&1+s_1s_2= -c_0d_0^{-1}e^{2J_{3,5}}(1+O(t^{-\delta})), 
\\
&(1+s_1s_2)(1+s_2s_3)-1=c_0^{-1}e_0 e^{-2J_{3,1}}(1+O(t^{-\delta})),
\end{align*}
where $J_{3,5}$ and $J_{3,1}$ are integrals given by \eqref{5.2} and 
\eqref{5.4}.
\end{prop}
%%%%%%%%%%%%%%%%%%%%%%%%%%%%
\par
In the case where $0<\phi <\pi/4$, we calculate the analytic continuations of
$$
\Phi_2^{\infty}(\lambda) \,\,\,\text{along}\,\, 
\hat{\mathfrak{c}}_{3\pi/8}^{7\pi/8}=\mathbf{c}_{\infty_2}^{5}\cup 
{\mathbf{c}}_5^{3-} \cup \mathbf{c}_3^{\infty_4}, \quad
\Phi_1^{\infty}(\lambda) \,\,\,\text{along}\,\, 
\hat{\mathfrak{c}}_{\pi/8}^{7\pi/8}=\mathbf{c}_{\infty_1}^{1}\cup \mathbf{c}
_1^3 \cup \mathbf{c}_3^{\infty_4}
$$
with $\mathbf{c}_{\infty_1}^1=(e^{\pi i/8}\infty,\lambda_1)^{\sim}$,
$\mathbf{c}_{1}^3=(\lambda_1,\lambda_3)^{\sim}$,
$\mathbf{c}_{3}^{\infty_4}=(\lambda_3,e^{7\pi i/8}\infty)^{\sim}$,
$\mathbf{c}^{5}_{\infty_2}=(e^{3\pi i/8}\infty,\lambda_5)^{\sim}$,
${\mathbf{c}}_{5}^{3-}=(\lambda_5,\lambda_3)^{\sim}$,
where ${\mathbf{c}}_5^{3-}$ lies on the left shore of the cut
$[\lambda_3,\lambda_5]$ (cf.~Figure \ref{fragment2}). 
%%%%%%%%%%%%%%%%%%%%%%%%%%%%%%%%%%%%%%
%%%%%%%% Figure 5.2 %%%%%%%%%%%%%%
%%%%%%%%%%%%%%%%%%%%%%%%%%%%%%%%%
{\small
\begin{figure}[htb]
\begin{center}
\unitlength=0.90mm
%%%%%%%%%%%%%%%%%%%%%%%%%%%%%%%%%%
\begin{picture}(80,60)(-40,-15)
 \put(-25,11){\makebox{$\mathbf{c}_3^{\infty_4}$}}
 \put(-14.5,22.5){\makebox{$\mathbf{c}_5^{3-}$}}
 \put(8,30){\makebox{$\mathbf{c}^5_{\infty_2}$}}

 \put(-6,35){\makebox{$\lambda_5$}}
 \put(-6,14){\makebox{$\lambda_{3}$}}
 \put(4,11){\makebox{$\lambda_{1}$}}

 \put(-4,0){\makebox{$0$}}
 \put(0,0){\circle{1.5}}

 \put(-5,29){\circle*{1.5}}

 \put(-5,19.5){\circle*{1.5}}

 \put(3,10){\circle*{1.5}}
 \put(-9,19){\vector(1,-3){0}}
 \put(-7,26){\vector(4,-1){0}}
 \put(-7.1,31.7){\vector(-2,-3){0}}
\thinlines
 \qbezier (-7.1,31.7) (-4.5,34) (-2.2, 31.1)
 \qbezier (-9,19) (-9,22) (-7, 22)
 \qbezier (-7,26) (-9.4,27) (-8.6, 29.5)

 \qbezier (-5,29) (-9,29) (-14, 40)
 \qbezier (3,10) (1,7) (0,0)
 \qbezier (3.3,10) (4.8,8) (3.9,3)
 \qbezier (3.2,7.5) (3.7,7) (3.1,3)

 \qbezier (3,10) (12,7) (33,19)
 \qbezier (-5,19.5) (2,15) (3,10)

%% \qbezier (4.1,19.9) (1.9,25) (4.1,28.4)
%% \qbezier (4.4,21.8) (3.4,25) (4.4,26.7)
 \qbezier (-4.0,19.9) (-1.8,25) (-4.0,28.4)
 \qbezier (-4.15,21.8) (-3.2,25) (-4.15,26.7)

\thicklines
 \qbezier (-5,19.5) (-16,14) (-33,19)
 \qbezier (-5,19.3) (-16,13.8) (-33,18.8)

 \qbezier (-5,19.5) (-6.3,25) (-5,29)
 \qbezier (-5.3,19.5) (-6.6,25) (-5.3,29)

 \qbezier (-5,29) (5,30) (13, 40)
 \qbezier (-5,29.2) (5,30.2) (13, 40.2)

\put(-10,-13){\makebox{$\hat{\mathfrak{c}}^{7\pi/8}_{3\pi/8}
=\mathbf{c}_{\infty_2}^5\cup \mathbf{c}_5^{3-}\cup \mathbf{c}_3^{\infty_4}$}}
\end{picture}
%%%%%%%%%%%%%%%%%%%%%%%%%%%%%%%%%%
%%%%%%%%%%%%%%%%%%%%%%%%%%%%%%%
\qquad
\begin{picture}(80,60)(-40,-15)
 \put(-22,11){\makebox{$\mathbf{c}^{\infty_4}_3$}}
 \put(1,17.5){\makebox{$\mathbf{c}^{3}_1$}}
 \put(25,10){\makebox{$\mathbf{c}_{\infty_1}^1$}}

 \put(-6,32){\makebox{$\lambda_5$}}
 \put(-11.5,20.5){\makebox{$\lambda_{3}$}}
 \put(-2.5,9){\makebox{$\lambda_{1}$}}

 \put(-4,0){\makebox{$0$}}
 \put(0,0){\circle{1.5}}

 \put(-5,29){\circle*{1.5}}

 \put(-5,19.5){\circle*{1.5}}

 \put(3,10){\circle*{1.5}}
\thinlines
 \qbezier (-5,29) (-9,29) (-14, 40)
 \qbezier (-5,29) (5,30) (13, 40)
 \qbezier (-5.5,19.5) (-6.7,25) (-5.5,28)
 \qbezier (3,10) (1,7) (0,0)
 \qbezier (3.3,10) (4.8,8) (3.9,3)
 \qbezier (3.2,7.5) (3.7,7) (3.1,3)

 \qbezier (-4.0,19.9) (-1.8,25) (-4.0,28.4)
 \qbezier (-4.15,21.8) (-3.2,25) (-4.15,26.7)

 \qbezier (-7.9,17) (-5.7,14.5) (-2.9,16.5)

 \qbezier (2.5,13.5) (5.7,14.5) (6.9,11.0)
\put(-7.9,17){\vector(-1,2){0}}
\put(2.5,13.5){\vector(-2,-1){0}}

\thicklines
 \qbezier (-5,19.5) (-16,14) (-33,19)
 \qbezier (-5,19.7) (-16,14.2) (-33,19.2)

 \qbezier (-5,19.5) (2,15) (3,10)
 \qbezier (-5.1,19.5) (1.9,15.0) (2.9,10.0)
 \qbezier (-5.0,19.6) (1.9,15.1) (2.9,10.1)

 \qbezier (3,10) (12,7) (33,19)
 \qbezier (3,9.8) (12,6.8) (33,18.8)

 \put(-10,-13){\makebox{$\hat{\mathfrak{c}}^{7\pi/8}_{\pi/8}
=\mathbf{c}_{\infty_1}^1 \cup \mathbf{c}_1^3 \cup \mathbf{c}_3^{\infty_4}$}}
\end{picture}
%%%%%%%%%%%%%%%%%%%%%%%%%%%%%%%%%%
%%%%%%%%%%%%%%%%%%%%%%%%%%%%%%%%%
\end{center}
\caption{Fragments of the Stokes graph for $0<\phi<\pi/4$}
\label{fragment2}
\end{figure}
}
%%%%%%%%%%%%%%%%%%%%%%%%%%%%%%
%%%%%%%%%%%%%%%%%%%%%%%%%%%%%%%%%%

This procedure results in
\begin{align*}
(S_2^*S_3^*)^{-1}=&\begin{pmatrix} 1+s_2^*s_3^* & -s_3^* \\ -s_2^* & 1
\end{pmatrix}
\\
=& \Phi_4^{\infty}(\lambda)^{-1}\Phi_2^{\infty}(\lambda)
\\
=&\tilde{\epsilon}_1 e^{J_3^{\infty 4}\sigma_3} 
\begin{pmatrix} 1 & 0 \\ 0 & -c_0^{-1} \end{pmatrix}
\begin{pmatrix} 1 & i \\ 0 & {1} \end{pmatrix}
\begin{pmatrix} 1 & 0 \\ 0 & -c_0 \end{pmatrix}
\\
&\times e^{J_{5,3}^- \sigma_3}
\begin{pmatrix} 1 & 0 \\ 0 & -d_0^{-1} \end{pmatrix}
\begin{pmatrix} 0 & i \\ i & {1} \end{pmatrix}
\begin{pmatrix} 1 & 0 \\ 0 & -d_0 \end{pmatrix}
e^{-J_5^{\infty 2} \sigma_3}
\\
=&\tilde{\epsilon}_1 
\begin{pmatrix} -c_0d_0^{-1} e^{-J^-_{5,3} +J_3^{\infty 4}-J_5^{\infty 2}} 
& -i(d_0 e^{J^-_{5,3}}+c_0e^{-J^-_{5,3}}) e^{J_3^{\infty 4}+J_5^{\infty 2}}  \\
-id_0^{-1} e^{-J^-_{5,3} -J_3^{\infty 4}-J_5^{\infty 2}}
&  e^{-J^-_{5,3} -J_3^{\infty 4}+J_5^{\infty 2}}
\end{pmatrix}
\end{align*}
with $\tilde{\epsilon}_1^2=1,$ $J_3^{\infty 4}\sigma_3=J_1^{\infty 4}\sigma_3|
_{\mathbf{c}_1^{\infty_4} \mapsto \mathbf{c}_3^{\infty_4}, \,\, 
\lambda_1 \mapsto \lambda_3 }$,
$J_5^{\infty 2}\sigma_3=J_3^{\infty 1}\sigma_3|
_{\mathbf{c}^3_{\infty_1} \mapsto \mathbf{c}^5_{\infty_2}, \,\, 
\lambda_3 \mapsto \lambda_5 }$ and
\begin{equation}
%%%%%%%%%%%%% (5.5) %%%%%%%%%%%%%%%%
\label{5.5}
 J^-_{5,3} \sigma_3 =\int^{\lambda_3,-}_{\lambda_5} \Lambda_3(\tau)d\tau, \quad
\text{$\mathbf{c}_5^{3-}=(\lambda_5,\lambda_3)^{\sim}$: on the left shore of
the cut $[\lambda_3,\lambda_5]$},
\end{equation}
and
\begin{align*}
(S_1^*S_2^*S_3^*)^{-1}
=& \Phi_4^{\infty}(\lambda)^{-1}\Phi_1^{\infty}(\lambda)
\\
=&\tilde{\epsilon}_2 e^{J_3^{\infty 4}\sigma_3} 
\begin{pmatrix} 1 & 0 \\ 0 & -c_0^{-1} \end{pmatrix}
\begin{pmatrix} 1 & 0 \\ -i & {1} \end{pmatrix}
\begin{pmatrix} 1 & 0 \\ 0 & -c_0 \end{pmatrix}
\\
&\times e^{J_{1,3}\sigma_3}
\begin{pmatrix} 1 & 0 \\ 0 & -e_0^{-1} \end{pmatrix}
\begin{pmatrix} 1 & i \\ 0 & {1} \end{pmatrix}
\begin{pmatrix} 1 & 0 \\ 0 & -e_0 \end{pmatrix}
e^{-J_1^{\infty 1} \sigma_3}
\\
=&\tilde{\epsilon}_2 
\begin{pmatrix} 
 e^{J_{1,3} +J_3^{\infty 4}-J_1^{\infty 1}}
& -ie_0 e^{J_{1,3} +J_3^{\infty 4}+J_1^{\infty 1}}
\\
ic_0^{-1} e^{J_{1,3} -J_3^{\infty 4}-J_1^{\infty 1}} 
& (c_0^{-1}e_0e^{J_{1,3}}+e^{-J_{1,3}}) e^{-J_3^{\infty 4}+J_1^{\infty 1}}  
\end{pmatrix}
\end{align*}
with $\tilde{\epsilon}_2^2=1,$
$J_1^{\infty 1}\sigma_3=J_3^{\infty 1}\sigma_3|
_{\mathbf{c}^3_{\infty_1} \mapsto \mathbf{c}^1_{\infty_1}, \,\, 
\lambda_3 \mapsto \lambda_1 }$ and $J_{1,3} \sigma_3=-J_{3,1} \sigma_3.$
%% \begin{equation}
%%%%%%%%%%%%%%%%%%%%%%%%%%%%%
Thus we have the following.
%%%%%%%%%%%%%%%%%%%%%%%%%%%%%%%%%%%%
%%%%%% Proposition 5.2 %%%%%%%%%%%%%
\begin{prop}\label{prop5.2}
Suppose that $0<\phi<\pi/4.$ Then
\begin{align*}
&1+s_2s_3= -c_0d_0^{-1} e^{-2J^-_{5,3}}(1+O(t^{-\delta})), 
\\
&(1+s_1s_2)(1+s_2s_3)-1=c_0^{-1}e_0 e^{2J_{1,3}}(1+O(t^{-\delta})),
\end{align*}
where $J^-_{5,3}$ is given by \eqref{5.5} and $J_{1,3}=-J_{3,1}$.
\end{prop}
%%%%%%%%%%%%%%%%%%%%%%%%%%%
%%%%%%%%%%%%%%%%%%%%%%%%%%%%%%%%
%%%%%% Section 6 %%%%%%%%%%%%%%%%%%
\section{Asymptotics of monodromy data}\label{sc6}
%%%%%%%%%%%%%%%%%%%%%%%%%%%%%%%%%%%%%%
Recall that $\lambda \mu(\infty,\lambda)=w(A_{\phi},\lambda^2)$ is
considered on the Riemann surface $\mathcal{R}_{\infty}$ and that 
$w(A_{\phi},z)$ defines the elliptic curve $\Pi_{A_{\phi},\phi}$ with the
primitive cycles $\mathbf{a}$ and $\mathbf{b}$ described as in Figure 
\ref{cycles1}. Note that, by $z=\lambda^2$, $\mathcal{R}_{\infty}$ is mapped to 
$\Pi_{A_{\phi},\phi}$. Let $\hat{\mathbf{a}},$ $\hat{\mathbf{b}} \subset
\mathcal{R}_{\infty}$ denote the inverse images of $\mathbf{a},$ $\mathbf{b}$,
respectively, such that $\hat{\mathbf{a}}\subset \mathcal{R}^+_{\infty}$ 
surrounds the cut $[\lambda_5,\lambda_3]$ and that $\hat{\mathbf{b}}$ links
with $\hat{\mathbf{a}}$.   
For $\lambda_j(t) \in \mathcal{R}_t$ and $\lambda_j=\lambda_j(\infty)\in
\mathcal{R}_{\infty}$, we have $\lambda_j(t)=\lambda_j(\infty)+O(t^{-1})$ 
if $0<|\phi|<\pi/4,$ and hence
the cycles $\hat{\mathbf{a}}$ and $\hat{\mathbf{b}}$ may be regarded to be
those on $\mathcal{R}_t$ as well for sufficiently large $t$.
Furthermore $\mathbf{a},$ $\mathbf{b} \subset \Pi_{A_{\phi},\phi}$ may be
regarded to be the primitive cycles on $\Pi_{a_{\phi}(t),\phi}$ by \eqref{5.1}.
%%%%%%%%%%%%%%%%%%%%%%%%%%%%%%%%%%%%%%%%%%
%%%%%%%% subSection 6.1 %%%%%
\subsection{Integrals}\label{ssc6.1}
%%%%%%%%%%%%%%%%%%%%%%%%%%%%%%%%%%%%
We would like to calculate the asymptotics of $\int_{\hat{\mathbf{a}},\,\,
\hat{\mathbf{b}}}\Lambda_3(\lambda)\lambda$. By \eqref{4.1}  
\begin{align*}
\mu(t,\lambda) &=\lambda^{-1}\left(\lambda^8+4e^{i\phi}\lambda^6 +4e^{2i\phi}
\lambda^4+4e^{3i\phi}a_{\phi}\lambda^2 
+2(\alpha-\beta)t^{-1}\lambda^4+\alpha^2t^{-2}\right )^{1/2}
\\
&=\lambda^{-1} {\mathbf{w}(\lambda)}\Bigl(1+2(\alpha-\beta)t^{-1}\frac{\lambda^4}
{\mathbf{w}(\lambda)^{2}} +O(t^{-2}\mathbf{w}(\lambda)^{-2}) \Bigr)^{1/2}
\\
&=\frac{\mathbf{w}(\lambda)}{\lambda}+(\alpha-\beta)t^{-1} \frac{\lambda^3}
{\mathbf{w}(\lambda)} +O(t^{-2})
\end{align*}
along $\hat{\mathbf{a}}$ and $\hat{\mathbf{b}}$, where 
$$
\mathbf{w}(\lambda)=w(a_{\phi},\lambda^2), \quad 
 w(a_{\phi},z)=\sqrt{z^4+4e^{i\phi}z^3 +4e^{2i\phi}z^2 +4 e^{3i\phi}a_{\phi}z}.
$$
Substitution $\lambda^2=z$ yields
%%%%%%%%%% (6.1) %%%%%%%%%%%%%%%%%%%%%%%%
\begin{align}\label{6.1}
\int_{\hat{\mathbf{a}},\,\,\hat{\mathbf{b}}}  \mu(t,\lambda)d\lambda
=&\frac 12 \int_{\mathbf{a},\,\, \mathbf{b}}\frac{w(a_{\phi},z)}z dz  
+ \frac 12(\alpha-\beta)t^{-1}\int_{\mathbf{a},\,\, \mathbf{b}}
\frac z{w(a_{\phi},z)} dz +O(t^{-2})  
\\
\notag
=&
e^{4i\phi}a_{\phi}\int_{{\mathbf{a}},\,\,{\mathbf{b}}}\frac{dz}{z w(a_{\phi},z)}
+\frac 12(\alpha-\beta)t^{-1}\int_{\mathbf{a},\,\,\mathbf{b}}\frac z{w(a_{\phi},
z)} dz 
\\
\notag
&+\frac{3e^{3i\phi}a_{\phi}}2 \omega_{\mathbf{a},\,\,\mathbf{b}} +O(t^{-2}),
\quad\,\, \omega_{\mathbf{a},\,\, \mathbf{b}} =\int_{\mathbf{a},\,\,\mathbf{b}} 
\frac{dz}{w(a_{\phi},z)},
\end{align}
in which the second equality is obtained by using
$$
\int w(a_{\phi},z)\frac{dz}z=\frac{w(a_{\phi},z)}2+e^{i\phi}\frac{w(a_{\phi},z)}
z +3 e^{3i\phi}a_{\phi} \int \frac{dz}{w(a_{\phi},z)}
+ 2e^{4i\phi}a_{\phi}\int\frac{dz}{z w(a_{\phi},z)}.
$$
\par
Let us calculate
$$
\mathrm{diag}\,T^{-1}T_{\lambda}|_{\sigma_3}d\lambda 
=\frac 14 \Bigl(1-\frac{b_3}{\mu}\Bigl)\frac{d}{d\lambda} \ln\frac{b_1+ib_2}
{b_1-ib_2} d\lambda.
$$
Recalling \eqref{3.1} and \eqref{3.2}, and setting $\lambda^2=z$, we have
%%%%%%%%%% (6.2) %%%%%%%%%%%%%%%%%%
\begin{align}\label{6.2}
&b_1-ib_2=2ie^{i\phi/2}(z-z_+), \quad b_1+ib_2=2i e^{-i\phi/2}(z-z_-),
\\
\notag
&z_+:=-\frac{e^{-i\phi}}2 \frac{\psi_t}{\psi} -\frac{\psi}2 -e^{i\phi}+O(t^{-1}),
\quad
z_-:=\frac{e^{-i\phi}}2 \frac{\psi_t}{\psi} -\frac{\psi}2 -e^{i\phi}+O(t^{-1}),
\end{align}
with $(b_1-ib_2)(z_+)=0,$ $(b_1+ib_2)(z_-)=0.$
Furthermore
$$
\frac{b_3}{\mu(t,\lambda)}=z(z+ 2(e^{i\phi}+\psi))\Bigl(\frac 1{w(a_{\phi},z)}
+O(t^{-1})\Bigr),
$$
which satisfies $(b_3/\mu)(z_{\pm})=1$ on the upper sheet $\Pi_+$ of
$\Pi_{a_{\phi},\phi}$. Then it follows that
%%%%%%%%%%%%%%% (6.3) %%%%%%%
\begin{align}\label{6.3}
\mathrm{diag}\,T^{-1}T_{\lambda}|_{\sigma_3}d\lambda 
 &=\frac 14\Bigl(\frac 1{z-z_-}-\frac 1{z-z_+}\Bigr)dz
\\
\notag
&+\frac 14\Bigl(z_+ -z_- + \frac{w(a_{\phi},z_+)}{z-z_+} 
- \frac{w(a_{\phi},z_-)}{z-z_-} \Bigr)\frac{dz}{w(a_{\phi},z)} +O(t^{-1})dz,
\end{align} 
and that
%%%%%%%%%%%%% (6.4) %%%%%%%%%%%%%%%%%%%
\begin{equation}\label{6.4}
\begin{split}
&\frac 12 \int^{z_3}_{z_5}\Bigl(\frac 1{z-z_+}-\frac 1{z-z_-} \Bigr)dz
=\frac 12\ln\frac{(\lambda_3^2-z_+)(\lambda_5^2-z_-)} 
{(\lambda_3^2-z_-)(\lambda_5^2-z_+)}=\ln(c_0d_0^{-1}),
\\ 
&\frac 12 \int^{z_1}_{z_3}\Bigl(\frac 1{z-z_+}-\frac 1{z-z_-} \Bigr)dz
=\frac 12\ln\frac{(\lambda_1^2-z_+)(\lambda_3^2-z_-)} 
{(\lambda_1^2-z_-)(\lambda_3^2-z_+)}=\ln(c_0^{-1}e_0).
\end{split}
\end{equation}
%%%%%%%%%%%%%%%%%%%%%%%%%%%%%%
\par
Suppose that $-\pi/4<\phi<0$, and recall Proposition \ref{prop5.1}.
Note that
$$
J_{3,5}=\int_{\lambda_3}^{\lambda_5} \Bigl(t\mu(t,\lambda)-\mathrm{diag}T^{-1}
T_{\lambda}|_{\sigma_3}\Bigr)d\lambda, 
$$
along $\mathbf{c}_3^5=(\lambda_3,\lambda_5)^{\sim}$ in \eqref{5.2}, and that
$\mathbf{c}_3^5$ 
is $ \frac 12 (- \hat{\mathbf{a}})$, which is the image of 
$\frac 12 (-\mathbf{a})$ of Figure \ref{cycles1} under the map $\lambda^2=z.$
The integral $J_{3,1}$ of \eqref{5.4} is along $\mathbf{c}_3^1$, which is
the image of $\tfrac 12 \mathbf{b}$.
Then by \eqref{6.1}, \eqref{6.3} and \eqref{6.4}, we have the following
proposition, in which
%%%%%%%%%%%%%%% (6.5) %%%%%%%%% (6.6) %%%%%%%%%%
\begin{align}\label{6.5}
&W(z)=\Bigl(z_+ -z_- + \frac{w(a_{\phi},z_+)}{z-z_+} 
- \frac{w(a_{\phi},z_-)}{z-z_-} \Bigr)\frac{1}{w(a_{\phi},z)}, 
\\
\label{6.6}
&\omega_{\mathbf{a},\,\,\mathbf{b}}=
 \int_{\mathbf{a},\,\,\mathbf{b}}\frac{dz}{w(a_{\phi},z)},
\qquad
C_{\alpha,\beta}(\mathbf{a},\,\, \mathbf{b})
= \frac 12 (\alpha-\beta)\int_{\mathbf{a},\,\, \mathbf{b}}
\frac{z dz}{w(a_{\phi},z)}.
\end{align}
%%%%%%%%%%%%%% Proposition 6.1 %%%%%%%%%%%%%%%%%%%%%
\begin{prop}\label{prop6.1}
Suppose that $-\pi/4<\phi<0$. Then
\begin{align*}
\ln(1+s_1s_2)=&-\frac t2\int_{\mathbf{a}}\frac{w(a_{\phi},z)}z dz
+\frac 14 \int_{\mathbf{a}}W(z)dz-C_{\alpha,\beta}({\mathbf{a}})
 +\pi i +O(t^{-\delta}),
\\
\ln((1+s_1s_2)(1&  +s_2s_3)-1)
\\
=&-\frac t2\int_{\mathbf{b}}\frac{w(a_{\phi},z)}z dz
+\frac 14\int_{\mathbf{b}}W(z)dz -C_{\alpha,\beta}({\mathbf{b}})+O(t^{-\delta}).
\end{align*}
\end{prop}
From Proposition \ref{prop5.2} we have the following.
%%%%%%%%%%%% Proposition 6.2 %%%%%%%%
\begin{prop}\label{prop6.2}
Suppose that $0<\phi<\pi/4.$ Then
\begin{align*}
\ln(1+s_2s_3)=& \frac t2\int_{\mathbf{a}}\frac{w(a_{\phi},z)}z dz
-\frac 14 \int_{\mathbf{a}}W(z)dz +C_{ \alpha,\beta}(\mathbf{a})
+\pi i +O(t^{-\delta}),
\\
\ln((1+s_1s_2)(&1  +s_2s_3)-1)
\\
=&-\frac t2\int_{\mathbf{b}}\frac{w(a_{\phi},z)}z dz
+\frac 14 \int_{\mathbf{b}}W(z)dz -C_{\alpha,\beta}(\mathbf{b})+O(t^{-\delta}).
\end{align*}
\end{prop}
%%%%%%%%%%%%%%%%%%%%%%%%%%%%%%%%%%%%%%%%%%%%%%
%%%%%%% Remark 6.1 %%%%%%%%%%%%%%%%%%%%%
\begin{rem}\label{rem6.1}
In the propositions above,
$$
\frac12 \int_{\mathbf{a},\,\,\mathbf{b}}\frac{w(a_{\phi},z)}{z}dz
=e^{4i\phi}a_{\phi}\int_{\mathbf{a},\,\,\mathbf{b}}\frac{dz}{zw(a_{\phi},z)}
+\frac 32e^{3i\phi}a_{\phi}\omega_{\mathbf{a},\,\,\mathbf{b}}.
$$
\end{rem}
%%%%%%%%%%%%%%%%%%%%%%%%%%%%%%%%%
%%%%%%%%%%% 6.2 %%%%%%%%%%%%%%%%
\subsection{Theta-function}\label{ssc6.2}
%%%%%%%%%%%%%%%%%%%%%%%%%%%%%%%%%%%%%
Further calculation needs the theta-function
$$
\vartheta(z,\tau)=\sum_{n=-\infty}^{\infty} e^{\pi i\tau n^2+2\pi izn},
\quad \tau=\frac{\omega_{\mathbf{b}}}{\omega_{\mathbf{a}}}, \quad
\im \tau>0, \quad \nu=\frac{1+\tau}2.
$$
For $z,$ $\tilde{z} \in \Pi_{a_{\phi},\phi}=\Pi_+ \cup \Pi_-$, set
$$
F(\tilde{z},z)=\frac{1}{\omega_{\mathbf{a}}} \int^z_{\tilde{z}} \frac{dz}
{w(a_{\phi},z)}.
$$
For $z_0 \in \Pi_{a_{\phi},\phi}=\Pi_+ \cup \Pi_-$ write $z_0^+=(z_0,
w(a_{\phi},z_0^+) ),$ $z_0^-=(z_0, -w(a_{\phi},z_0^+)$.
%%%%%%%%%%%%%%%%%%%%%%%%%%%%%%%%%%%%%%%%%
%%%%%% Proposition 6.3 %%%%%%%
\begin{prop}\label{prop6.3}
For any $z_0\in \Pi_{a_{\phi},\phi},$ $w(z)=w(a_{\phi},z)$ fulfils
\begin{align*}
\frac{dz}{(z-z_0)w(z)}&=\frac 1{w(z_0^+)} d\ln \frac
{\vartheta(F(z_0^+,z)+\nu,\tau)}{\vartheta(F(z_0^-,z)+\nu,\tau)} -g_0(z_0)
\frac{dz}{w(z)},
\\
g_0(z_0)&=\frac{w'(z^+_0)}{2w(z^+_0)}-\frac 1{\omega_{\mathbf{a}} w(z_0^+)}
\Bigl( \pi i+\frac{\vartheta'}{\vartheta}(F(z^-_0,z^+_0)+\nu,\tau) \Bigr).
\end{align*}
\end{prop}
%%%%%%%%%%%%%%%%%%%%%%%%%%%%
\begin{proof}
Let $z=\varphi(u)$ be an elliptic function such that $\varphi_u^2=w(\varphi)^2$
and $z_0=\varphi(u_0)$. Then $(z-z_0)^{-1}w(z)^{-1}dz=(\varphi(u)-\varphi(u_0))
^{-1}du.$ Write $\varphi(u_0^{\pm})=z_0^{\pm},$ $\varphi_u(u^{\pm}_0)
=\pm w(z_0^+).$ Around each $u=u_0^{\pm}$, 
$$
(\varphi(u)-\varphi(u_0^{\pm}))^{-1}=\pm \frac{(u-u_0^{\pm})^{-1}}{w(z_0^{+})}
-\frac {w'(z_0^+)}{2w(z_0^+)}
+O(u-u_0^{\pm}), 
$$
and hence we have
\begin{align*}
(\varphi(u)-\varphi(u_0^{\pm}))^{-1}=& \frac 1{\omega_{\mathbf{a}}w(z_0)}
\Bigl(\frac{\vartheta'}{\vartheta}(\omega_{\mathbf{a}}^{-1}(u-u_0^+) +\nu,\tau)
-\frac{\vartheta'}{\vartheta}(\omega_{\mathbf{a}}^{-1}(u-u_0^-) +\nu,\tau)
\Bigr)+C_0
\\
=& \frac 1{\omega_{\mathbf{a}}w(z_0)}
\Bigl(\frac{\vartheta'}{\vartheta}(F(z_0^+,z) +\nu,\tau)
-\frac{\vartheta'}{\vartheta}(F(z_0^-,z)) +\nu,\tau)\Bigr)+C_0
\end{align*}
for some constant $C_0$. Passage to the limit $z\to z^+_0$, i.e. $u\to u_0^+$
leads to
$$
C_0=\frac 1{\omega_{\mathbf{a}}w(z_0^+)} \Bigl(
\frac{\vartheta'}{\vartheta}(F(z_0^-,z_0^+) +\nu,\tau) -\frac{\vartheta''(\nu)}
{2\vartheta'(\nu)}\Bigr)-\frac{w'(z^+_0)}{2w(z^+_0)},
$$
which implies the required formula.
\end{proof}
%%%%%%%%%%%%%%%%%%%%%%%%%%%%%%%
%%%%% Corollary 6.4 %%%%%%%%%%%%%
\begin{cor}\label{cor6.4}
For any $z_0 \in \Pi_{a_{\phi},\phi}$, and for $w(z)=w(a_{\phi},z)$,
%%%%%%%%%%%% (6.7) %%%%%% (6.8) %%%%%%%% (6.9) %%%%%%%%%%%%%%%
\begin{align}\label{6.7}
& \int_{\mathbf{a}} \frac{dz}{(z-z_0)w(z)}=-g_0(z_0) \omega_{\mathbf{a}},
\\
\label{6.8}
& \Bigl(\int_{\mathbf{b}} -\tau \int_{\mathbf{a}}\Bigr)\frac{dz}{(z-z_0)w(z)}
=\frac{2\pi i}{w(z_0^+)}F(z_0^-, z_0^+),
\\
\label{6.9}
& \Bigl(\int_{\mathbf{b}} -\tau \int_{\mathbf{a}}\Bigr)\frac{dz}{zw(z)}
=\frac{2\pi i}{a_{\phi}\omega_{\mathbf{a}}}e^{-3i\phi}.
\end{align}
\end{cor}
%%%%%% Corollary 6.5 %%%%%%%%%%%%%%%
\begin{cor}\label{cor6.5}
For $z_+,$ $z_-$ and $W(z)$ given by \eqref{6.2} and \eqref{6.5}, and for
$w(z)=w(a_{\phi},z)$, we have, on $\Pi_{a_{\phi},\phi}$,
\begin{align*}
 \int_{\mathbf{a}}W(z)dz=& \Bigl(z_+ -\frac{w'(z_+^+)}2\Bigr)\omega_{\mathbf{a}}
 +\frac{\vartheta'}{\vartheta}(F(z_+^-,z_+^+)+\nu,\tau)
\\
 &\phantom{--------} - \Bigl(z_- -\frac{w'(z_-^+)}2\Bigr)\omega_{\mathbf{a}} -
\frac{\vartheta'}{\vartheta}(F(z_-^-,z_-^+)+\nu,\tau),
\\
 \Bigl(\int_{\mathbf{b}}-\tau  \int_{\mathbf{a}}\Bigr)&W(z)dz=
    2\pi i(F(z_+^-,z_+^+)-F(z^-_-,z_-^+)).
\end{align*}
\end{cor}
%%%%%%%%%%%%%%%%%%%%%%%%%
Another expression of $\int_{\mathbf{a}}W(\lambda)d\lambda$ is derived by
using more information on the poles of $\mathrm{P}(u;a_{\phi})$ 
(cf. Proposition \ref{prop7.5}).
%%%%%%%%%%% Proposition 6.51 %%%%%%%%%%%%
\begin{prop}\label{prop6.51}
Under the same condition as above
\begin{equation*}
\int_{\mathbf{a}} W(z)dz
= 2\Bigl( \frac{\vartheta'}{\vartheta}(\tfrac 12F(z_+^-,z_+^+)+\tfrac 12
 +\tfrac{\tau}6,\tau)
 - \frac{\vartheta'}{\vartheta}(\tfrac 12F(z_-^-,z_-^+)+\tfrac 12
 +\tfrac{\tau}6,\tau) \Bigr).
\end{equation*}
\end{prop}
%%%%%%%%%%%%%%%%%%%%%%%%%
\begin{proof}
Let $z_0 \in \Pi_{a_{\phi},\phi}=\Pi_+\cup \Pi_-.$ Note that $z_0-\tfrac 12
w'(a_{\phi}, z_0)$ is holomorphic around $z_0^+=\infty^+ \in \Pi_+$ and
admits a pole at $z_0^+ =\infty^- \in \Pi_-$ with the residue $2$.
Then, we have
$$
(z_0-\tfrac 12 w'(a_{\phi},z^+_0))\omega_{\mathbf{a}} 
+(\vartheta'/\vartheta)(F(z_0^-,z_0^+)+\nu,\tau)=
2(\vartheta'/\vartheta)(\tfrac 12 F(z_0^-,z_0^+)+\tfrac 12 +\tfrac {\tau}6,\tau)
+c_0
$$
for some constant $c_0.$ Indeed,
as shown later by Proposition \ref{prop7.5}, the elliptic function 
$\varphi_0(u)$ defined by 
$u=\int_0^{\varphi_0}w(a_{\phi},z)^{-1}dz$ has poles with the residue $\mp 1$  
at $\pm \tfrac 13\omega_{\mathbf{b}} +\omega_{\mathbf{a}}\mathbb{Z} 
 +\omega_{\mathbf{b}} \mathbb{Z}$, which implies $\omega_{\mathbf{a}}^{-1}
\int^{\infty^+}_0 w(a_{\phi},z)^{-1}dz +\tfrac 12 +\tfrac {\tau}6
= \tfrac 12 -\tfrac {\tau}6$ for $z_0^+=\infty^+ \in \Pi_+,$ and 
$\omega_{\mathbf{a}}^{-1}\int^{\infty^-}_0 w(a_{\phi},z)^{-1}dz +\tfrac 12
+\tfrac {\tau}6= \tfrac 12 +\tfrac {\tau}2=\nu$ for $z_0^-=\infty^- \in \Pi_-$;  
and furthermore, at the branch points $z_0=0, z_1, z_3, z_5$,
the leading terms of $-\tfrac 12 w'(a_{\phi},z^+_{0})\omega_{\mathbf{a}}$ and 
$(\vartheta'/\vartheta)(F(z_{0}^-,z_{0}^+)+\nu,\tau)$ are cancelled out 
(cf.~Subsection \ref{ssc6.3}).
Putting $z_0=z_{\pm}$ and using Corollary \ref{cor6.5}, 
we obtain the proposition.
\end{proof}
%%%%%%%%%%%%%%%%%%%%%%%%%%%
%%%%%%%%%%%%%%%%%%%%%%%%%%%%%%%%%%%%%%%
%%%%%%% subSection 6.3 %%%%%%%
\subsection{Expression of $B_{\phi}(t)$}\label{ssc6.3}
%%%%%%%%%%%%%%%%%%%%%%%%%%%%%%%%%%%%%%%%%
Recall that our calculations are carried out under the supposition \eqref{5.1}
in the strip $S_{\phi}(t'_{\infty},\kappa_1,\delta_1).$ Let 
\begin{equation*}
\Omega_{\mathbf{a},\,\,\mathbf{b}}:=\int_{\mathbf{a},\,\,\mathbf{b}}\frac{dz}
{w(A_{\phi},z)}, \qquad
\mathcal{J}_{\mathbf{a},\,\,\mathbf{b}}:=\int_{\mathbf{a},\,\,\mathbf{b}}\frac
{w(A_{\phi},z)}{z} dz
\end{equation*}
for $\mathbf{a},$ $\mathbf{b}$ on 
$\Pi_{A_{\phi},\phi}=\Pi_+ \cup \Pi_- =\lim_{a_{\phi}\to A_{\phi}}
\Pi_{a_{\phi},\phi},$ i.e. $\Pi_{A_{\phi},\phi}=\lim_{t\to\infty}\Pi_{a_{\phi}
(t),\phi}$. We would like to express $B_{\phi}(t)$ defined by \eqref{5.1} 
in terms of these quantities.
\par
By Corollary \ref{cor6.5} the integral $\int_{\mathbf{a}}W(z)dz$ is a linear 
combination of $z_{\pm}$, $w'(z_{\pm}^+)$ and 
$(\vartheta'/\vartheta)(F(z_{\pm}^-,z_{\pm}^+)+\nu,\tau)$. 
In $S_{\phi}(t'_{\infty}, \kappa_1,\delta_1)$, the functions $\psi$, $1/\psi$ 
and $\psi'$ has no poles, and hence $z_{\pm}=z_{\pm}(t)$ are bounded in $S_{\phi}
(t'_{\infty},\kappa_1,\delta_1)$ and so are $w'(z^+_{\pm})$ except 
for neighbourhoods of the zeros of $w(z)$, i.e., $0, z_1, z_3, z_5$. 
Furthermore, around these 
points, the leading terms of $-\tfrac 12 w'(z^+_{\pm})\omega_{\mathbf{a}}$ and 
$(\vartheta'/\vartheta)(F(z_{\pm}^-,z_{\pm}^+)+\nu,\tau)$ are cancelled out. 
Indeed we have, say, around $z^+_{\pm}=0$, $-\tfrac 12 w'(z_{\pm}^+)
\omega_{\mathbf{a}}=-\tfrac 12 e^{3i\phi/2} A_{\phi}^{1/2}\omega_{\mathbf{a}}
(z_{\pm}^+)^{-1/2}+O(1)$ and $F(z_{\pm}^-,z_{\pm}^+)^{-1}=\tfrac 12 
e^{3i\phi/2}A_{\phi}^{1/2} \omega_{\mathbf{a}}(z_{\pm}^+)^{-1/2}+O(1)$. 
Since $z_{\pm}=z_{\pm}(t)$ moves on $\Pi_{a_{\phi},\phi}$ crossing $\mathbf{a}$-
and $\mathbf{b}$-cycles, $\omega_{\mathbf{a}}F(z_{\pm}^-,z_{\pm}^+)
=2p_{\pm}(t)\omega_{\mathbf{a}}+2q_{\pm}(t)\omega_{\mathbf{b}}+O(1)$ 
with $p_{\pm}(t),$ $q_{\pm}(t) \in\mathbb{Z},$ which implies
the boundedness of $\re (\vartheta'/\vartheta)(F(z^{-}_{\pm},z_{\pm}^+)+\nu,
\tau)$. Thus we have verified that $\re \int_{\mathbf{a}} W(z)dz$ is bounded in
$S_{\phi}(t'_{\infty},\kappa_1,\delta_1).$
\par
Suppose that $-\pi/4<\phi<0$. By \eqref{5.1}
\begin{equation*}
\frac 1z(w(a_{\phi},z)-w(A_{\phi},z))=\frac{2e^{3i\phi}t^{-1}B_{\phi}(t)}
{w(A_{\phi},z)}(1+O(t^{-1}B_{\phi}(t))).
\end{equation*}
By using this with $B_{\phi}(t)\ll 1$ the first formula in Proposition 
\ref{prop6.1} is written in the form
$$
\ln(1+s_1s_2)=-\frac t2\int_{\mathbf{a}}\Bigl(\frac{w(A_{\phi},z)}z+\frac
{2e^{3i\phi}t^{-1}B_{\phi}(t)}{w(A_{\phi},z)}\Bigr)dz
+\frac 14\int_{\mathbf{a}} W(z)dz
-C_{\alpha,\beta}(\mathbf{a})+\pi i+O(t^{-\delta}),
$$
that is,
$$
t\mathcal{J}_{\mathbf{a}}+2e^{3i\phi}\Omega_{\mathbf{a}}B_{\phi}(t)
=\frac 12\int_{\mathbf{a}}W(z)dz
 -2C_{\alpha,\beta}(\mathbf{a})+2\pi i -2\ln(1+s_1s_2)+O(t^{-\delta}).
$$
In Proposition \ref{prop6.1}, $(\ln(1+s_1s_2), \ln((1+s_1s_2)(1+s_2s_3)-1))$, 
generally depending on $t$, is a solution of the direct monodromy problem.
Suppose that
%%%%%%%%%%%%%%% (6.10) %%%%%%%%%%%%%%%
\begin{equation}\label{6.10}
\ln(1+s_1s_2)\ll 1, \quad  \ln((1+s_1s_2)(1+s_2s_3)-1) \ll 1 \,\,\, \text{in
$S_{\phi}(t'_{\infty},\kappa_1,\delta_1)$.} 
\end{equation}
By the Boutroux equations \eqref{2.2}, we have 
$\re e^{3i\phi}\Omega_{\mathbf{a}}B_{\phi}(t)\ll 1$ 
as $t\to\infty$ in $S_{\phi}(t'_{\infty},\kappa_1,\delta_1)$. From the
second formula of Proposition \ref{prop6.1} we similarly derive
$$
t\mathcal{J}_{\mathbf{b}}+2e^{3i\phi}\Omega_{\mathbf{b}}B_{\phi}(t)
=\frac 12\int_{\mathbf{b}}W(z)dz -2C_{\alpha,\beta}(\mathbf{b})
- 2\ln((1+s_1s_2)(1+s_2s_3)-1)+O(t^{-\delta}),
$$
in which $\int_{\mathbf{b}}W(z)dz$ is expressed by $\vartheta(z,\hat{\tau})$
with $\hat{\tau}=(-\omega_{\mathbf{a}})/\omega_{\mathbf{b}},$ and this formula 
yields $\re e^{3i\phi}\Omega_{\mathbf{b}} B_{\phi}(t) \ll 1$. 
These two estimates leads to
the inequality $|B_{\phi}(t)|\le C_0$ in $S_{\phi}(t'_{\infty},\kappa_1,
\delta_1)$ for some $C_0>0,$ while the implied constant of \eqref{5.1} may be 
supposed to be $2C_0$ if $t'_{\infty}$ is taken sufficiently large. 
Hence the boundedness of $B_{\phi}(t)$ may be derived under \eqref{6.10} 
independently of \eqref{5.1}. The case $0<\phi<\pi/4$ is discussed in the same
way by using the Boutroux equations \eqref{2.2}. 
%%%%%%%%%%%%%%%%%%%%%%%%%%%%%%%%%%%%%%
%%%%%% Remark 6.2 %%%%%%%%%%%%
\begin{rem}\label{rem6.2}
Under a relaxed condition, say $B_{\phi}(t)\ll \ln |t| $ instead of $\ll 1$
of \eqref{5.1},
each turning point is located within the distance $O(t^{-1}\ln|t|)$ from the
limit one, and Propositions \ref{prop6.1} and \ref{prop6.2} are obtained
by choosing a slightly smaller $\delta$. Then the equivalence between 
\eqref{6.10} and the boundedness of $B_{\phi}(t)$ may be proved as above. 
\end{rem}
%%%%%%%%%%%%%%%%%%%%%%%%%%%%%%%%%%%%%%%%%%%%%%%%%%%%%%%%%
%%%%% Proposition 6.6 %%%%%%%%%%%%%%
\begin{prop}\label{prop6.6}
Suppose that $0<|\phi|<\pi/4$ and let
$\mathfrak{l}(\mathbf{s},\phi)=\ln(1+s_1s_2)$ if $-\pi/4<\phi<0,$ and 
$= -\ln (1+s_2s_3)$ if $0<\phi<\pi/4.$ 
In $S_{\phi}(t'_{\infty},\kappa_1,\delta_1)$, we have
$$
\mathfrak{l}(\mathbf{s},\phi) \ll 1,\quad \ln((1+s_1s_2)(1+s_2s_3)-1) \ll 1 
$$
if and only if $B_{\phi}(t)\ll 1$, and then
$$
t\mathcal{J}_{\mathbf{a}}+2 e^{3i\phi}\Omega_{\mathbf{a}}B_{\phi}(t)
=\frac 12\int_{\mathbf{a}}W(z)dz -2C_{\alpha,\beta}(\mathbf{a})+2\pi i 
-2\mathfrak{l}(\mathbf{s},\phi) +O(t^{-\delta}).
$$
\end{prop}
%%%%%%%%%%%%%%%%%%%%%%%%%
The following fact guarantees the possibility of the limit $a_{\phi}\to
A_{\phi}$ under integration.
%%%%%%%%%%%%%%%%%%%%%%%%%%%%%%%%%%%%%%%%%%%
%%%% Proposition 6.7 %%%%%%%%%%%
\begin{prop}\label{prop6.7}
Suppose that $0<|\phi|<\pi/4.$ Then, in $S_{\phi}(t'_{\infty},\kappa_1,
\delta_1)$,
$$
\biggl(\int^{z^+_+}_{z^-_+} -\int^{z^+_-}_{z^-_-} \biggr)\frac{dz}{w(a_{\phi},z)}
=\biggl(\int^{z^+_+}_{z^-_+}-\int^{z^+_-}_{z^-_-}\biggr)\frac{dz}{w(A_{\phi},z)}
+O(t^{-1}).
$$
\end{prop}
%%%%%%%%%%%%%%%%%%%%%%%%%%%%%%%%%
\begin{proof}
By \eqref{5.1} it is easy to see $\omega_{\mathbf{a},\,\,\mathbf{b}}
=\Omega_{\mathbf{a},\,\, \mathbf{b}}+O(t^{-1}).$ Suppose that $-\pi/4<\phi<0$. 
By Proposition \ref{prop6.1}, Remark \ref{rem6.1}, Corollary \ref{cor6.5} and
\eqref{6.9},
\begin{align*}
\ln((1+s_1s_2)& (1+s_2s_3)-1)- \tau\ln(1+s_1s_2)
\\
=& -a_{\phi}e^{4i\phi}t \Bigl(\int_{\mathbf{b}}-\tau\int_{\mathbf{a}} \Bigr)
\frac{dz}{zw(a_{\phi},z)} 
+\frac 14 \Bigl(\int_{\mathbf{b}}-\tau\int_{\mathbf{a}} \Bigr)W(z)dz
\\
& -(C_{\alpha,\beta}(\mathbf{b})-\tau C_{\alpha,\beta}(\mathbf{a})) - \pi i\tau
+O(t^{-\delta})
\\
=& -\frac{2\pi i e^{i\phi}t}{\omega_{\mathbf{a}}}+\frac{\pi i}{2}
(F(z^-_+,z^+_+)-F(z^-_-,z^+_-)) -C_{\alpha,\beta}^* -\pi i\tau +O(t^{-\delta})
\end{align*}
with $C_{\alpha,\beta}^*=C_{\alpha,\beta}(\mathbf{b})-\tau C_{\alpha,\beta}
(\mathbf{a}).$ This implies
$$
-2e^{i\phi}t + \frac 12 \biggl(\int_{z^-_+}^{z^+_+}-\int^{z^+_-}_{z^-_-}\biggr)
\frac{dz}{w(a_{\phi},z)} =O(1).
$$
Write $\omega_{\mathbf{a}}F(z_{\pm}^-,z_{\pm}^+)=2p_{\pm}(t)\omega_{\mathbf{a}}
+2q_{\pm}(t)\omega_{\mathbf{b}} +O(1)$ with 
$p_{\pm}(t),$ $q_{\pm}(t)\in \mathbb{Z}$, 
and set $-i\mathcal{J}_{\mathbf{a}}t/2+\pi q(t)=X,$
$-i\mathcal{J}_{\mathbf{b}}t/2-\pi p(t)=Y,$ where $p(t)=p_+(t)-p_-(t),$
$q(t)=q_+(t)-q_-(t).$ By the Boutroux equations \eqref{2.2}, $\im X$ and $\im Y$
are bounded. Note that $\omega_{\mathbf{b}}\mathcal{J}_{\mathbf{a}} 
- \omega_{\mathbf{a}}\mathcal{J}_{\mathbf{b}}=  
\Omega_{\mathbf{b}}\mathcal{J}_{\mathbf{a}} 
- \Omega_{\mathbf{a}}\mathcal{J}_{\mathbf{b}}+O(t^{-1})
= -4\pi i e^{i\phi} +O(t^{-1})$ by Corollary \ref{cor8.10}.  
Then the estimate above is
\begin{align*}
& -2e^{i\phi}t +\omega_{\mathbf{a}}p(t) +\omega_{\mathbf{b}}q(t)+O(1)
\\
=&-2e^{i\phi}t +\pi^{-1}\omega_{\mathbf{a}}(-Y-i\mathcal{J}_{\mathbf{b}}t/2) 
 +\pi^{-1}\omega_{\mathbf{b}}(X+i\mathcal{J}_{\mathbf{a}}t/2) +O(1)
\\
=&-2e^{i\phi}t +\pi^{-1}it(\omega_{\mathbf{b}}\mathcal{J}_{\mathbf{a}} 
- \omega_{\mathbf{a}}\mathcal{J}_{\mathbf{b}})/2 
 +\pi^{-1}(\omega_{\mathbf{b}}X-\omega_{\mathbf{a}}Y) +O(1)
\\
=& \pi^{-1}(\omega_{\mathbf{b}}X-\omega_{\mathbf{a}}Y) +O(1)\ll 1
\end{align*}
with $\im(\omega_{\mathbf{b}}/\omega_{\mathbf{a}})>0$ uniformly.
This implies $X$, $Y \ll 1$, and hence 
$$
\pi p(t)=-i\mathcal{J}_{\mathbf{b}} t/2 +O(1), \quad 
\pi q(t)=i\mathcal{J}_{\mathbf{a}} t/2 +O(1).  
$$
Observing that $w(a_{\phi},z)^{-1}-w(A_{\phi},z)^{-1}=-2ze^{3i\phi} 
B_{\phi}(t)t^{-1}w(A_{\phi},z)^{-3} +O(t^{-2}),$ we have
\begin{align*}
& \biggl| \biggl( \int^{z^+_+}_{z^-_+} -\int^{z^+_-}_{z^-_-} \biggr) \Bigl(
\frac 1{w(a_{\phi},z)} -\frac 1{w(A_{\phi},z)} \Bigr)dz \biggr|
\ll \biggl| \biggl( \int^{z^+_+}_{z^-_+} -\int^{z^+_-}_{z^-_-} \biggr) 
\frac {zB_{\phi}(t)t^{-1}}{w(A_{\phi},z)^3} dz \biggr| +O(t^{-1})
\\
& \ll \biggl|t^{-1} \biggl( \int^{z^+_+}_{z^-_+} -\int^{z^+_-}_{z^-_-} \biggr) 
\frac {z dz }{w(A_{\phi},z)^3}\biggr| +O(t^{-1})
\ll |t^{-1} ( p(t)j_{\mathbf{a}}+q(t)j_{\mathbf{b}}) |+O(t^{-1})
\\
& =|\mathcal{J}_{\mathbf{b}} j_{\mathbf{a}} - \mathcal{J}_{\mathbf{a}}
 j_{\mathbf{b}} |+O(t^{-1})
=\tfrac 12 
|(\partial/\partial A_{\phi}) (\mathcal{J}_{\mathbf{b}} \Omega_{\mathbf{a}}
 - \mathcal{J}_{\mathbf{a}} \Omega_{\mathbf{b}} )|+O(t^{-1}) \ll t^{-1}
\end{align*}
with $j_{\mathbf{a},\,\, \mathbf{b}}=\int_{\mathbf{a},\,\, \mathbf{b}}
z w(A_{\phi},z)^{-3} dz$. This completes the proof.
\end{proof} 
%%%%%%%%%%%%%%%%%%%%%%%%%%%%%%%%%%%%
%%%%%%% Section 7 %%%%%%%%%%%%%%%%%%
\section{Proofs of the main results}\label{sc7}
%%%%%%%%%%%%%%%%%%%%%%%%%%%%%%%%%%%%%%%%%%%%%%%%%%
\subsection{Derivation of $\psi(t)$}\label{ssc7.1}
%%%%%%%%%%%%%%%%%%%%%%%%%%%%%%%%%%%%%%%%%%%%%%%%%%%%%%
Suppose that $-\pi/4<\phi<0.$ Let $\mathbf{s}=(s_1,s_2,s_3,s_4)$ with 
$(1+s_1s_2)(1+ s_2s_3)-1 \not=0,$ $1+s_1s_2\not=0$ be a solution of the
direct monodromy problem discussed above in which $(\psi,\psi_t)$ be such
that \eqref{5.1} is valid in $S_{\phi}(t'_{\infty},\kappa_1,\delta_1)$.
As in the proof of Proposition \ref{prop6.7} we set 
\begin{align*}
\ln((1+s_1s_2)& (1+s_2s_3)-1)- \tau\ln(1+s_1s_2)
\\
=& -\frac{2\pi i e^{i\phi}t}{\omega_{\mathbf{a}}}+\frac{\pi i}{2}
(F(z^-_+,z^+_+)-F(z^-_-,z^+_-)) -C_{\alpha,\beta}^*-\pi i\tau +O(t^{-\delta})
\end{align*}
with $C_{\alpha,\beta}^*=C_{\alpha,\beta}(\mathbf{b})-\tau C_{\alpha,\beta}
(\mathbf{a}).$ Then by Proposition \ref{prop6.7} we have
%%%%%%%%%%%% (7.1) %%%%%%%%%%%%%%%%%%%%%
\begin{equation}\label{7.1}
\biggl(\int^{z^+_+}_0 - \int^{z^+_-}_0 \biggr) \frac{dz}{w(A_{\phi},z)}
=2e^{i\phi}t + 2\tilde{\chi}_0 +\Omega_{\mathbf{b}}+O(t^{-\delta}),
\end{equation}
in which
\begin{align*}
& 2\tilde{\chi}_0=\frac 1{\pi i}(\Omega_{\mathbf{a}}
\ln((1+s_1s_2)(1+s_2s_3)-1)- \Omega_{\mathbf{b}}\ln(1+s_1s_2) )
+\Gamma_{\alpha,\beta},
\\
&\Gamma_{\alpha,\beta}=\frac {(\alpha-\beta)}{2\pi i}\Bigl(\Omega_{\mathbf{a}}
\int_{\mathbf{b}} \frac{z dz}{w(A_{\phi},z)} - \Omega_{\mathbf{b}}
\int_{\mathbf{a}} \frac{z dz}{w(A_{\phi},z)} \Bigr).
\end{align*}
%%%%%%%%%%%%%%%%%%%%%%%%
\par
Let us derive the asymptotic form as in Theorem \ref{thm2.1} from \eqref{7.1}. 
Set
$$
u=-\int^{z^+_{+}}_0 \frac{dz}{w(A_{\phi},z)}, \quad
v=-\int^{z^+_{-}}_0 \frac{dz}{w(A_{\phi},z)}.
$$
By the change of variables 
$z= e^{3i\phi}A_{\phi}(\wp -\tfrac 13 e^{2i\phi})^{-1}$ 
(cf. Remark \ref{rem2.1}),
\begin{align*}
 &-\int^{z^+_{\pm}}_0 \frac{dz}{w(A_{\phi},z)}=
 \int^{\wp_{\pm}}_{\infty} \frac{d\wp}{\sqrt{4\wp^3-g_2\wp-g_3}},
\\  
&\wp_{\pm}=\tfrac 13 {e^{2i\phi}} +e^{3i\phi}{A_{\phi}}z_{\pm}^{-1}, \quad
 g_2=e^{4i\phi}(\tfrac 43-4A_{\phi}), \quad
g_3=e^{6i\phi}(-\tfrac 8{27} +\tfrac 43 A_{\phi}-A_{\phi}^2),
\end{align*}
and hence $\wp(u)=\wp_+,$ $\wp(v)=\wp_-$. Observing that,
by \eqref{4.2}, \eqref{5.1} and \eqref{6.2},
\begin{align*}
&z_+ + z_-= -\psi-2e^{i\phi}+O(t^{-1}), 
\quad z_+z_-= -e^{3i\phi}A_{\phi}\psi^{-1} +O(t^{-1}),
\\
& w(A_{\phi},z_{\pm})=\lambda b_3(\lambda)|_{\lambda^2=z_{\pm}}
=z_{\pm}^2+2(e^{i\phi} +\psi)z_{\pm}+O(t^{-1}),
\end{align*}
we have, up to the error term $+O(t^{-1})$,
\begin{align*}
\wp(u)+\wp(v)=&\wp_+ +\wp_- =\tfrac 23 e^{2i\phi}+e^{3i\phi}
A_{\phi}(z_++z_-)(z_+z_-)^{-1}
=\tfrac 23 e^{2i\phi}+\psi^2+2e^{i\phi}\psi,
\\
-(\wp'(u)-\wp'(v))=&e^{3i\phi}A_{\phi}
( z_+^{-2}w(A_{\phi},z_+)-z_-^{-2}w(A_{\phi},z_-))
= 2(e^{i\phi}+\psi)e^{3i\phi}A_{\phi}(z_+^{-1}-z_-^{-1})
\\
=& 2(e^{i\phi}+\psi)(\wp_+ -\wp_-)
= 2(e^{i\phi}+\psi)(\wp(u) -\wp(v)).
\end{align*}
Then by the addition theorem
$$
\wp(u+v)=-\wp(u)-\wp(v)+ \frac 14 \Bigl(\frac{\wp'(u)-\wp'(v)}{\wp(u)-\wp(v)}
\Bigr)^2 =\frac {e^{2i\phi}}3 +O(t^{-1}),
$$
which implies, by Proposition \ref{prop7.5},
\begin{align*}
\biggl(\int^{z_+^+}_{0} + \int^{z_-^+}_0 \biggr) \frac{dz}{w(A_{\phi},z)}
=& -(u+v)=-\int_{\infty}^{e^{2i\phi}/3} \frac{d\wp}{\sqrt{4\wp^3-g_2\wp-g_3}}
\\
=&\int^{\infty}_0 \frac{dz}{w(A_{\phi},z)}
=\Omega_0 \in\{\pm \tfrac 13\Omega_{\mathbf{b}}\}.
\end{align*}
This combined with \eqref{7.1} leads to
%%%%%%%%%%%%%%%%% (7.2) %%%%%%%%%%%%%%%%%%%%%%
\begin{equation}\label{7.2}
\begin{split}
&\int^{z^+_+}_0 \frac{dz}{w(A_{\phi},z)}= e^{i\phi}t+\tilde{\chi}_0+ \frac 12(
\Omega_{\mathbf{b}}+\Omega_0)+O(t^{-\delta}), 
\\
&\int^{z^+_-}_0 \frac{dz}{w(A_{\phi},z)}= -e^{i\phi}t-\tilde{\chi}_0-\frac 12(
\Omega_{\mathbf{b}} -\Omega_0)+O(t^{-\delta}).
\end{split}
\end{equation}
%%%%%%%%%%%%%%%%%%%%%%%%%%%%%%%%
%%%% Remark 7.1 %%%%%%%%%%%%%%%%%%
\begin{rem}\label{rem7.1}
A similar argument with use of the addition theorem leads to
\begin{align*}
 &\biggl(\int^{z_+^+}_{0} + \int^{z_-^+}_0 \biggr) \frac{dz}{w(a_{\phi},z)}
 =-\int_{\infty}^{e^{2i\phi}/3} \frac{d\wp}{\sqrt{4\wp^3-\tilde{g}_2\wp-
\tilde{g}_3}} +O(t^{-1})
\\
& =-\int_{\infty}^{e^{2i\phi}/3} \frac{d\wp}{\sqrt{4\wp^3-{g}_2\wp-
{g}_3}} +O(t^{-1})
=\biggl(\int^{z_+^+}_{0} + \int^{z_-^+}_0 \biggr) \frac{dz}{w(A_{\phi},z)}
+O(t^{-1}),
\end{align*}
where $\tilde{g}_2=e^{4i\phi}(\tfrac 43-4a_{\phi}),$
$\tilde{g}_3=e^{6i\phi}(-\tfrac {27}8+\tfrac 4 3a_{\phi}-a_{\phi}^2).$
Combining this with Proposition \ref{prop6.7} we have
$$
 \int^{z^+_{\pm}}_0\frac {dz}{w(a_{\phi},z)}= 
 \int^{z^+_{\pm}}_0\frac {dz}{w(A_{\phi},z)}+O(t^{-1}). 
$$
\end{rem}
%%%%%%%%%%%%%%%%%%%%%%%%%%%%%%%%
Recall that $\mathrm{P}(u;A_{\phi})$ solves $\mathrm{P}_u^2
=w(A_{\phi},\mathrm{P})^2.$ By Proposition \ref{prop7.5}, 
%%%%%%%%%% (7.3) %%%%%%%%%%%%%%%%%%%%%%
\begin{equation}\label{7.3}
\mathrm{P}(u;A_{\phi})= \frac 1{\Omega_{\mathbf{a}}} \Bigl( 
\frac{\vartheta'}{\vartheta}
\Bigl(\frac u{\Omega_{\mathbf{a}}} +\frac {\tau^*}3 +\nu, \tau^*\Bigr)
-\frac{\vartheta'}{\vartheta}
\Bigl(\frac u{\Omega_{\mathbf{a}}} -\frac {\tau^*}3 +\nu,\tau^*\Bigr)
+ C_{\mathrm{P}} \Bigr)
\end{equation}
with $C_{\mathrm{P}}=-(\vartheta'/\vartheta)( \tfrac 13\tau^* +\nu,\tau^*)
+(\vartheta'/\vartheta)(-\tfrac 13\tau^* +\nu,\tau^*)$ and 
$\tau^*=\Omega_{\mathbf{b}}/\Omega_{\mathbf{a}}$. Then we have
\begin{align*}
\psi(t)=&-z_+ - z_- -2e^{i\phi}+O(t^{-1})
\\
=& -\mathrm{P}(\tilde{t}_{\phi}+\tfrac 12 (\Omega_{\mathbf{b}} 
+\Omega_0);A_{\phi})
 -\mathrm{P}(\tilde{t}_{\phi}+\tfrac 12 (\Omega_{\mathbf{b}} 
-\Omega_0);A_{\phi})
 -2e^{i\phi}+O(t^{-\delta})
\\
=& -\frac 1{\Omega_{\mathbf{a}}}
\Bigl(\frac{\vartheta'}{\vartheta}\Bigl(\frac{\tilde{t}_{\phi}}
{\Omega_{\mathbf{a}}}-\frac{\tau^*}6 \pm \frac{\tau^*}6 +\nu \Bigr)
 -\frac{\vartheta'}{\vartheta}\Bigl(\frac{\tilde{t}_{\phi}}
{\Omega_{\mathbf{a}}}+\frac{\tau^*}6 \pm \frac{\tau^*}6 +\nu \Bigr)
\\
 &+\frac{\vartheta'}{\vartheta}\Bigl(\frac{\tilde{t}_{\phi}}
{\Omega_{\mathbf{a}}}-\frac{\tau^*}6 \mp \frac{\tau^*}6 +\nu \Bigr)
 -\frac{\vartheta'}{\vartheta}\Bigl(\frac{\tilde{t}_{\phi}}
{\Omega_{\mathbf{a}}}+\frac{\tau^*}6 \mp \frac{\tau^*}6 +\nu \Bigr)
\Bigr) +C_{\phi}+O(t^{-\delta})
\\
=&  \frac 1{\Omega_{\mathbf{a}}}
\Bigl( \frac{\vartheta'}{\vartheta}\Bigl(\frac{\tilde{t}_{\phi}}
{\Omega_{\mathbf{a}}}+\frac{\tau^*}3 +\nu \Bigr)
 -\frac{\vartheta'}{\vartheta}\Bigl(\frac{\tilde{t}_{\phi}}
{\Omega_{\mathbf{a}}}-\frac{\tau^*}3 +\nu \Bigr)
\Bigr) +C_{\phi}+O(t^{-\delta})
\\
=& \mathrm{P}(\tilde{t}_{\phi};A_{\phi}) +C_{\phi,0} +O(t^{-\delta}),
\end{align*}
where $C_{\phi}, C_{\phi,0} \in \mathbb{C},$ and 
$\tilde{t}_{\phi}=e^{i\phi}t+\tilde{\chi}_0$. 
If $\psi_t=(d/dt)\psi +O(t^{-1})$, then by \eqref{4.2}
$e^{-2i\phi}((d/dt)\psi)^2=w(A_{\phi},\psi)^2 +O(t^{-1})$, which implies 
$C_{\phi,0}=0.$ In the case $0<\phi<\pi/4$, for $\mathbf{s}$ such that
$(1+s_1s_2)(1+s_2s_3)-1\not=0,$ $1+s_2s_3 \not=0$, the same argument is 
possible. Thus we have the following.
%%%%%%%%%%%%%%%%%%%%%%%%%%%%%%%%%%%%%%%%%%
%%%%%%% Proposition 7.1 %%%%%%%%%%%
\begin{prop}\label{prop7.1}
If $\psi_t=d\psi/dt+O(t^{-1}),$ then
equation \eqref{7.1} in $S_{\phi}(t'_{\infty},\kappa_1,\delta_1)$ implies
$$
\psi(t)=\mathrm{P}(e^{i\phi}t +\tilde{\chi}_0;A_{\phi})+O(t^{-\delta}),
$$
in which 
\begin{equation*}
 2\tilde{\chi}_0=\frac 1{\pi i}(\Omega_{\mathbf{a}}
\ln((1+s_1s_2)(1+s_2s_3)-1)- \Omega_{\mathbf{b}}\mathfrak{l}(\mathbf{s},\phi))
+\Gamma_{\alpha,\beta},
\end{equation*}
with $\mathfrak{l}(\mathbf{s},\phi)=\ln(1+s_1s_2)$ if $-\pi/4<\phi<0,$ and 
$= -\ln(1+s_2s_3) $ if $0<\phi<\pi/4$. 
\end{prop}
%%%%%%%%%%%%%%%%%%%%%%%%%%%%%%%%
Let us calculate the value $\Gamma_{\alpha,\beta}$ in $\tilde{\chi}_0.$
The substitution $z=\mathrm{P}(u;A_{\phi})$ with \eqref{7.3} leads to
\begin{align*}
\int_{\mathbf{b}} \frac{zdz}{w(A_{\phi},z)} =&\int^{\Omega_{\mathbf{b}}}_0
\mathrm{P}(u;A_{\phi})du
\\
=& \int^{\Omega_{\mathbf{b}}}_0 
\frac 1{\Omega_{\mathbf{a}}} \Bigl( 
\frac{\vartheta'}{\vartheta}
\Bigl(\frac u{\Omega_{\mathbf{a}}} +\frac {\tau^*}3 +\nu \Bigr)
-\frac{\vartheta'}{\vartheta}
\Bigl(\frac u{\Omega_{\mathbf{a}}} -\frac {\tau^*}3 +\nu \Bigr)
+ C_{\mathrm{P}} \Bigr) du
\\
=& \Bigl[ \ln \vartheta(\Omega_{\mathrm{a}}^{-1}u+\tfrac 13 \tau^* +\nu)
 -  \ln \vartheta(\Omega_{\mathrm{a}}^{-1}u-\tfrac 13 \tau^* +\nu)\Bigr]_0
 ^{\Omega_{\mathbf{b}}} +C_{\mathrm{P}} \tau^*
\\
=& -\pi i(\tau^* + 2(\tfrac 13\tau^* +\nu))
    +\pi i(\tau^* + 2(-\tfrac 13\tau^* +\nu))+ C_{\mathrm{P}}\tau^*
\\
=& (-\tfrac 43 \pi i + C_{\mathrm{P}})\tau^*, 
\end{align*}
and
$$
\int_{\mathbf{a}} \frac{zdz}{w(A_{\phi},z)} 
= \int^{\Omega_{\mathbf{a}}}_0 
\frac 1{\Omega_{\mathbf{a}}} \Bigl( 
\frac{\vartheta'}{\vartheta}
\Bigl(\frac u{\Omega_{\mathbf{a}}} +\frac {\tau^*}3 +\nu \Bigr)
-\frac{\vartheta'}{\vartheta}
\Bigl(\frac u{\Omega_{\mathbf{a}}} -\frac {\tau^*}3 +\nu \Bigr)
+ C_{\mathrm{P}} \Bigr) du
= C_{\mathrm{P}}.
$$
From these quantities, the required constant follows.
%%%%%%%% Proposition 7.2 %%%%%%%%%%%%%%%%%
\begin{prop}\label{prop7.2}
We have
\begin{align*}
&\Gamma_{\alpha,\beta}=\Gamma_{\alpha,\beta,\mathbf{b}}
  -\Gamma_{\alpha,\beta,\mathbf{a}} = -\frac 23(\alpha-\beta)\Omega_{\mathbf{b}},
\\
&\Gamma_{\alpha,\beta,\mathbf{b}}=\Gamma_{\alpha,\beta}+\frac 1{2\pi i}
(\alpha-\beta)C_{\mathrm{P}}\Omega_{\mathbf{b}},
\quad
\Gamma_{\alpha,\beta,\mathbf{a}}=\frac 1{2\pi i}
(\alpha-\beta)C_{\mathrm{P}}\Omega_{\mathbf{b}}
\end{align*}
with $C_{\mathrm{P}}=-(\vartheta'/\vartheta)( \tfrac 13\tau^* +\nu,\tau^*)
+(\vartheta'/\vartheta)(-\tfrac 13\tau^* +\nu,\tau^*)$.  
\end{prop}
By Propositions \ref{prop7.1} and \ref{7.2} we may derive from \eqref{7.1}
asymptotic forms as in Theorems \ref{thm2.1} and \ref{thm2.2}.
%%%%%%%%%%%%%%%%%%%%%%%%%%%%%%%%%%%%%%%%%%%%%%%%%%%%%%%%%%%%
%%% Proposition 7.21 %%%%%%%%%%%%%%%%%%%%
\begin{prop}\label{prop7.21}
In \eqref{7.2}, $\Omega_0=-\tfrac 13 \Omega_{\mathbf{b}}$. 
\end{prop}
%%%%%%%%%%%
\begin{proof}
Note that $\psi(t)=\mathrm{P}(e^{i\phi}t+\tilde{\chi}_0 
;A_{\phi})+O(t^{-\delta})$. 
By Proposition \ref{prop7.5},
if $\psi(t)= e^{3i\phi}A_{\phi} (e^{i\phi}t-t_0)^2(1+o(1))$ around $t=t_0$, then 
$\psi(t)\sim (e^{i\phi}t-t_0^-)^{-1}-e^{i\phi}+o(1)$ and
$\psi(t)\sim -(e^{i\phi}t-t_0^+)^{-1}-e^{i\phi}+o(1)$
around each of the points $t_0^{\pm}:=t_0\pm \tfrac 13 
\Omega_{\mathbf{b}}$ up to $O(t^{-\delta}).$ Then, by
$z_{\pm}(t)= \tfrac 12 (\mp e^{-i\phi}\psi_t\psi^{-1}-\psi-2e^{i\phi})
+O(t^{-1})$, it follows that
\begin{align*}
& z_+(t)=O(e^{i\phi} t- t_0^-),  
\quad z_+(t) = -(e^{i\phi}t-t_0)^{-1}+O(1),
\quad z_+(t)= (e^{i\phi}t-t_0^+)^{-1}+O(1),
\\
& z_-(t) = -(e^{i\phi}t-t_0^-)^{-1}+O(1), 
\quad z_-(t) = (e^{i\phi}t-t_0)^{-1}+O(1),
\quad z_-(t) =O(e^{i\phi}t - t_0^+)
\end{align*}
around each point. Since
$z_{\pm}(t)=\mathrm{P}(e^{i\phi}t+\tilde{\chi}_0 +\tfrac 12(\Omega_{\mathbf{b}}
\pm \Omega_0);A_{\phi})+O(t^{-\delta})$
we conclude that $\Omega_0=-\tfrac 13 \Omega_{\mathbf{b}}.$
\end{proof}
%%%%%%%%%%%%%%%%%%%%%%%%%%%%%%%%%%%%%%%%%%%%%%%%%%%%%%%%%%%%%%%
\subsection{Asymptotic representation of $B_{\phi}(t)$}\label{ssc7.2}
%%%%%%%%%%%%%%%%%%%%%%%%%%
Recalling Propositions \ref{prop6.51} and \ref{prop6.6}, and applying 
Remark \ref{rem7.1}, we have
\begin{align*}
t\mathcal{J}_{\mathbf{a}}+2e^{3i\phi}\Omega_{\mathbf{a}}B_{\phi}(t)
=&\frac{\vartheta'}{\vartheta}(\tfrac 12F_*(z^-_+,z^+_+)+\tfrac 12
 +\tfrac{\tau^*}6,\tau^*)
-\frac{\vartheta'}{\vartheta}(\tfrac 12 F_*(z^-_-,z^+_-)+\tfrac 12
 +\tfrac{\tau^*}6,\tau^*)
\\
&- (\alpha-\beta)C_{\mathrm{P}} +2\pi i -2\mathfrak{l}(\mathbf{s},\phi)
+O(t^{-\delta}),
\end{align*}
where 
$$
F_*(z^-_{\pm},z^+_{\pm}) 
=\frac 1{\Omega_{\mathbf{a}}} \int_{z^-_{\pm}}^{z^+_{\pm}} \frac{dz}
{w(A_{\phi},z)}
=\frac 2{\Omega_{\mathbf{a}}} \int_{0}^{z^+_{\pm}} \frac{dz}
{w(A_{\phi},z)}.
$$
By Proposition \ref{prop7.21}, insertion of \eqref{7.2} with 
$\Omega_0=- \tfrac 13 \Omega_{\mathbf{b}}$ yields
the asymptotic representation of the correction function $B_{\phi}(t),$
which plays an essential role in the justification of our asymptotic
solution.
%%%%%%%%%%%%%%%%%% Proposition 7.3 %%%%%%%%%%%%%%%%
\begin{prop}\label{prop7.3}
In $S_{\phi}(t'_{\infty},\kappa_1,\delta_1),$
\begin{equation*}
t\mathcal{J}_{\mathbf{a}}+2 e^{3i\phi}\Omega_{\mathbf{a}}B_{\phi}(t)
=2 \frac{\vartheta'}{\vartheta}\Bigl(\frac {e^{i\phi}t +\tilde{\chi_0}}{\Omega
_{\mathbf{a}}}   +\nu,\tau^*\Bigr) 
- (\alpha-\beta)C_{\mathrm{P}} +2\pi i 
-2\mathfrak{l}(\mathbf{s},\phi)+O(t^{-\delta}),
\end{equation*}
where $\mathfrak{l}(\mathbf{s},\phi)$ and $C_{\mathrm{P}}$ are constants
given in Propositions \ref{prop7.1} and \ref{prop7.2}.
\end{prop}
%%%%%%%%%%%%%%%%%%%%%%%%%%%%%%%%%%%%
\subsection{Proofs of Theorems \ref{thm2.1} and \ref{thm2.2}}\label{ssc7.3}
For a prescribed monodromy data $\mathbf{s}$, the asymptotic expression of
$\psi(t)$ given in Proposition \ref{prop7.1} is, at least formally, 
a solution of the inverse monodromy problem. To prove Theorems \ref{thm2.1}
and \ref{thm2.2} let us make the justification for $y=e^{-i\phi}x\psi(t)$
as a solution of $\mathrm{P}_{\mathrm{IV}}$ along the lines in 
\cite[pp.~105--106, pp.~120--121]{Kitaev-3}.  
Suppose that $-\pi/4<\phi<0.$ Let $\mathbf{s}=(s_1,s_2,s_3,s_4)$ with
$(1+s_1s_2)(1+s_2s_3)-1\not=0,$ $1+s_1s_2\not=0$ be a given point on the 
monodromy manifold $\mathcal{M}_0(\alpha,\beta)$ for isomonodromy system
\eqref{3.1}. Set
\begin{align*}
\psi_{\mathrm{as}}=\psi_{\mathrm{as}}(\mathbf{s},t)
:=&\mathrm{P}(e^{i\phi}t +\tilde{\chi}_0; A_{\phi}),
\\
(B_{\phi})_{\mathrm{as}}=(B_{\phi})_{\mathrm{as}}(\mathbf{s},t)
:=&\frac{e^{-3i\phi}}{2\Omega_{\mathbf{a}}} \Bigl(
2\frac{\vartheta'}{\vartheta}\Bigl(\frac {e^{i\phi}t +\tilde{\chi_0}}{\Omega
_{\mathbf{a}}}   +\nu,\tau^*\Bigr) 
-t\mathcal{J}_{\mathbf{a}}+ C_{\alpha,\beta,\mathbf{s},\phi} \Bigr),
\\
 C_{\alpha,\beta,\mathbf{s},\phi}&=2\pi i- (\alpha-\beta)C_{\mathrm{P}} 
-2\mathfrak{l}(\mathbf{s},\phi), 
\end{align*}
which are leading term expressions of $\psi(t)$ and $B_{\phi}(t)$ without
$O(t^{-\delta})$ in Propositions \ref{prop7.1} and \ref{prop7.3}.
Taking \eqref{4.2} and \eqref{5.1} into account, we set
\begin{align*}
e^{-i\phi}\psi^*_{\mathrm{as}}=& -\tfrac 12 e^{-i\phi} \psi_{\mathrm{as}}t^{-1} 
+\sqrt{\psi_{\mathrm{as}} ((\psi_{\mathrm{as}}+2e^{i\phi})^2\psi
_{\mathrm{as}}+4e^{3i\phi}A_{\phi})
 +\Delta(t,\psi_{\mathrm{as}},(B_{\phi})_{\mathrm{as}})t^{-1}  },
\\
\Delta(t,\psi_{\mathrm{as}}, & (B_{\phi})_{\mathrm{as}})
=\psi_{\mathrm{as}}( 4(B_{\phi})_{\mathrm{as}}-(4\alpha-\beta)\psi_{\mathrm{as}}
-2(2\alpha-\beta)e^{i\phi}) +\tfrac 14\beta^2t^{-1},
\end{align*}
where the branch of the square root is chosen in such a way that $\psi^*
_{\mathrm{as}}$ is compatible with $(d/dt)\psi_{\mathrm{as}}$, that is,
$\psi^*_{\mathrm{as}}=(d/dt)\psi_{\mathrm{as}} +O(t^{-1}).$
Then for $a_{\phi}(t,\psi_{\mathrm{as}}, \psi^*_{\mathrm{as}})$
and for $(B_{\phi})_{\mathrm{as}}$ inequality \eqref{5.1} is valid in the domain
$$
\tilde{S}(\phi,t_{\infty}, \kappa_0,\delta_2)=\{t \,|\, \re t>t_{\infty},\,
|\im t|<\kappa_0 \} \setminus \bigcup_{t_* \in Z_{\phi}} \{|t-t_*|<\delta_2\},
\quad 
Z_{\phi}= Z_{\phi}^{\infty} \cup Z_{\phi}^0,
$$
where $Z_{\phi}^{\infty}=\{ t_*\,|\, e^{i\phi}t_*+\tilde{\chi}_0 
=\pm \tfrac 13 \Omega_{\mathbf{b}}+\Omega_{\mathbf{a}}\mathbb{Z} 
+\Omega_{\mathbf{b}}\mathbb{Z} \},$ $Z_{\phi}^{0}=\{ t_*\,|\, 
e^{i\phi}t_*+\tilde{\chi}_0  \in \Omega_{\mathbf{a}}\mathbb{Z} 
+\Omega_{\mathbf{b}}\mathbb{Z} \}.$ Consider system \eqref{3.1} with $\mathcal
{B}(t,\lambda)$ containing $(\psi_{\mathrm{as}},\psi_{\mathrm{as}}^*)= 
(\psi_{\mathrm{as}}(\mathbf{s},t),\psi_{\mathrm{as}}^*(\mathbf{s},t))$. 
Then the direct monodromy problem for this system by the
WKB analysis results in the monodromy data $(\mathbf{s})_{\mathrm{as}}(t)$ 
such that $\|(\mathbf{s})_{\mathrm{as}}(t)-\mathbf{s} \| \le C|t|^{-\delta}$
for $|t|>t_{\infty}(\mathbf{s})$, in which $C$ and $\delta$ are some constant 
independent of $\mathbf{s}$.
Then the justification scheme of Kitaev \cite{Kitaev-1} applies to our case.
By the maximal modulus principle the excluded disc around each point in
$Z_{\phi}^{0}$ is removed. Thus we obtain Theorems \ref{thm2.1} and \ref{thm2.2}.
%%%%%%%%%%%%%%%%%%%%%%%%%%%%%%%%%%%%%
\subsection{Proof of Theorem \ref{thm2.4}}\label{ssc7.4}
%%%%%%%%%%%%%%%%%%%%%%%%%%%%%%%%
For given $n \in\mathbb{Z}$ the substitution
\begin{align*}
& \phi=\tfrac 12 \pi n+\tilde{\phi}, \quad x=e^{i\pi n/2}\tilde{x}, \quad
\xi=e^{i\pi n/4}\tilde{\xi}, \quad \mathrm{u}=e^{i\pi n/4}
\tilde{\mathrm{u}}_{(n)},
 \quad \mathrm{v}=e^{i\pi n/4}\tilde{\mathrm{v}}_{(n)},
\\ 
&\alpha=e^{i\pi n}\tilde{\alpha},\quad \beta= e^{i\pi n} \tilde{\beta}, \quad
\psi=e^{i\pi n/2}\tilde{\psi}, \quad y=e^{i\pi n/2}\tilde{y},
\end{align*}
where $(\tilde{\mathrm{u}}_{(n)}, \tilde{\mathrm{v}}_{(n)})=(\tilde{\mathrm{v}},
\tilde{\mathrm{u}})$ if $n$ is odd, and $=(\tilde{\mathrm{u}},
\tilde{\mathrm{v}})$ if $n$ is even, changes isomonodromy system \eqref{1.1} to
%%%%%%%%% (7.4) %%%%%%%%%%%%%%%%
\begin{align}\label{7.4}
\frac{d\tilde{\Psi}}{d\tilde{\xi}} =& \sigma_2^n \Biggl( \Bigl(\frac{\tilde{\xi}
^3}2 + \tilde{\xi}(\tilde{x} +\tilde{\mathrm{u}}_{(n)}\tilde{\mathrm{v}}_{(n)})
+ \frac{\tilde{\alpha}}{\tilde{\xi}}\Bigr)\sigma_3 
\\
\notag
&\phantom{--}
+i \begin{pmatrix} 0 & \tilde{\xi}^2\tilde{\mathrm{u}}_{(n)} +2\tilde{x}
\tilde{\mathrm{u}}_{(n)}+ (\tilde{u}_{(n)})_{\tilde{x}}  \\
\tilde{\xi}^2 \tilde{\mathrm{v}}_{(n)} + 2\tilde{x}\tilde{\mathrm{v}}_{(n)}
- (\tilde{v}_{(n)})_{\tilde{x}} & 0 \end{pmatrix} \Biggr) \sigma_2^n\tilde{\Psi},
\\
\notag
 \tilde{\beta}=&(\tilde{\mathrm{u}}_{(n)})_{\tilde{x}}
\tilde{\mathrm{v}}_{(n)} - \tilde{\mathrm{u}}_{(n)}
(\tilde{\mathrm{v}}_{(n)})_{\tilde{x}} +2\tilde{x}\tilde{\mathrm{u}}_{(n)}
\tilde{\mathrm{v}}_{(n)} -(\tilde{\mathrm{u}}_{(n)}\tilde{\mathrm{v}}_{(n)})^2,
\quad 
\tilde{y}=\tilde{\mathrm{u}}_{(n)}\tilde{\mathrm{v}}_{(n)},
\end{align} 
%%%%%%%%%%%%%%%%%%%%%%%%%%%%%%%%%%%%
and $y=e^{-i\phi}x\psi$ to $\tilde{y}=e^{-i \tilde{\phi}}\tilde{x}\tilde{\psi}.$
For $k\in\mathbb{Z}$ system \eqref{7.4} admits the canonical solutions
$$
\sigma_2^n\tilde{\Psi}^{\infty}_k(\tilde{\xi})=\Psi_k^{\infty}(\tilde{\xi})
=(I+O(\tilde{\xi}^{-1}))\exp ((\tfrac 18 \tilde{\xi}^4+\tfrac 12 \tilde{x}
\tilde{\xi}^2 +(\tilde{\alpha}-\tilde{\beta}) \ln \tilde{\xi} )\sigma_3)
$$
as $\tilde{\xi}\to \infty$ through the sector $|\arg\tilde{\xi}+\tfrac{\pi}8
-\tfrac{\pi}4k|<\tfrac {\pi}4$, where $\Psi^{\infty}_k(\xi)$ are solutions of
\eqref{1.1} given by \eqref{2.1}. The Stokes matrices $\tilde{S}_k$ 
$(k\in \mathbb{Z})$ with
$$
\tilde{S}_{2l-1}=\begin{pmatrix}  1 & \tilde{s}_{2l-1} \\ 0 & 1 \end{pmatrix},
\quad
\tilde{S}_{2l}=\begin{pmatrix}  1 & 0 \\ \tilde{s}_{2l}  & 1 \end{pmatrix}
\quad (l \in \mathbb{Z})
$$
for system \eqref{7.4} are defined by $\Psi_{k +1}^{\infty}(\tilde{\xi})=\Psi_k
^{\infty}(\tilde{\xi})\tilde{S}_k.$ Observing that 
\begin{align*}
\Psi^{\infty}_k(\tilde{\xi})=&(I+O(\xi^{-1})) \exp((-1)^n(\tfrac 18\xi^4
+\tfrac 12 x\xi^2+(\alpha-\beta)\ln\xi)\sigma_3) e^{-(-1)^n(\alpha-\beta)
(i \pi n/4)\sigma_3}
\\
=&\sigma_2^n \Bigl((I+O(\xi^{-1}))\exp((\tfrac 18\xi^4+\tfrac 12 x\xi^2 +
(\alpha-\beta)\ln \xi)\sigma_3) e^{-(\alpha-\beta)(i\pi n/4) \sigma_3}\Bigr)
\sigma_2^n 
\end{align*}
in $|\arg \xi+\tfrac {\pi}8 -\tfrac{\pi}4(k+n) |<\tfrac{\pi}4$, we have
$\Psi_k^{\infty}(\tilde{\xi})=\sigma_2^n \Psi^{\infty}_{k+n}(\xi)e^{-(\alpha
-\beta)(i\pi n/4)\sigma_3}\sigma_2^n.$ This relation immediately leads to
$$
\tilde{S}_k= \sigma_2^n e^{(\alpha-\beta)(i\pi n/4)\sigma_3} S_{k+n}
 e^{-(\alpha-\beta)(i\pi n/4)\sigma_3} \sigma_2^n
$$ 
(cf. \cite[(13)]{Kapaev-3}), which implies $\tilde{s}_k\tilde{s}_{k+1}
=s_{k+n}s_{k+1+n}$. Thus system \eqref{1.1} for $0<|\phi-\pi n/2|
<\pi/4$ is converted to \eqref{7.4} for $0<|\tilde{\phi}|<\pi/4$ with 
the monodromy data 
$\tilde{\mathbf{s}}=(\tilde{s}_1,\tilde{s}_2,\tilde{s}_3,\tilde{s}_4)=
\mathbf{s}_n=(s_{1+n},s_{2+n},s_{3+n},s_{4+n})$. Then application of
Corollary \ref{cor2.3} yields the theorem. 
%%%%%%%%%%%%%%%%%%%%%%%%%%%%%%%%%%%%%%%%%%%%%%%%%%%
\subsection{Properties of the elliptic function $\mathrm{P}(u;A)$}
\label{ssc7.5}
%%%%%%%%%%%%%%%%%%%%%%%%
The elliptic function $p=\mathrm{P}(u;A)$ is a solution of 
%%%%%%%%% (7.5) %%%%%%%%%%%%%%%%
\begin{equation}\label{7.5}
4e^{3i\phi}A=\frac{(p')^2}{p}-p(p+2e^{i\phi})^2 \qquad (p'=dp/du).
\end{equation}
About this equation relations in \eqref{7.2} suggest the following interesting
fact.
%%%%%%%%% Proposition 7.4 %%%%%%%%%%%%
\begin{prop}\label{prop7.4}
Let $p=\eta=\eta(u)$ be a given solution of \eqref{7.5}. Then the 
functions $\chi_{\pm}= \mp \tfrac 12 \eta' \eta^{-1}-\tfrac 12\eta-e^{i\phi}$
also solves \eqref{7.5}.
\end{prop}
%%%%%%%%%%%%%%%%%%%%%%
\begin{proof}
Note that
$$
(\chi_{\pm}+\tfrac 12 \eta +e^{i\phi})^2 =\tfrac 14 (\eta')^2\eta^{-2}
=\tfrac 14\eta^{-1} (4e^{3i\phi}A +\eta(\eta+2e^{i\phi})^2)
=e^{3i\phi}A\eta^{-1}+\tfrac 14 (\eta+2e^{i\phi})^2,
$$ 
which implies that $\chi_{\pm}$ and $\eta$ satisfy the relation
%%%%%%%%% (7.6) %%%%%%%%%%%%
\begin{equation}\label{7.6}
\eta^2\chi_{\pm}+\eta \chi_{\pm}^2 +2e^{i\phi}\eta\chi_{\pm} -e^{3i\phi}A=0.
\end{equation}
Since $\eta$ solves \eqref{7.5}, using \eqref{7.6} we have
\begin{align*}
\mp\chi'_{\pm}&=\tfrac 12(\eta''\eta^{-1}-(\eta')^2\eta^{-2})\pm \tfrac 12\eta'
\\
&= \tfrac 12 \eta^{-1} (\eta^2(\eta+2e^{i\phi})-2e^{3i\phi}A)\pm \tfrac 12 \eta'
\\
&= \eta(\pm \tfrac 12 \eta'\eta^{-1}+\tfrac 12 \eta+e^{i\phi})
-\chi_{\pm}^2 -(\eta +2e^{i\phi})\chi_{\pm}
\\
&=-2\eta \chi_{\pm} -\chi_{\pm}(\chi_{\pm}+2e^{i\phi}),
\end{align*}
and hence $\eta=\pm \tfrac 12 \chi'_{\pm}\chi_{\pm}^{-1}-\tfrac 12\chi_{\pm}
-e^{i\phi}$. Then by \eqref{7.6} 
\begin{align*}
(\chi'_{\pm})^2\chi_{\pm}^{-1} =& \chi_{\pm}(2 \eta +(\chi_{\pm}+2e^{i\phi}))^2
\\
=& \chi_{\pm}(\chi_{\pm}+2e^{i\phi})^2 +4(\eta^2 \chi_{\pm}+\eta\chi_{\pm}
(\chi_{\pm}+2e^{i\phi}))
\\
=& \chi_{\pm}(\chi_{\pm}+2e^{i\phi})^2 +4e^{3i\phi}A,
\end{align*}
which implies $\chi_{\pm}$ solves \eqref{7.5}. 
\end{proof}
Recall that $\mathrm{P}(u;A)$ is a solution of \eqref{7.5} such that 
$\mathrm{P}(0;A)=0.$
Using Proposition \ref{prop7.4} we have the following.
%%%%%%%%%%%%%%%%%%%%%%%%%%%%%%%%%%%%%%%%%%%%
%%%%%%%% Proposition 7.5 %%%%%%%%%%%%%
\begin{prop}\label{prop7.5}
The elliptic function $\mathrm{P}(u;A)$ has simple poles with residue $-1$ at 
$u=\tfrac 13\Omega_{\mathbf{b}} +\Omega_{\mathbf{a}}\mathbb{Z}
 +\Omega_{\mathbf{b}}\mathbb{Z}$
and simple poles with residue $1$ at $u= - \tfrac 13\Omega_{\mathbf{b}}
 +\Omega_{\mathbf{a}}\mathbb{Z} +\Omega_{\mathbf{b}}\mathbb{Z}$.
\end{prop}
%%%%%%%%%%%%%%%%%%%%%%%%%%%%%%
%%%%%%%%%%%%%%%%%%%%%%%%%%%%%%%%%%%%%%%%%
%%%%%%%%% Figure 7.1 %%%%%%%%%%%%%%%%%%
%%%%%%%%%%%%%%%%%%%%%%%%%%%%%%%%%%%
{\small
\begin{figure}[htb]
\begin{center}
\unitlength=0.80mm
%%%%%%%%%%%%%%%%%%%%%%%%%%%%%%%%%%
\begin{picture}(80,60)(-40,-28)

 \put(31,-16){\makebox{$\mathbf{l}_{-1}$}}
 \put(15,22){\makebox{$\mathbf{l}_{1}$}}
 \put(-10,-23){\makebox{$\mathbf{l}_{-2}$}}
 \put(-23,11){\makebox{$\mathbf{l}_{2}$}}

 \put(39,-3){\makebox{$0$}}
 \put(8, -3){\makebox{$z_1$}}
 \put(-11, -3){\makebox{$z_3$}}
 \put(-43, -3){\makebox{$z_5$}}

 \put(-1,24){\makebox{$\infty^+$}}
 \put(4,-27){\makebox{$\infty^-$}}

 \put(-35,18){\makebox{$\Pi_+$}}
 \put(-35,-19){\makebox{$\Pi_-$}}

 \put(5,20){\circle*{1}}

 \put(5,-20){\circle*{1}}
 \put(-35,0){\circle*{1.5}}
 \put(35,0){\circle*{1.5}}

 \put(-15,0){\circle*{1.5}}
 \put(15,0){\circle*{1.5}}

 \qbezier (-11,18) (-16,15.5) (-18,12)
 \put(-11,18){\vector(4,1){0}}

 \qbezier (-11,-18) (-15,-16.5) (-18,-12)
 \put(-18,-12){\vector(-1,3){0}}

 \qbezier (20,20.5) (24,19.5) (28,16)
 \put(28,16){\vector(2,-3){0}}

 \qbezier (20,-20.5) (24,-19.5) (28,-16)
 \put(20,-20.5){\vector(-1,0){0}}

\thinlines
 \qbezier (-35,0.6) (-20,0.6) (-15,0.6)
 \qbezier (-35,-0.6) (-20,-0.6) (-15,-0.6)

 \qbezier (35,0.6) (20,0.6) (15,0.6)
 \qbezier (35,-0.6) (20,-0.6) (15,-0.6)

\thicklines
 \qbezier (5,20) (30,15) (35,0)
 \qbezier [30] (5,-20) (30,-15) (35,0)
 \qbezier [30] (5.2,-20.2) (30.2,-15.2) (35.2,0.2)

 \qbezier (5,20) (-15,15) (-20,0.3)
 \qbezier [25](5,-20.2) (-15,-15.2) (-20,-0.7)
 \qbezier [25](5.2,-20) (-14.8,-15) (-19.8,-0.5)
\end{picture}
%%%%%%%%%%%%%%%%%%%%%%%%%%%%%%%%%%
%%%%%%%%%%%%%%%%%%%%%%%%%%%%%%%%%%
%%%%%%%%%%%%%%%%%%%%%%%%%%%%%%%%%
\end{center}
\caption{Cycle $\mathbf{b}=\mathbf{l}_{-1}\cup\mathbf{l}_{-2}
\cup \mathbf{l}_2 \cup\mathbf{l}_1$}
\label{cycle-b}
\end{figure}
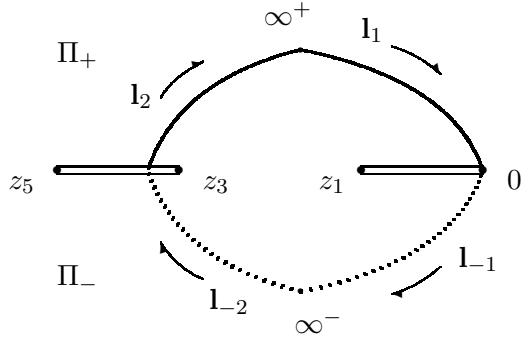
}
%%%%%%%%%%%%%%%%%%%%%%%%%%%%%%
%%%%%%%%%%%%%%%%%%%%%%%%%%%%%%
%%%%%%%%%%%%%%%%%%%%%%%%%%%%%%%%%%%%
\begin{proof}
Consider the loop on $\Pi_{A,\phi}=\Pi_+\cup\Pi_-$ defined by 
$\mathbf{l}_{-1} \cup \mathbf{l}_{-2} \cup \mathbf{l}_2 \cup \mathbf{l}_1$
homotopic to the cycle $\mathbf{b}$, in which $\mathbf{l}_{-1}=[0,\infty^-],$
$\mathbf{l}_{-2}=[\infty^-, z_3] \subset \Pi_-$, and 
$\mathbf{l}_{2}=[z_3,\infty^+],$ $\mathbf{l}_{1}=[\infty^+, 0] \subset \Pi_+$  
as in Figure \ref{cycle-b}. Let
$$
 \gamma_-:=\int_{\mathbf{l}_{-1}}\frac{dz}{w(A,z)}, \quad 
\gamma_+:=\Omega_{\mathbf{b}}-\gamma_-=\int_{\mathbf{l}_{-1} \cup
 \mathbf{l}_{-2}\cup \mathbf{l}_2}\frac{dz}{w(A,z)} = \Bigl(\int_{\mathbf{b}}
 -\int_{\mathbf{l}_1}\Bigr) \frac{dz}{w(A,z)},  
$$
with $\int_{\mathbf{l}_1} w(A,z)^{-1}dz=\int_{\mathbf{l}_{-1}} w(A,z)^{-1}dz.$
Then $\mathrm{P}(u)=\mathrm{P}(u;A)$ has simple poles at 
$u= u_{\pm} \equiv \gamma_{\pm} 
\mod \Omega_{\mathbf{a}}\mathbb{Z}+\Omega_{\mathbf{b}}\mathbb{Z}$, 
and double zeros at $u=u_0\equiv 0,$ and around these poles and zeros, 
$\mathrm{P}(u)=\pm (u-u_{\pm})^{-1}- e^{i\phi}+o(1)$ and 
$\mathrm{P}(u)=e^{3i\phi}A(u-u_0)^2(1+o(1)).$ 
Let
$$
\mathrm{P}_{-}(u)= \tfrac 12 \mathrm{P}'(u)\mathrm{P}(u)^{-1}
-\tfrac 12 \mathrm{P}(u)-e^{i\phi}.
$$
It is easy to see that $\mathrm{P}_-(u)$ has simple poles with residue $1$ 
at $u= u_0\equiv 0 $, with residue $-1$ at $u=u_+\equiv \gamma_+ $, 
and at least vanishes at $u=u_- \equiv \gamma_-$.
By Proposition \ref{prop7.4}, $\mathrm{P}_{-}(u)$ also solves \eqref{7.5}, 
and hence
$\mathrm{P}_-(u+\gamma_-)=\mathrm{P}(u)$. Then the string of poles and 
zeros of $\mathrm{P}(u)$ given by
$ \mathcal{S}:=\Omega_{\mathbf{b}}\mathbb{Z} \cup
 (\gamma_- +\Omega_{\mathbf{b}}\mathbb{Z})
\cup (\gamma_+ +\Omega_{\mathbf{b}}\mathbb{Z})$
coincides with the shifted one $\mathcal{S}+\gamma_-$.    
From this fact we derive $2\gamma_-=\gamma_+=\Omega
_{\mathbf{b}}-\gamma_-,$ which implies $\gamma_-=\tfrac 13 
\Omega_{\mathbf{b}}.$ Thus the proposition is obtained.  
\end{proof}
%%%%%%%%%%%%%%%%%%%%%%%%%%%%%%%%%%%%%%%%%%%%
%%%%%%%%%%%%%%%%%%%%%%%%%%%%%%%%%%%%%%
%%%%%%% Section 8 %%%%%%%%%%%%%%%%%%%
%%%%%%%%%%%%%%%%%%%%%%%%%%%%%%%%%%%%%%
\section{Boutroux equations}\label{sc8}
We show basic facts on the Boutroux equations used in the 
preceding sections.
%%%%%%%%%%%%%%%%%%%%%%%%%%%%%%%%%%%%%%%%%%%%%%%%%%%%
\subsection{Basic facts}\label{ssc8.1}
%%%%%%%%%%%%%%%%%%%%%%%%%%%%%%%%%%%%%%%%%%
The definitions of the cycles $\mathbf{a},$ $\mathbf{b}$, 
$\mathbf{a}_*,$ $\mathbf{b}_*$ are based on the following stable configuration
of zeros.
%%%%% Lemma 8.1 %%%%%%%%%%%
\begin{lem}\label{lem8.1}
As long as $A\in \mathcal{D}_0:=\mathbb{C}\setminus 
(\{c\le 0\} \cup\{c\ge \frac 8{27}\})$,
the polynomial $\zeta^3+4\zeta^2+4\zeta+4 A$ has zeros $\zeta_1,$ $\zeta_3$, 
$\zeta_5$ with the properties$:$
\par
$(1)$ $\zeta_j=\zeta_j(A)$ $(j=1,3,5)$ are continuous in $A;$
\par
$(2)$ $\re \zeta_5 < \re \zeta_3 < \re \zeta_1 ;$
\par 
$(3)$ $(\zeta_5, \zeta_3,\zeta_1) \to (-2,-2 , 0) $ as $A\to 0$, and 
$(\zeta_5, \zeta_3,\zeta_1) \to (-\frac 83 ,-\frac 23, -\frac 23) $ 
as $A\to \frac 8{27}.$  
\end{lem}
%%%%%%%%%%%%%%%%%%%%%%%%%%%%%%%%%%%
\begin{proof}
%% Let us show that $\zeta_j$ $(j=1,3,5)$ may be chosen in such a way that 
%% (1), (2) and (3) are always valid under
%% the supposition $A \in \mathcal{D}_0$. 
Suppose that $A \in \mathcal{D}_0$. 
The substitution $\zeta=\frac 43 (\xi-1)$ leads to the polynomial 
$$
\varpi(\xi)=4\xi^3-3\xi -c_0, \quad c_0=1-\tfrac {27}4 A \in 
\mathcal{D}_1
=\mathbb{C}\setminus(\{c\le -1\}\cup \{c\ge 1\}), 
$$
which has multiple zeros if and only if $c_0=\pm 1.$
The zeros of $\varpi(\xi)$ may be numbered in such a way that 
$-1< \xi_5<\xi_3< \xi_1<1$ for $-1<c_0<1$, and clearly
$\xi_j$ $(j=1,3,5)$ move continuously in $c_0 \in \mathcal{D}_1.$ 
Then to verify (2) for $\zeta_j= \frac 43 (\xi_j-1)$ $(j=1,3,5)$, 
it is sufficient to show that $\xi_j-\xi_k \not\in i\mathbb{R}$ 
as long as $c_0\in \mathcal{D}_1$ if $j\not=k$. By Cardano's formula $\xi_j,$
$\xi_k \in \{ \gamma_0,\gamma_{\pm 1}\},$ in which $2\gamma_0=u^{1/3}+u^{-1/3},$ 
$2\gamma_{\pm 1}=\omega^{\pm 1} u^{1/3}+\omega^{\mp 1} u^{-1/3}$ with
$\omega=e^{2\pi i/3}$ and $u=c_0+ \sqrt{c_0^2-1}$. 
Suppose that $2(\gamma_0-\gamma_1)=(1-\omega)u^{1/3}-(1-\omega^{-1})u^{-1/3}=ir$
with $r\in \mathbb{R}.$ Then $v=(1-\omega)u^{1/3}$ fulfils $v^2-irv+3=0$, which
implies $(1-\omega)u^{1/3}\in i \mathbb{R}$ with $1-\omega=i\sqrt{3}\omega^2.$
Hence $u=c_0+\sqrt{c_0^2-1}=r_0 \in\mathbb{R}$, that is, $c_0=(r_0+r_0^{-1})/2
\in \{c\le -1\}\cup\{c\ge 1\}$, contradicting $c_0 \in \mathcal{D}_1.$
Supposition $\gamma_1-\gamma_{-1} \in i\mathbb{R}$ implies $u^{1/3}\in 
\mathbb{R}$ leading to the same contradiction.
Thus the property (2) is verified. It is easy to see that (1) and (3) are
fulfilled.  
\end{proof}
%%%%%%%%%%%%%%%%%%%%%%%%%%%%%%%%%%
The following corollary is obtained by setting $\zeta=e^{-i\phi}z.$
%%%%% Corollary 8.2 %%%%%%%%%%%
\begin{cor}\label{cor8.2}
As long as $(\phi,A)\in \mathbb{R} \times \mathcal{D}_0$, 
the polynomial $z^3+4e^{i\phi}z^2+4e^{2i\phi}z
+4 e^{3i\phi}A$ has zeros $z_1,$ $z_3$, $z_5$ with the properties$:$
\par
$(1)$ $z_j=z_j(\phi,A)$ $(j=1,3,5)$ are continuous in $(\phi, A);$
\par
$(2)$ $\re e^{-i\phi}z_5 < \re e^{-i\phi}z_3 < \re e^{-i\phi}z_1 ;$
\par 
$(3)$ $(z_5, z_3,z_1) \to  (-2 e^{i\phi},-2 e^{i\phi}, 0)$ as $A\to 0$, and 
$(z_5, z_3,z_1) \to (-\frac 83 e^{i\phi},-\frac 23 e^{i\phi},
 -\frac 23 e^{i\phi})$ as $A\to \frac 8{27}.$  
\end{cor}
%%%%%%%%%%%%%%%%%%%%%%%%%%%%%%%%%%%
\par
Let us define the elliptic curve $\Pi_{A}^*=\Pi_+^*\cup \Pi_-^*$ and the
cycles $\mathbf{a}_*$ and $\mathbf{b}_*$ on $\Pi^*_{A}$. As long as  
$A\in \mathcal{D}_0$, by Lemma \ref{lem8.1}, the elliptic curve
$$
v(A,\zeta)^2=\zeta^4+4\zeta^3+4\zeta^2+4A\zeta=\zeta(\zeta-\zeta_1)
(\zeta-\zeta_3)(\zeta-\zeta_5),
$$
is the two sheeted Riemann surface $\Pi^*_{A}=\Pi^*_+\cup \Pi^*_-$ glued
along the cuts $[\zeta_5,\zeta_3],$ $[\zeta_1, 0]$, where $\Pi^{*}_{\pm}$
are two copies of $P^1(\mathbb{C})\setminus ([\zeta_5,\zeta_3]\cup [\zeta_1,0])$.
Then on $\Pi^*_{A}$ the cycles $\mathbf{a}_*$ and $\mathbf{b}_*$ are defined
as in Figure \ref{cycles2}.
Note that $\Pi^*_{A}$ and $(\mathbf{a}_*, \mathbf{b}_*)$ are parametrised
by $A\in \mathcal{D}_0$. Let the elliptic curve $\Pi_{A,\phi}=\Pi_+\cup \Pi_-$ 
and the pair of cycles $(\mathbf{a}, \mathbf{b})$ be the images of 
$\Pi^*_{A}$ and $(\mathbf{a}_*, \mathbf{b}_*)$ under the map $z=e^{i\phi}
\zeta$. Then $\Pi_{A,\phi}$ is given by $w(A,z)^2=z^4+4e^{i\phi}z^3+4e^{2i\phi}
z^2 +4e^{3i\phi} Az$ as in Section \ref{sc2}, and the cycles $\mathbf{a}$ 
and $\mathbf{b}$ are as in Figure \ref{cycles1} for $0<|\phi|<\pi/4$. 
%%%%%%%%%%%%%%%%%%%%%%%%%%%%%%%%%%%%%%%%%%%%%%%%%%%%%%%%%%%%%%%%%%%%%%%%
%%%%%%%%%%%%%%%%%%%% Figure 8.1 %%%%%%%%%%%%%%%%%%%%%%%%%%%%%%%
{\small
\begin{figure}[htb]
\begin{center}
\unitlength=0.75mm
%%%%%%%%%%%%%%%%%%%%%%%%%%%%%%%%%%
%%%%%%%%%%%%%%%%%%%%%%%%%%%%%%%%%%
%%%%%%%%%%%%%%%%%%
\begin{picture}(70,45)(-40,-20)
 \put(-57.5,-15.2){\makebox{$\zeta_5$}}
 \put(-15.5,-12.2){\makebox{$\zeta_3$}}
 \put(23.5,-9.2){\makebox{$\zeta_1$}}
 \put(50,3.2){\makebox{$0$}}

 \put(-52,7){\makebox{$\mathbf{a}_*$}}
 \put(-12,13){\makebox{$\mathbf{b}_*$}}

 \put(52,-13){\makebox{$\Pi_+^*$}}

\thinlines

 \qbezier(-45,6.2) (-37,9.8) (-31,10.3)
 \qbezier(-5,11.6) (3,13.8) (9,12.9)

 \put(-31,10.3){\vector(4,-1){0}}
 \put(9,12.9){\vector(3,-2){0}}
 \put(-50,-6.3){\line(5,1){35}}
 \put(-50,-7.7){\line(5,1){35}}
 \put(20,0.7){\line(3,1){30}}
 \put(20,-0.7){\line(3,1){30}}
\thicklines
 \put(-50,-7){\circle*{1.5}}
 \put(-15,0){\circle*{1.5}}
 \put(20,0){\circle*{1.5}}
 \put(50,10){\circle*{1.5}}
  \qbezier(-45,3) (-65,-9.5) (-45,-14.5)
  \qbezier(-20,-10.5) (0,1.5) (-20,7)
  \qbezier(-45,3) (-31, 9.5) (-20,7)
  \qbezier(-45,-14.5) (-34.5, -17) (-20,-10.5)

 \qbezier(-13, 6) (12.5,15.5) (25,1.8)
 \qbezier(-16, 4) (-19.5,2) (-19.2,-0.5)

 \qbezier[10](26.0,-0.3) (27.5,-3) (20,-7.5)
 \qbezier[25](-12.5,-7.3) (4,-14.8) (20,-7.5)
 \qbezier[7](-15.6,-5.4) (-19,-3) (-18.6,-2.4)

\end{picture}
\end{center}
\caption{Cycles $\mathbf{a}_*,$ $\mathbf{b}_*$ on $\Pi^*_{A}
=\Pi^*_+\cup\Pi^*_-$}
\label{cycles2}
\end{figure}
}
%%%%%%%%%%%%%%%%%%%%%%%%%%%%%%%%%%%%%%%%%%%%
%%%%%%%%%%%%%%%%%%%%%%%%%%%%%%%%%%%%%%%%%%%%%%%
\subsection{Uniqueness}\label{ssc8.2}
%%%%%%%%%%%%%%%%%%%%%%%%%%%%%%%%%%%%%%%
Let us treat the integrals
\begin{align*}
&J_{\mathbf{a}_*}(A)=\int_{\mathbf{a}_*}\frac{v(A,\zeta)}{\zeta}d\zeta, \qquad
J_{\mathbf{b}_*}(A)=\int_{\mathbf{b}_*}\frac{v(A,\zeta)}{\zeta}d\zeta,
\\
&v(A,\zeta)=\sqrt{\zeta^4 +4\zeta^3+ 4\zeta^2 +4A\zeta} 
=e^{-2i\phi}w(A,z).
\end{align*}
%%%%%%%%%%%%%%%%%%%%%%%%%%%%%%%%%%%%%%%%%%%%%%%%%%%%%%%%%%%%%%%%%%%%%%%%
Then the Boutroux equations \eqref{2.2} become \eqref{2.3}, that is,
\begin{equation*}
\tag*{(BE)$_{\phi}$}
\re e^{2i\phi} \int_{\mathbf{a}_*}\frac{v(A_{\phi},\zeta)}{\zeta} d\zeta
=\re e^{2i\phi} \int_{\mathbf{b}_*}\frac{v(A_{\phi},\zeta)}{\zeta} d\zeta=0.
\end{equation*}
%%%%%%%%%%%%%%% Example 8.1 %%%%%%%%%%%%%%%%%%
\begin{exa}\label{exa8.1}
Let $\phi=0,$ $A=\tfrac 8{27}.$ Then $\zeta_1=\zeta_3 =-\tfrac 23,$
$\zeta_5= -\tfrac 83,$ and hence
$$
J_{\mathbf{a}_*} (\tfrac 8{27}) 
=2 \int^{-2/3}_{-8/3} (\zeta+\tfrac 23) 
\sqrt{\zeta(\zeta+\smash{\tfrac 83)}}\frac{d
\zeta}{\zeta}=2e^{i\pi/2} \int^{2/3}_{8/3} (\tfrac 23 -t)
\sqrt{\frac{\smash{\tfrac 83}-t}t} dt=\frac{4i}{\sqrt{3}}
$$
and $J_{\mathbf{b}_*}(\tfrac 8{27})=0.$ This implies $A_0=
\tfrac 8{27}$ is a solution of the Boutroux equations (BE)$_{\phi=0}$. 
\end{exa}
In accordance with Kitaev \cite[Section 7]{Kitaev-2} we would like to show the
uniqueness of a solution of (BE)$_{\phi=0}$. To do so we begin with the
following.
%%%%%%%%%%%%%%%%%% Proposition 8.3 %%%%%%%%%%%%%%%%%%%%%%%
\begin{prop}\label{prop8.3}
If $\re J_{\mathbf{a}_*}(A)=\re J_{\mathbf{b}_*}(A)=0,$ then $A\in \mathbb{R}.$
\end{prop}
%%%%%%%%%%%%%%%%%%%%%%%%%%%%%%%%%%%%%%%%%%%%%%%%%%%%
\begin{proof} 
The condition $\re J_{\mathbf{a}_*}(A)=0$ implies
%%%%%%%%%%%%%%%%% (8.1) %%%%%%%%%%%%%%%%%%%
\begin{align} \label{8.1}
0= & J_{\mathbf{a}_*}(A)+\overline{J_{\mathbf{a}_*}(A)}
=  J_{\mathbf{a}_*}(A)+J_{\overline{\mathbf{a}_*}}(\overline{A})
=  J_{\mathbf{a}_*}(A)-J_{\mathbf{a}_*}(\overline{A})
= 4(A-\overline{A}) I_0, 
\\
\notag
I_0:=& \int_{\mathbf{a}_*} 
\frac{d\zeta}{v(A,\zeta)+v(\overline{A},\zeta)}.
\end{align}
Let us derive $A\in \mathbb{R}$ by supposing the contrary 
$A-\overline{A}\not=0.$ Then, by Lemma \ref{lem8.1}, it may be supposed that 
$\re \zeta_5 < \re\zeta_3 < \re \zeta_1 $, which is divided into two 
cases according to the location of $0.$
\par
{\bf (i)} Case $\re \zeta_3 \le 0$. 
It is sufficient to show that $ I_0 \not=0$ in \eqref{8.1}. 
The algebraic function $v(A,\zeta)+v(\overline{A},\zeta)$ is considered 
on the two sheeted Riemann surface $\Pi_A^{*0}$ glued along the cuts 
$[\zeta_1,0]$, $[\overline{\zeta_1},0]$,    
$[\zeta_5,\zeta_3]$, $ [\overline{\zeta_5}, \overline{\zeta_3}],$ which is
constructed by adding to $\Pi^*_A$ the new two cuts $[\overline{\zeta_1},0]$,    
$ [\overline{\zeta_5}, \overline{\zeta_3}]$ and gluing along them. 
Choose the cycle $\mathbf{a}_*$ on $\Pi_A^{*0}$ in such a way that 
$\mathbf{a}_*$ 
surrounds the cuts $[\zeta_5,\zeta_3]$ and $[\overline{\zeta_5},\overline
{\zeta_3}]$ as in Figure \ref{cycle-a}, (a), and 
modify $\mathbf{a}_*$ as in Figure \ref{cycle-a}, (b), where $\zeta_5=a+ib$,
$\zeta_3=c+ib'$, $a<c\le 0$, $b, b' \in \mathbb{R}.$
%%%%%%%%%%%%%%%%%%%%%%%%%%%%%%%%%%%%%%%%%%
%%%%%%%%%%%%%%%%%%%%%% Figure 8.2 %%%%%%%%%
%%%%%%%%%%%%%%%%%%%%%%%%%%%%%%%%%%%%%%%%%%%%%%
{\small
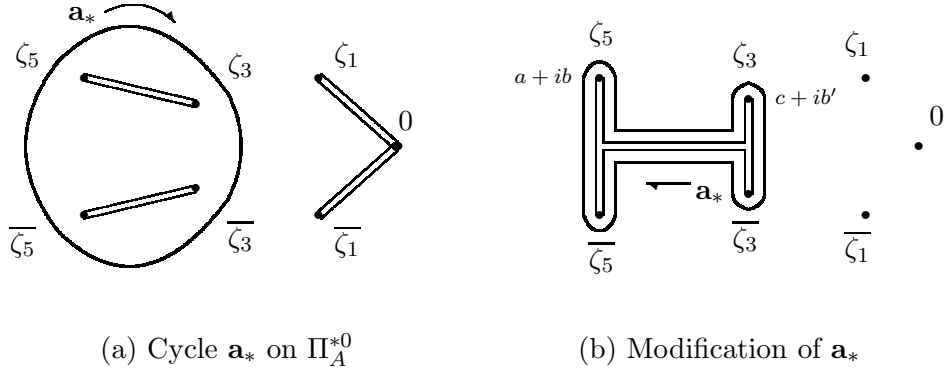
\begin{figure}[htb]
\begin{center}
\unitlength=0.70mm
%%%%%%%%%%%%%%%%%%%%%%%%%%%%%%%%%%
%%%%%%%%%%%%%%%%%%%%%%%%%%%%%%%
\begin{picture}(60,65)(-25,-35)

\put(-29,13){\circle*{1.5}}
\put(-29,-13){\circle*{1.5}}

\put(-8,8){\circle*{1.5}}
\put(-8,-8){\circle*{1.5}}

 \qbezier (-29,13.7) (-8,8.7) (-8,8.7)
 \qbezier (-29,12.4) (-8,7.4) (-8,7.4)

 \qbezier (-29,-13.7) (-8,-8.7) (-8,-8.7)
 \qbezier (-29,-12.4) (-8,-7.4) (-8,-7.4)

\put(15,13){\circle*{1.5}}
\put(15,-13){\circle*{1.5}}
\put(30,0){\circle*{1.5}}
 \qbezier (15.5,13.7) (15.5,13.8) (30,0.7)
 \qbezier (15,12.3) (15,12.3) (29.5,-0.7)

 \qbezier (15.5,-13.7) (15.5,-13.7) (30,-0.7)
 \qbezier (15,-12.3) (15,-12.3) (29.5,0.7)
 \put(-12,23.5){\vector (2,-3){0}}
 \qbezier (-25,25.5) (-18,29) (-12,23.5)

 \put(-42,16){\makebox{$\zeta_5$}}
 \put(-43,-20){\makebox{$\overline{\zeta_5}$}}
 \put(-2,14){\makebox{$\zeta_3$}}
 \put(-2,-19){\makebox{$\overline{\zeta_3}$}}
 \put(18,16){\makebox{$\zeta_1$}}
 \put(18,-20){\makebox{$\overline{\zeta_1}$}}
 \put(30,3){\makebox{$0$}}
 \put(-32,24){\makebox{$\mathbf{a}_*$}}

\thicklines
 \qbezier (-34,16) (-18,32) (-2,10)
 \qbezier (-34,-16) (-18,-32) (-2,-10)
 \qbezier (-34,-16) (-46,0) (-34,16)
 \qbezier (-2,-10) (3,0) (-2,10)

\put(-26,-40){\makebox{(a) Cycle $\mathbf{a}_*$ on $\Pi_A^{*0}$}}
\end{picture}
%%%%%%%%%%%%%%%%%%%%%%%%%%%%%%
%%%%%%%%%%%%%%%%%%
\qquad
%%%%%%%%%%%%%%%%%%%%%%%%%%%%%%%
\begin{picture}(60,65)(-30,-35)

 \put(-12,20){\makebox{$\zeta_5$}}
 \put(-12,-23){\makebox{$\overline{\zeta_5}$}}
 \put(52,4){\makebox{$0$}}

 \put(15,16){\makebox{$\zeta_3$}}
 \put(15,-19){\makebox{$\overline{\zeta_3}$}}
 \put(36,18){\makebox{${\zeta_1}$}}
 \put(36,-21){\makebox{$\overline{\zeta_1}$}}
 \put(-26,12){\makebox{\tiny$a+ib$}}
 \put(23,8){\makebox{\tiny$c+ib'$}}

 \put(8,-10){\makebox{$\mathbf{a}_*$}}

\qbezier (7,-7) (5,-7) (-1,-7) 
\put(-1,-7){\vector(-3,1){0}}

\put(18,9){\circle*{1.5}}
\put(18,-9){\circle*{1.5}}
\put(-10,13){\circle*{1.5}}
\put(-10,-13){\circle*{1.5}}

\put(40,13){\circle*{1.5}}
\put(40,-13){\circle*{1.5}}
\put(50,0){\circle*{1.5}}

\qbezier (-10.7,13) (-10.7, 0) (-10.7,-13) 
 \qbezier (-9.3,13) (-9.3, 0.7) (-9.3,0.7) 
 \qbezier (-9.3,-13) (-9.3,-0.7) (-9.3,-0.7) 

\qbezier (17.3,0.7) (-9.3,0.7) (-9.3,0.7) 
\qbezier (17.3,-0.7) (-9.3,-0.7) (-9.3,-0.7) 
\qbezier (17.3,0.7) (17.3,0.7) (17.3, 9) 
\qbezier (18.7,9) (18.7,0.7) (18.7,-9) 
\qbezier (17.3,-0.7) (17.3,-0.7) (17.3, -9) 

\thicklines

\qbezier (-13,13) (-12, 16) (-10,16) 
\qbezier (-7,13) (-8, 16) (-10,16) 

\qbezier (-13,-13) (-12,-16) (-10,-16) 
\qbezier (-7,-13) (-8,-16) (-10,-16) 

\qbezier (-13,13) (-13, 0) (-13,-13) 
\qbezier (-7,13) (-7, 3) (-7,3) 
\qbezier (-7,-13) (-7, -3) (-7,-3) 

\qbezier (21,-9) (21, 2) (21,9) 
\qbezier (15,9) (15, 8) (15,3) 
\qbezier (15,-9) (15, -8) (15,-3) 

\qbezier (15,3) (13,3) (-7,3) 
\qbezier (15,-3) (13,-3) (-7,-3) 
\qbezier (15,9) (15,11) (18,12) 
\qbezier (21,9) (21,11) (18,12) 
\qbezier (15,-9) (15,-11) (18,-12) 
\qbezier (21,-9) (21,-11) (18,-12) 

\put(-14,-40){\makebox{(b) Modification of $\mathbf{a}_*$}}
\end{picture}
%%%%%%%%%%%%%%%%%%%%%%%%%%%%%%
%%%%%%%%%%%%%%%%%%
%%%%%%%%%%%%%%%%%%%%%%%%%%%%%%%%%
\end{center}
\caption{Cycle $\mathbf{a}_*$ and modification in the case $\re \zeta_3 \le 0$} 
\label{cycle-a}
\end{figure}
}
%%%%%%%%%%%%%%%%%%%%%%%%%%%%%%%%%%%%%%%%%%%
%%%%%%%%%%%%%%%%%%%%%%%%%%%%%%%%%%%%%%%%%%%%%%
%%%%%%%%%%%%%%%%%%%%%%%%%%%%%%%%%%%%%%%%%%%%%
To simplify the description we write $v_{A}(\zeta)=v(A,\zeta),$
$v_{\overline{A}}(\zeta)=v(\overline{A},\zeta)$ and
$(v_{A}\pm v_{\overline{A}})(\zeta)=v(A,\zeta)\pm v(\overline{A},\zeta)$.
Let us take the contour of $I_0$ to be
the modified $\mathbf{a}_*$ starting from and terminating in the point
$\zeta=\re \zeta_5 +=a+$ on the upper shore of the cut $[a,c]$. 
Then $I_0$ is decomposed into the left- and right-vertical, and the horizontal 
parts:
\begin{align*}
 I_0=&I_{\mathrm{hor}} + I_{\mathrm{right\text{-}v}}+I_{\mathrm{left\text{-}v}},
\\
I_{\mathrm{hor}}=& 2 \int^c_{a} \frac{ds}{(v_A+v_{\overline{A}})(s)},
\\
 I_{\mathrm{right\text{-}v}}=&\int_0^{b'}
 \frac{idt}{(v_A+v_{\overline{A}})(c+it)}
+\int_{b'}^{-b'} \frac{idt}{(-v_A+v_{\overline{A}})(c+it)}
 +\int^{0}_{-b'} \frac{idt}{(-v_A-v_{\overline{A}})(c+it)},
\\
-I_{\mathrm{left\text{-}v}}=&\int_0^{-b} \frac{idt}{(v_A+v_{\overline{A}})(a+it)}
+\int_{-b}^{b} \frac{idt}{(v_A-v_{\overline{A}})(a+it)}
 +\int_{b}^{0} \frac{idt}{(-v_A-v_{\overline{A}})(a+it)}.
\end{align*}
Here
\begin{align*}
I_{\mathrm{right\text{-}v}}=&  \frac{i}{4(A-\overline{A})} \biggl(
\int_0^{b'}\frac{(v_A-v_{\overline{A}})(c+it)}{c+it}dt 
+\int_0^{b'}\frac{(v_A+v_{\overline{A}})(c+it)}{c+it}dt\biggr) 
\\
& - \frac{i}{4(A-\overline{A})} \biggl(
\int_0^{-b'}\frac{(v_A+v_{\overline{A}})(c+it)}{c+it}dt 
-\int_0^{-b'}\frac{(v_A-v_{\overline{A}})(c+it)}{c+it}dt\biggr) 
\\
=& \frac{i}{2(A-\overline{A})} \biggl(\int_0^{b'}\frac{v_A(c+it)}{c+it}dt
 +\int_0^{b'}\frac{v_{\overline{A}}(c-it)}{c-it}dt \biggr)\in  \mathbb{R}, 
\end{align*}
and similarly $I_{\mathrm{left\text{-}v}} \in \mathbb{R}.$
Let $A=\kappa +i\kappa'$ with $\kappa'\not=0.$ 
If, say $\im \zeta_5>0$, along the upper shore of the cut $[a,c]$,
$v_A(s)= e^{-\pi i/2}\sqrt{-\varphi(s)-4i\kappa' s}$,  
$v_{\overline{A}}(s)= -e^{\pi i/2} \sqrt{-\varphi(s)+4i\kappa' s}$, 
where
$\varphi(s)=s^4+4s^3+4s^2+4\kappa s <0$, i.e. $-\varphi(s)>0$ for 
$a < s < c \le 0.$  
Then the horizontal integral along the upper shore is 
$$
\int^c_{a}\frac{ds}{(v_A+v_{\overline{A}})(s)} = \frac i{\sqrt{2}}
 \int^c_{a}\frac{ds}{ \sqrt{-\varphi(s) +\sqrt{\varphi(s)^2
+16(\kappa')^2s^2}}}=i\gamma_0
$$ 
with some $\gamma_0>0$.  
Thus we have $I_{\mathrm{hor}} \in \mathbb{C}\setminus \mathbb{R},$ and hence
$I_0\not=0,$ which contradicts \eqref{8.1}. This implies $A\in \mathbb{R}$
in this case.
%%%%%%%%%%%%%%%%%%%%%%%%%%%%%%%%%%%%%%%%%%%%%%%%%%%%%%%%%%%
%%%%%%%%%%%%%%%%%%%%%%%%%%%%%%%%%%%%%%%%%%%%%%
\par
{\bf (ii)} Case $\re \zeta_3 >0$.  
Let $\Pi^{**}_A=\Pi^{**}_+\cup \Pi^{**}_{-}$ 
be a two sheeted Riemann surface glued
along the cuts $[\zeta_3,\zeta_1]$, $[\zeta_5,0]$ with
$\Pi^{**}_{\pm} =P^1(\mathbb{C})\setminus ([\zeta_3,\zeta_1]\cup[\zeta_5,0])$. 
Instead of $\Pi^*_A$ and $\mathbf{a}_*$ of the case (i) we treat 
$\Pi^{**}_A$ and a cycle on it.
Draw a closed curve $\mathbf{b}_0$ on the upper sheet 
$\Pi^{**}_+\subset \Pi^{**}_A$ surrounding the cut $[\zeta_3,\zeta_1]$ 
clockwise as in Figure \ref{xx}.
%%%%%%%%%%%%%%%%%%%%%%%%%%%%%%%%%%%%%%%%%%%%%%%%%%%%%
%%%%%%%%%%%%%%%%%%%%%%%%%%%%%%%%%%%%%%%%%%%%%%%%%%%%%%%%%%%%%%%%%%%%
%%%%%%%%%%%%%%%% Figure 8.3 %%%%%%%%%%%%%%%%%%%
%%%%%%%%%%%%%%%%%%%%%%%%%%%%%%%%%%%%%%%%%%%%%%
%%%%%%%%%%%%%%%%%%%%%%%%%%%%
{\small
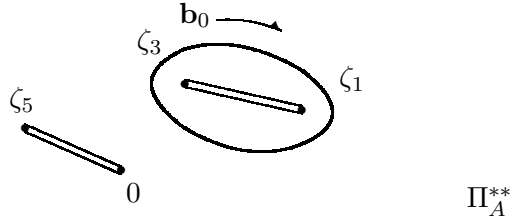
\begin{figure}[htb]
\begin{center}
\unitlength=0.70mm
%%%%%%%%%%%%%%%%%%%%%%%%%%%%%%%%%%
%%%%%%%%%%%%%%%%%%%%%%%%%%%%%%%
\begin{picture}(70,40)(-35,-10)

\put(-22,-3){\circle*{1.5}}
 \qbezier (-22,-2.3) (-22,-2.3) (-40,5.7)
 \qbezier (-22,-3.7) (-22,-3.7) (-40,4.3)

\put(12,8.5){\circle*{1.5}}

\put(-10,13.5){\circle*{1.5}}
\put(-40,5){\circle*{1.5}}

 \qbezier (12,9.2) (12,9.2) (-10,14.2)
 \qbezier (12,7.8) (12,7.8) (-10,12.8)

 \put(8,23.5){\vector (1,-1){0}}
 \qbezier (-4,25.5) (2,26) (8,23.5)

 \put(-43,9){\makebox{$\zeta_5$}}
 \put(-19,20){\makebox{$\zeta_3$}}
 \put(19,13){\makebox{$\zeta_1$}}
 \put(-21,-9){\makebox{$0$}}
 \put(-11,25){\makebox{$\mathbf{b}_0$}}
 \put(43,-10){\makebox{$\Pi_A^{**}$}}

\thicklines
 \qbezier (12,16) (0,23) (-10,20)
 \qbezier (12,2) (2,-2) (-10,5)
 \qbezier (12,2) (23.5,7.0) (12,16)
 \qbezier (-10,20) (-22.5,14.4) (-10,5)

\end{picture}
%%%%%%%%%%%%%%%%%%%%%%%%%%%%%%
%%%%%%%%%%%%%%%%%%%%%%%%%%%%%%%%%
\end{center}
\caption{Cycle $\mathbf{b}_0$ on $\Pi_A^{**}$ in the case $\re \zeta_3>0$} 
\label{xx}
\end{figure}
}
%%%%%%%%%%%%%%%%%%%%%%%%%%%%%%%%%%%%%%%%%
Consider the algebraic function 
$v^*(A,\zeta)=  \sqrt{\zeta^4+4\zeta^3+4\zeta^2+4A\zeta}$
on $\Pi^{**}_A$, where the branch is chosen in such a way that $v^*(A,\zeta)$
coincides with $v(A,\zeta)$ along the upper shore of the cut 
$[\zeta_3,\zeta_1]$. 
The cycle $\mathbf{b}_0$ on $\Pi^{**}_A$ is a substitute for $\mathbf{b}_*$
on $\Pi_A^*.$ Indeed
$$
J_{\mathbf{b}_0}^*(A):=
 \int_{\mathbf{b}_0} \frac{v^*(A,\zeta)}{\zeta}d\zeta  
=2\int^{\zeta_1}_{\zeta_3} \frac{v(A,\zeta)}{\zeta}d\zeta 
= J_{\mathbf{b}_*}(A).
$$
Hence $\re J_{\mathbf{b}_*}(A)=0$ is equivalent to 
$\re J^*_{\mathbf{b}_0}(A)=0.$ Under this supposition,
\begin{align*}
0= & J^*_{\mathbf{b}_0}(A)+\overline{J^*_{\mathbf{b}_0}(A)}
=  J^*_{\mathbf{b}_0}(A)+J^*_{\overline{\mathbf{b}_0}}(\overline{A})
=  J^*_{\mathbf{b}_0}(A)-J^*_{\mathbf{b}_0}(\overline{A})
= 4(A-\overline{A}) I_0^*, 
\\
I_0^*:=& \int_{\mathbf{b}_0} \frac{d\zeta}{v^*(A,\zeta)+v^*(\overline{A},\zeta)}
\end{align*}
instead of \eqref{8.1}. As in the case (i),
to prove $A\in \mathbb{R}$ it is sufficient to show $ I_0^*  \not=0.$ 
The algebraic function $v^*(A,\zeta)+v^*(\overline{A},\zeta)$ is considered 
on the two sheeted Riemann surface $\Pi_A^{**0}$ glued along the cuts 
$[\zeta_3,\zeta_1]$, $[\overline{\zeta_3},\overline{\zeta_1}]$, 
$[\zeta_5,0]$, $[ \overline{\zeta_5},0]$, and the cycle $\mathbf{b}_0$ may be 
taken in such a way that $\mathbf{b}_0$ surrounds the cuts $[\zeta_3,\zeta_1]$ 
and $[\overline{\zeta_3},\overline{\zeta_1}]$ as in Figure \ref{x2}, (a).  
%%%%%%%%%%%%%%%%%%%%%%%%%%%%%%%%%%%%%%%%%%%%%%%%%%%%%
%%%%%%%%%%%%%%%%%%%%%% Figure 8.4 %%%%%%%%%%%%%%%%%%%%%%%%%
%%%%%%%%%%%%%%%%%%%%%%%%%%%%%%%%%%%%%%%%%%%%%%%%%%%%
%%%%%%%%%%%%%%%%%%%%%%%%%%%%
%%%%%%%%%%%%%%%%%%%%%%%%%%%%
{\small
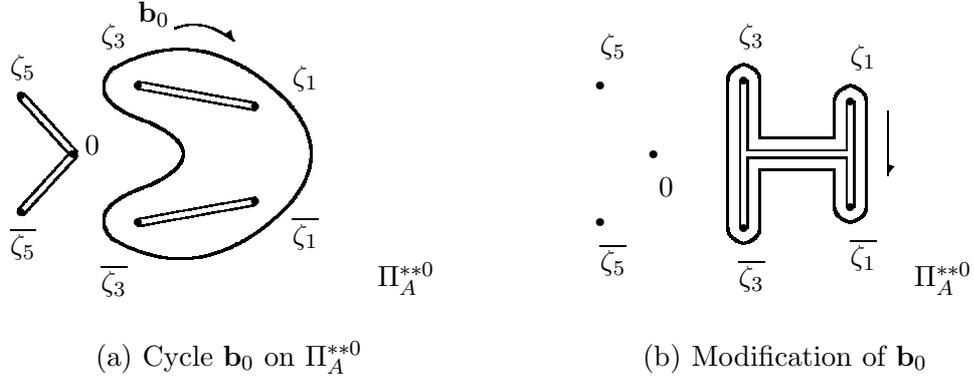
\begin{figure}[htb]
\begin{center}
\unitlength=0.70mm
%%%%%%%%%%%%%%%%%%%%%%%%%%%%%%%%%%
%%%%%%%%%%%%%%%%%%%%%%%%%%%%%%%
\begin{picture}(70,65)(-35,-35)

\put(-32,11){\circle*{1.5}}
\put(-32,-11){\circle*{1.5}}
 \qbezier (-22.5,-0.5) (-32.5,10.5) (-32.5,10.5)
 \qbezier (-21.5,0.5) (-31.5,11.5) (-31.5,11.5)
 \qbezier (-22.5,0.5) (-32.5,-10.5) (-32.5,-10.5)
 \qbezier (-21.5,-0.5) (-31.5,-11.5) (-31.5,-11.5)

\put(12,9){\circle*{1.5}}
\put(12,-9){\circle*{1.5}}

\put(-10,13){\circle*{1.5}}
\put(-10,-13){\circle*{1.5}}
\put(-22,0){\circle*{1.7}}

 \qbezier (12,9.7) (12,9.7) (-10,13.7)
 \qbezier (12,8.3) (12,8.3) (-10,12.3)
 \qbezier (12,-9.7) (12,-9.7) (-10,-13.7)
 \qbezier (12,-8.3) (12,-8.3) (-10,-12.3)

 \put(8,21.5){\vector (2,-3){0}}
 \qbezier (-3,23.8) (2,26) (8,21.5)

 \put(-34,15){\makebox{$\zeta_5$}}
 \put(-34,-19){\makebox{$\overline{\zeta_5}$}}
 \put(-17,21){\makebox{$\zeta_3$}}
 \put(-17,-26){\makebox{$\overline{\zeta_3}$}}
 \put(19,13){\makebox{$\zeta_1$}}
 \put(19,-17){\makebox{$\overline{\zeta_1}$}}
 \put(-20,0){\makebox{$0$}}
 \put(-10,25){\makebox{$\mathbf{b}_0$}}
 \put(35,-25){\makebox{$\Pi_A^{**0}$}}

\thicklines
 \qbezier (-15,16) (2,26) (18,10)
 \qbezier (-15,-16) (2,-26) (18,-10)
 \qbezier (18,-10) (27,0) (18,10)
   \qbezier (-8,-6) (5,0) (-8,6)
   \qbezier (-15,16) (-20,10) (-8,6)
   \qbezier (-15,-16) (-20,-10) (-8,-6)

\put(-18,-40){\makebox{(a) Cycle $\mathbf{b}_0$ on $\Pi^{**0}_A$}}
\end{picture}
%%%%%%%%%%%%%%%%%%%%%%%%%%%%%%
%%%%%%%%%%%%%%%%%%
\qquad\qquad
%%%%%%%%%%%%%%%%%%%%%%%%%%%%%%%
\begin{picture}(70,65)(-35,-35)

 \put(-6,-8){\makebox{$0$}}

 \put(-17,19){\makebox{$\zeta_5$}}
 \put(-17,-22){\makebox{$\overline{\zeta_5}$}}

 \put(9,21){\makebox{$\zeta_3$}}
 \put(9,-25){\makebox{$\overline{\zeta_3}$}}

 \put(30,17){\makebox{$\zeta_1$}}
 \put(30,-21){\makebox{$\overline{\zeta_1}$}}

 \qbezier (37,8) (37,0) (37,-4) 
 \put(37,-4){\vector(-1,-3){0}}

\put(-7,0){\circle*{1.5}}
\put(-17,13){\circle*{1.5}}
\put(-17,-13){\circle*{1.5}}

\put(10,14){\circle*{1.5}}
\put(10,-14){\circle*{1.5}}

\put(30,10){\circle*{1.5}}
\put(30,-10){\circle*{1.5}}

 \qbezier (9.3,14) (9.3, 0) (9.3,-14) 
 \qbezier (10.7,14) (10.7, 0.7) (10.7,0.7) 
 \qbezier (10.7,-14) (10.7, -0.7) (10.7,-0.7) 

 \qbezier (30.7,10) (30.7, 0) (30.7,-10) 
 \qbezier (29.3,10) (29.3, 0.7) (29.3,0.7) 
 \qbezier (29.3,-10) (29.3,-0.7) (29.3,-0.7) 
 
 \qbezier (10.7,-0.7) (11.7,-0.7) (29.3,-0.7) 
 \qbezier (10.7,0.7) (11.7,0.7) (29.3,0.7) 

 \put(42,-25){\makebox{$\Pi_A^{**0}$}}

\thicklines

 \qbezier (33,10) (33, 12) (30,13) 
 \qbezier (27,10) (27, 12) (30,13) 
 \qbezier (33,-10) (33,-12) (30,-13) 
 \qbezier (27,-10) (27,-12) (30,-13) 

 \qbezier (7,14) (7,16) (10,17) 
 \qbezier (13,14) (13,16) (10,17) 
 \qbezier (7,-14) (7,-16) (10,-17) 
 \qbezier (13,-14) (13,-16) (10,-17) 

 \qbezier (7,14) (7, 0) (7,-14) 
 \qbezier (13,14) (13, 3) (13,3) 
 \qbezier (13,-14) (13, -3) (13,-3) 

 \qbezier (33,10) (33, 0) (33,-10) 
 \qbezier (27,10) (27, 3) (27,3) 
 \qbezier (27,-10) (27, -3) (27,-3) 

 \qbezier (13,3) (13,3) (27,3) 
 \qbezier (13,-3) (13,-3) (27,-3) 

\put(-9,-40){\makebox{(b) Modification of $\mathbf{b}_0$}}
\end{picture}
%%%%%%%%%%%%%%%%%%%%%%%%%%%%%%
%%%%%%%%%%%%%%%%%%
%%%%%%%%%%%%%%%%%%%%%%%%%%%%%%%%%
\end{center}
\caption{Modification of $\mathbf{b}_0$ in the case $\mathrm{Re}\, \zeta_3 >0$} 
\label{x2}
\end{figure}
}
%%%%%%%%%%%%%%%%%%%%%%%%%%%%%%%%%%%%%%%%%%%%%
%%%%%%%%%%%%%%%%%%%%%%%%%%%%%%%%%%%%%%%%%%%
The cycle $\mathbf{b}_0$ is modified as in Figure \ref{x2}, (b), in which 
$\zeta_1=a+ib,$ $\zeta_3 =c+ib'$, $a>c>0,$ $b,b'\in \mathbb{R}.$ The contour of
the integral $I^*_0$ is taken to be the modified $\mathbf{b}_0$ that starts
from and terminates in the point $\zeta=\re\zeta_3 +=c+$ on the upper shore 
of the cut $[c,a].$ Dividing $I^*_0$ into the horizontal, 
and the left- and right-vertical parts, we may show $I^*_0 \not=0$ by the same 
arguments as in the case (i). Thus the proposition is proved.
\end{proof}
%%%%%%%%%%%%%%%%%%%%%%%%%%%%%%%%%%%%%%%%%%%%%%%%%%%%%%%%%%%%%%
%%%%%%%%%%%%%%%%%%%%%%%%%%%%%%%%%%%%%%%%%%%%%%%%%%%%%%%%%%%%%%%
As in Example \ref{exa8.1}, $\re J_{\mathbf{a}_*}(\tfrac{8}{27})
=\re J_{\mathbf{b}_*}(\tfrac 8{27})=0$.  
Let us examine $\re J_{\mathbf{a}_*}(A)$ or $\re J_{\mathbf{b}_*}(A)$ for 
$A \in \mathbb{R}\setminus \{\tfrac 8{27}\}.$ 
\par
Consider the polynomial $f(\zeta)=v(A,\zeta)^2=\zeta^4+4\zeta^3+4\zeta^2
+4A\zeta$ for $A\in \mathbb{R}$ on the real line $(-\infty,+\infty)\subset
\Pi_+^*.$ The zeros of $f(\zeta)$ are located as follows, in which (z.3) 
is given by Lemma \ref{lem8.1} and (z.4) is treated in Example \ref{exa8.1}: 
\par
{\bf (z.1)} if $A<0$, then $0 <\zeta_1,$ and 
$\zeta_5,\zeta_3 \not\in \mathbb{R};$
\par
{\bf (z.2)} if $A=0$, then $\zeta_5=\zeta_3 < \zeta_1=0;$  
\par
{\bf (z.3)} if $0<A<\tfrac 8{27}$, then $\zeta_5<\zeta_3 < \zeta_1<0;$  
\par
{\bf (z.4)} if $A=\tfrac 8{27}$, then $\zeta_5<\zeta_3 = \zeta_1 <0;$  
\par
{\bf (z.5)} if $A>\tfrac 8{27}$, then $\zeta_5 <0$, and  
$\zeta_3,\zeta_1 \not\in \mathbb{R}.$
%%%%%%%%%%%%%%%%%%%%%%%%%%%%%%%%%%%%%%%%%%%%%%%%%%%%%%%%%%%%
%%%%%%%%%%%%%%%%%%%%%%%%%%%%%%%%%%%%%%%%%%%%%%%%%%%%%%%%%%%%
%%%%%%%%%%%%%%%%% Figure 8.5 %%%%%%%%%%%%%%%%%%%%%
%%%%%%%%%%%%%%%%%%%%%%%%%%%%%%%%%%%%%%%%%%%%%
{\small
\begin{figure}[htb]
\begin{center}
\unitlength=0.77mm
%%%%%%%%%%%%%%%%%%%%%%%%%%%%%%%%%%
%%%%%%%%%%%%%%%%%%%%%%%%%%%%%%%%%%
%%%%%%%%%%%%%%%%%%
\begin{picture}(80,88)(-20,-50)

 \put(50,30){\makebox{{\bf (z.5a)} $\tfrac 8{27}<A<4$:}}
 \put(70,10){\circle{1.5}}
 \put(80,10){\circle*{1.5}}
 \put(90,10){\circle{1.5}}
 \put(92,5){\makebox{$0$}}
 \put(80,12){\makebox{$c$}}
 \put(67,5){\makebox{$\zeta_5$}}
 \put(80,20){\circle{1.5}}
 \put(80,0){\circle{1.5}}
 \put(73,22){\makebox{$\zeta_3$}}
 \put(73,-3){\makebox{$\zeta_1$}}
 \qbezier [16](79.5,0.5)(84.5,5.5) (89.5,10.5) 
 \qbezier [16](80.5,-0.5)(85.5,4.5) (90.5,9.5) 
 \qbezier [16](69.5,10.5)(74.5,15.5) (79.5,20.5) 
 \qbezier [16](70.5,9.5)(75.5,14.5) (80.5,19.5) 
 \qbezier [15](80,9.3)(80,5) (80,0.7) 

 \put(50,-20){\makebox{{\bf (z.5b)} $A \ge 4$:}}
 \put(87,-40){\makebox{$0$}}
 \put(102,-24){\makebox{$\zeta_3$}}
 \put(102,-48){\makebox{$\zeta_1$}}
 \put(70,-35){\circle{1.5}}
 \put(100,-35){\circle*{1.5}}
 \put(90,-35){\circle{1.5}}
 \put(100,-25){\circle{1.5}}
 \put(100,-45){\circle{1.5}}
 \put(100,-33){\makebox{$c$}}
 \put(67,-40){\makebox{$\zeta_5$}}
 \qbezier [15](100,-44.3)(100,-40) (100,-35.7) 
 \qbezier [16](99.5,-45.5)(94.5,-40.5) (89.5,-35.5) 
 \qbezier [16](100.5,-44.5)(95.5,-39.5) (90.5,-34.5) 
 \qbezier [40](99.9,-24.4)(84.9,-29.4) (69.9,-34.4) 
 \qbezier [40](100.1,-25.6)(85.1,-30.6) (70.1,-35.6) 

 \put(-75,30){\makebox{{\bf (z.1)} $A<0$:}}
 \put(10,30){\circle{1.5}}
 \put(20,30){\circle{1.5}}
 \qbezier(10,30.7) (10,30.7) (20,30.7)
 \qbezier(10,29.3) (10,29.3) (20,29.3)
 \put(10,25){\makebox{$0$}}
 \put(20,25){\makebox{$\zeta_1$}}
 \put(-15,34){\circle{1.5}}
 \put(-15,26){\circle{1.5}}
 \qbezier [13](-14.3,26) (-14.3,30) (-14.3,34)
 \qbezier[13](-15.7,26) (-15.7,30) (-15.7,34)
 \put(-21,24){\makebox{$\zeta_3$}}
 \put(-21,34){\makebox{$\zeta_5$}}

 \put(-75,15){\makebox{{\bf (z.2)} $A=0$:}}
 \put(10,15){\circle{1.5}}
 \put(-10,15){\circle{1.5}}
 \put(10,10){\makebox{$\zeta_{1,0}$}}
 \put(-10,10){\makebox{$\zeta_{5,3}$}}

 \put(-75,0){\makebox{{\bf (z.3)} $0<A<\tfrac 8{27}$:}}
 \put(10,0){\circle{1.5}}
 \put(0,0){\circle{1.5}}
 \put(-10,0){\circle{1.5}}
 \put(-20,0){\circle{1.5}}
 \qbezier(-20,0.7) (-20,0.7) (-10,0.7)
 \qbezier(-20,-0.7) (-20,-0.7) (-10,-0.7)
 \qbezier(0,0.7) (0,0.7) (10,0.7)
 \qbezier(0,-0.7) (0,-0.7) (10,-0.7)
 \put(10,-5){\makebox{$0$}}
 \put(0,-5){\makebox{$\zeta_{1}$}}
 \put(-10,-5){\makebox{$\zeta_{3}$}}
 \put(-20,-5){\makebox{$\zeta_{5}$}}

 \put(-75,-15){\makebox{{\bf (z.4)} $A=\tfrac 8{27}$:}}
 \put(10,-15){\circle{1.5}}
 \put(-5,-15){\circle{1.5}}
 \put(-20,-15){\circle{1.5}}
 \qbezier(-20,-15.7) (-20,-15.7) (-5,-15.7)
 \qbezier(-20,-14.3) (-20,-14.3) (-5,-14.3)
 \qbezier(10,-15.7) (10,-15.7) (-5,-15.7)
 \qbezier(10,-14.3) (10,-14.3) (-5,-14.3)
 \put(10,-20){\makebox{$0$}}
 \put(-7,-20){\makebox{$\zeta_{3,1}$}}
 \put(-22,-20){\makebox{$\zeta_{5}$}}

 \put(-75,-30){\makebox{{\bf (z.5)} $A>\tfrac 8{27}$:}}
 \put(-20,-30){\circle{1.5}}
 \put(10,-30){\circle{1.5}}
 \put(10,-35){\makebox{$0$}}
 \put(-20,-35){\makebox{$\zeta_5$}}
 \put(21,-34){\circle{1.5}}
 \put(21,-26){\circle{1.5}}
 \qbezier [13](20.3,-26) (20.3,-26) (20.3,-34)
 \qbezier [13](21.7,-26) (21.7,-26) (21.7,-34)
 \put(23,-26){\makebox{$\zeta_3$}}
 \put(23,-36){\makebox{$\zeta_1$}}

 \put(-50,-54){\makebox{$:$ $f(\zeta)>0$;}}
 \put(-10,-54){\makebox{$:$ $f(\zeta)<0$;}}
 \put(30,-54){\makebox{$:$ cut, $f(\zeta)<0$}}
 \qbezier (20,-53.6) (24,-53.6) (28,-53.6)
 \qbezier (20,-52.4) (24,-52.4) (28,-52.4)

\thicklines
 \qbezier[13] (70.5,10)(80.5,10) (89.5,10) 
 \qbezier[13] (70.4,10)(80.4,10) (89.4,10) 

  \qbezier(55,10) (55,10) (69.3,10)
  \qbezier(110,10) (110,10) (90.7,10)

 \qbezier[13] (70.5,-35)(80.5,-35) (89.5,-35) 
 \qbezier[13] (70.4,-35)(80.4,-35) (89.4,-35) 

  \qbezier(55,-35) (55,-35) (69.3,-35)
  \qbezier(110,-35) (110,-35) (90.7,-35)

  \qbezier(-30,30) (-30,30) (9.3,30)
  \qbezier(20.7,30) (30,30) (30,30)

  \qbezier(-30,15) (-30,15) (-10.7,15)
  \qbezier(-9.3,15) (-9.3,15) (9.3,15)
  \qbezier(10.7,15) (30,15) (30,15)

 \qbezier(10.7,0) (10.7,0) (30,0)
 \qbezier(-0.7,0) (-0.7,0) (-9.3,0)
 \qbezier(-30,0) (-30,0) (-20.7,0)

 \qbezier(10.7,-15) (10.7,-15) (30,-15)
 \qbezier(-30,-15) (-30,-15) (-20.7,-15)

 \qbezier(10.7,-30) (10.7,-30) (30,-30)
 \qbezier(-30,-30) (-30,-30) (-20.7,-30)
 \qbezier[20](-19.3,-30) (-4.3,-30) (9.3,-30)
 \qbezier[20](-19.4,-30) (-4.4,-30) (9.4,-30)

 \qbezier(-60,-53) (-54,-53) (-52,-53)
 \qbezier [5](-20,-53) (-16,-53) (-12,-53)
 \qbezier [5](-19.9,-53) (-15.9,-53) (-11.9,-53)
\end{picture}
\end{center}
\caption{$f(\zeta)$ on $(-\infty,+\infty)$}
\label{zz}
\end{figure}
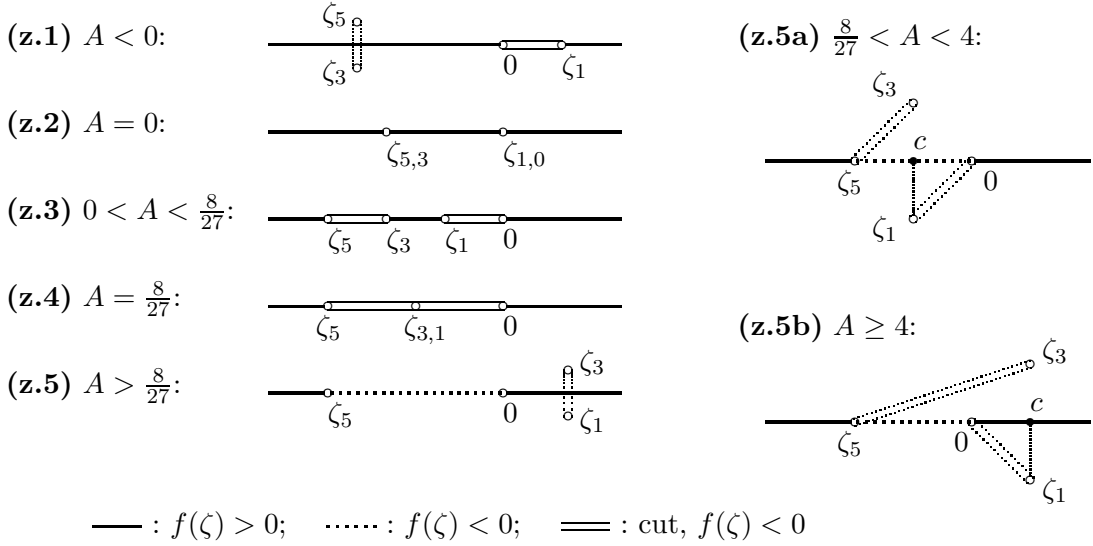
}
%%%%%%%%%%%%%%%%%%%%%%%%%%%%%%%%%%%%%%%%%%%
%%%%%%%%%%%%%%%%%%%%%%%%%%%%%%%%%%%%%%%%%%%
\par\noindent
Let us show that $\re J_{\mathbf{a}^*}(A)\not= 0$ or
$\re J_{\mathbf{b}^*}(A)\not= 0$ for $A\in \mathbb{R}\setminus
\{\tfrac 8{27}\}.$ 
In each case, $f(\zeta)$ behaves as in Figure \ref{zz}: say, in the case
(z.3), $f(\zeta)>0$ on $(-\infty,\zeta_5) \cup (\zeta_3,\zeta_1) 
\cup (0,+\infty)$, and
$f(\zeta)<0$ on $(\zeta_5,\zeta_3) \cup (\zeta_1,0)$.
It is easy to see that, in (z.2) or (z.3),
$$
J_{\mathbf{b}_*}(A)=2\int_{\zeta_3}^{\zeta_1} \frac{v(A,\zeta)}{\zeta}d\zeta
=2\int_{\zeta_3}^{\zeta_1} \frac{-\sqrt{f(\zeta)}}{\zeta}d\zeta <0,
$$
and that, in (z.2), $J_{\mathbf{a}_*}(0)=0.$
In (z.1), we write $\zeta_5=c +ib,$ $\zeta_3=c -ib$ with $c <0,$
$b >0$ such that $2c + \zeta_1 =-4.$ Then 
$$
J_{\mathbf{b}_*}(A)=2\int_{\zeta_3}^{\zeta_1} \frac{v(A,\zeta)}{\zeta}d\zeta
=2\int_{c}^{0} \frac{-\sqrt{f(\zeta)}}{\zeta}d\zeta <0.
$$
It remains to discuss the case (z.5). Set $\zeta_1=c-ib$ and $\zeta_3=c+ib$ 
with $b>0.$ The relation $2c+\zeta_5=-4$ yields $A=2(c+2)(c+1)^2$ and $b=
\sqrt{(3c+2)(c+2)}.$ 
Then, according to the location of $c$ we consider two cases:
{\bf (z.5a)} $-\tfrac 23 <c<0$ if $\tfrac 8{27} <A<4$; and 
{\bf (z.5b)} $c\ge 0$ if $A \ge 4$.  
First consider the case (z.5a) as in Figure \ref{zz}. 
Note that
$$
\frac{\partial}{\partial A} \int^{\zeta_1}_0 \frac{v(A,\zeta)}{\zeta}d\zeta
=2 \int^{\zeta_1}_0 \frac{d\zeta}{v(A,\zeta)} +(\zeta_1)_A
 \frac{v(A,\zeta_1)}{\zeta_1}  
=2 \int^{\zeta_1}_0 \frac{d\zeta}{v(A,\zeta)}, 
$$
since $v(A,\zeta_1)=0.$ Then, by $f(\zeta)<0$ on $(c,0)$,
$$
\frac{\partial}{\partial A} \re J_{\mathbf{a}_*}(A)=
2\frac{\partial}{\partial A} \re \int^{\zeta_1}_0 \frac{v(A,\zeta)}{\zeta}d\zeta
=4 \re \int^{\zeta_1}_0 \frac{d\zeta}{v(A,\zeta)}
=4 \re \int^{\zeta_1}_c \frac{d\zeta}{v(A,\zeta)},  
$$
in which the contour is taken along the upper shore of the cut $[0,\zeta_1].$
Using  
\begin{align*}
&f(c+it)=g(t)+ih(t),
\\
& g(t)=t^4-2(3c^2+6c+2)t^2+f(c),
\\
& h(t)=-4(c+1)t^3 +4(c+1)(c+2)(3c+2)t
\end{align*}
fulfilling $g(t)<0$ and $h(t)<0$ for $-\sqrt{(3c+2)(c+2)}<t<0$, and
$v(A,c+it)=e^{\pi i/2} \sqrt{{p}(t)+i{q}(t)}$ with ${p}(t)
=-g(t)$, ${q}(t)=-h(t)$, we have, for $\tfrac 8{27}<A<4,$ 
$$
\frac{\partial}{\partial A} \re J_{\mathbf{a}_*}(A)= 4
\re \int^{-b}_0\frac{dt}
{\sqrt{{p}(t)+i{q}(t)}}
=-2\sqrt{2} \int^{0}_{-b} \frac{\sqrt{\sqrt{{p}(t)^2+{q}(t)^2
}+{p}(t)}}{\sqrt{{p}(t)^2+{q}(t)^2}}dt<0.
$$
From this combined with
$\re J_{\mathbf{a}_*}(\tfrac 8{27})=0$, it follows that $\re J_{\mathbf{a}_*}
(A) <0$ in the case (z.5a).
\par
In the case (z.5b) in Figure \ref{zz}, $f(\zeta)>0$ on $(0,c)$, and hence
\begin{align*}
-\frac{\partial}{\partial A} \re J_{\mathbf{a}_*}(A)=&
2\frac{\partial}{\partial A} \re \int^{\zeta_1}_0 \frac{v(A,\zeta)}{\zeta}d\zeta
=4 \re \int^{\zeta_1}_0 \frac{d\zeta}{v(A,\zeta)}
=4 \re \int^{\zeta_1}_c \frac{d\zeta}{v(A,\zeta)} + I(c),  
\\
I(c) =& 4 \int_0^c \frac{d\zeta}{v(A,\zeta)} \ge 0,
\end{align*}
in which the contour is taken along the upper shore of the cut $[\zeta_1,0]$. 
Note that $g(t)>0$ for $-\sqrt{c(3c+4)} <t<0$ and $h(t)<0$ for $-\sqrt{(3c+2)
(c+2)} <t<0$, and set $q(t)=-h(t).$ Then
\begin{align*}
 \re \int^{\zeta_1}_c \frac{d\zeta}{v(A,\zeta)} &=\re \int^{-b}_0 \frac{idt}
{v(A,c+it)} =\im \int^0_{-b} \frac{dt}{\sqrt{g(t)-iq(t)}}  
\\
&= \frac 1{\sqrt{2}} \int_{-b}^0 \frac{\sqrt{\sqrt{g(t)^2+q(t)^2}-g(t)}}
{\sqrt{g(t)^2+q(t)^2}} dt >0,
\end{align*}
which implies $(\partial/\partial A)\re J_{\mathbf{a}_*}(A)<0$ for $A\ge 4.$ 
This fact combined with $\re J_{\mathbf{a}_*}(4)<0$ leads to
$\re J_{\mathbf{a}_*}(A)<0$ for $A\ge 4.$
\par
Thus we have shown the following uniqueness.  
%%%%%%%%%%%%%%%%%%%%%%%%%%%%%%%%%%%%%%%%%%%%%%%%%%%%
%%%%%% Proposition 8.4 %%%%%%%%
\begin{prop}\label{prop8.4}
The Boutroux equations $(\mathrm{BE})_{\phi=0}$ admit a unique
solution $A_0=\tfrac 8{27}$.
\end{prop}
%%%%%%%%%%%%%%%%%%%%%%%%%%%%%%%%%%%%%%%%
%%%%%% Corollary 8.5 %%%%%%%%%
\begin{cor}\label{cor8.5}
For every $A\in \mathbb{C}$, $(J_{\mathbf{a}_*}(A),J_{\mathbf{b}_*}(A)) \not
=(0,0).$
\end{cor}
%%%%%%%%%%%%%%%%%%%%%%%%%%%%%%%%%%%%%%
%%%%%% Proposition 8.6 %%%%%%%%%%%%%%%%
\begin{prop}\label{prop8.6}
Suppose that $0<|\phi|\le \pi/4.$ Let $A_{\phi}$ solve $(\mathrm{BE})_{\phi}$,
and let the elliptic curve  
$v^2= \zeta^4 +4\zeta^3+4\zeta^2 +4A_{\phi} \zeta$ be degenerate. Then
$\phi=\pm \pi/4$ and $A_{\pm \pi/4}=0.$ 
\end{prop}
%%%%%%%%%%%%%%%%%%%%%%%%%%%%
\begin{proof}
The degeneration occurs only when $A_{\phi}=0,$ $\tfrac 8{27}$. By Example
\ref{exa8.1} $A=\tfrac 8{27}$ does not solve $(\mathrm{BE})_{\phi}$ for 
$0<|\phi|\le \pi/4.$ As shown in the case (z.2) above $J_{\mathbf{a}_*}(0)=0$
and $J_{\mathbf{b}_*}(0)<0$, which implies $A=0$ solves $(\mathrm{BE})_{\phi}$
only when $\phi=\pm \pi/4.$ This completes the proof.
\end{proof}
%%%%%%%%%%%%%%%%%%%%%%%%%%%%%%%%%%%
\subsection{Trajectory}\label{ssc8.3}
%%%%%%%%%%%%%%%%%%%%%%%%%%%%%%%%%%%%%%%
The ratio $I(A)={J_{\mathbf{b}_*}(A)}/{J_{\mathbf{a}_*}(A)}$
by Novokshenov \cite[Appendix I]{Novokshenov-2} is useful
in examining the Boutroux equations. 
The following is easily verified.
%%%%%%%%%%%%%%%%%%%%%%%%%%%%%%%%%%%%%%%%%%%%%
%%%%%%%% Proposition 8.7 %%%%%%%%%%%%%%%%%
\begin{prop}\label{prop8.7}
$(1)$ If $I(A)\in \mathbb{R}$, then $A$ solves $(\mathrm{BE})_{\phi}$ for some
$\phi \in \mathbb{R}.$
\par
$(2)$ If $A$ solves $(\mathrm{BE})_{\phi}$ for some $\phi \in \mathbb{R}$ and
$J_{\mathbf{a}_*}(A)\not=0,$ then $I(A) \in \mathbb{R}.$
\end{prop}
%%%%%%%%%%%%%%%%%%%%%%%%%%%%%%%%%%%
%%%%%%%% Remark 8.1 %%%%%%%%%%%%%
\begin{rem}\label{rem8.1}
By Corollary \ref{cor8.5} if $J_{\mathbf{a}_*}(A)=0$ then 
$I(A)^{-1}={J_{\mathbf{a}_*}(A)}/{J_{\mathbf{b}_*}(A)}=0$, and around such
a point $I(A)^{-1}$ has the same property as in Proposition \ref{prop8.7}.
\end{rem}
%%%%%%%%%%%%%%%%%%%%%%%%%%%%%%%%%%%%%%
Let $D_0=\{A\in \mathbb{C}\,|\, \text{$A$ solves $(\mathrm{BE})_{\phi}$ for
some $\phi$}\}$.
%%%%%%%%%%%%%%%%%%%%%%%%%%%%%%
%%%%%%% Proposition 8.8 %%%%%%%%%%%%
\begin{prop}\label{prop8.8}
The set $D_0$ is bounded.
\end{prop}
%%%%%%%%%%%%
\begin{proof}
Let us calculate $\lim_{A\to\infty} I(A)$. The zeros of $\zeta^3+4\zeta^2
+4\zeta +4A=0$ are asymptotically expressed as $\zeta_j \sim e^{(1-2j)\pi i/3}
(4A)^{1/3}$ $(j=1,3,5)$ as $A\to \infty$. Then, 
by $\zeta=e^{\pi i/3}(4A)^{1/3}z$,
\begin{align*}
J_{\mathbf{a}_*}(A) =&2\int^{\zeta_3}_{\zeta_5} \frac{v(A,\zeta)}{\zeta}d\zeta
\sim 2e^{2\pi i/3}(4A)^{2/3} \int^1_{e^{2\pi i/3}} \frac{\sqrt{z^4-z}}z dz,
\\ 
J_{\mathbf{b}_*}(A) =&2\int^{\zeta_1}_{\zeta_3} \frac{v(A,\zeta)}{\zeta}d\zeta
\sim 2e^{2\pi i/3}(4A)^{2/3} \int_1^{e^{-2\pi i/3}} \frac{\sqrt{z^4-z}}z dz.
\end{align*}
Since 
$$
\int\frac{\sqrt{z^4-z}}{z}dz=\frac 12 \sqrt{z^4-z} -\frac 34 \int\frac{dz}
{\sqrt{z^4-z}},
$$
$\lim_{A\to \infty}I(A)$ is a ratio of periods of the elliptic curve
$v^2=z^4-z,$ which implies $\im \lim_{A\to \infty}I(A) \not=0.$ 
By Proposition \ref{prop8.7} the set $D_0$ is bounded.
\end{proof}
Recall the periods 
$\Omega^*_{\mathbf{a}_*}=\int_{\mathbf{a}_*} v(A,\zeta)^{-1}d\zeta$ and
$\Omega^*_{\mathbf{b}_*}=\int_{\mathbf{b}_*} v(A,\zeta)^{-1}d\zeta$ 
of $\Pi^*_A$ given in Section \ref{ssc2.2}. 
To examine the conformality of $I(A)$ we need the following
(see also \cite[Lemma 6.1]{Vere}). 
%%%%%%%%%%%%%%%%%%%%%%%%%%%%%%%%
%%%%%%%% Lemma 8.9 %%%%%%%%%%%%%%%
\begin{lem}\label{lem8.9}
$\Omega^*_{\mathbf{b}_*} J_{\mathbf{a}_*}(A)-\Omega^*_{\mathbf{a}_*} 
J_{\mathbf{b}_*}(A)=-4\pi i.$
\end{lem}
%%%%%%%%%%%%%%%%%%%%%%%%%%%%%%%%%
\begin{proof}
Let $\zeta(t)$ be the elliptic function given by $(\zeta')^2=v(A,\zeta)^2
=\zeta^4+4\zeta^3+4\zeta^2 +4A\zeta,$ and let $P,Q,R,S$ be the 
vertices of its period parallelogram such that $Q=P+\Omega^*_{\mathbf{a}_*}$ and
$S=P+\Omega^*_{\mathbf{b}_*}$. 
Then the function
${g}(u)=\int^u_P {\zeta'(t)^2}{\zeta(t)^{-1}} dt$ fulfils
$$
{g}(u +\Omega^*_{\mathbf{b}_*})-g(u)=\int^{u+\Omega^*_{\mathbf{b}_*}}_{u}
\frac{\zeta'(t)^2}{\zeta(t)}dt= \int_{\mathbf{b}_*} \frac{v(A,\zeta)}{\zeta}
d\zeta =J_{\mathbf{b}_*}(A),
$$
and hence
$$
\biggl(\int_P^Q +\int_R^S\biggr) g(u)du=\int_P^Q(g(u)-g(u+\Omega^*_{\mathbf{b}
_*}))du =\int_P^Q(-J_{\mathbf{b}_*}(A))du = -\Omega^*_{\mathbf{a}_*}
J_{\mathbf{b}_*}(A).
$$
Combining this with $(\int_Q^R +\int_S^P)g(u)du=\Omega^*_{\mathbf{b}_*}
J_{\mathbf{a}_*}(A)$ similarly obtained, we have
$$
\Omega^*_{\mathbf{b}_*}J_{\mathbf{a}_*}(A)-
\Omega^*_{\mathbf{a}_*}J_{\mathbf{b}_*}(A) =\int_{(PQRSP)}g(u)du= 
2\pi i (\mathrm{Res}(u_{\infty}^-)+\mathrm{Res}(u_{\infty}^+))
=-4\pi i,
$$
where $u_{\infty}^{\pm}$ are poles of $g(u)$ in the periodic parallelogram.
\end{proof}
As an immediate corollary for $\Omega_{\mathbf{a},\,\mathbf{b}}$ and 
$\mathcal{J}_{\mathbf{a},\,\mathbf{b}}$ (cf. Section \ref{ssc6.3})
we have the following. 
%%%%%%%%%%%%%%%%%%%%%%%%%%%%%%%%%%%
%%%%% Corollary 8.10 %%%%%%%%%%%%%%
\begin{cor}\label{cor8.10}
$\Omega_{\mathbf{b}} \mathcal{J}_{\mathbf{a}}-\Omega_{\mathbf{a}} 
\mathcal{J}_{\mathbf{b}}=-4\pi i e^{i\phi}.$
\end{cor}
%%%%%%%%%%%%%%%%%%%%%%%%%%%%%%%
Observing that $(\partial/\partial A)J_{\mathbf{a}_*,\,\mathbf{b}_*}(A)
=2\Omega^*_{\mathbf{a}_*,\,\mathbf{b}_*},$ we have the following.
%%%%%%%%%%%%%%%%%%%%%%%%%%%%%%%%%%%%%%%%
%%%%% Proposition 8.11 %%%%%%%%%%%%%%
\begin{prop}\label{prop8.11}
$I'(A)=-8\pi i J_{\mathbf{a}_*}(A)^{-2}$ and
$(1/I)'(A)=8\pi i J_{\mathbf{b}_*}(A)^{-2}$. 
\end{prop}
%%%%%%%%%%%%%%%%%%%%%%%%%%%%%%%%%%%%%%%%%%%%%
%%%%% Remark 8.2 %%%%%%%%%%%%%%%%%%
\begin{rem}\label{rem8.2}
By Example \ref{exa8.1} and the proposition above, $I(A)$ is conformal
around $A=A_0=\tfrac 8{27}$, and given by
$$
I(A)=\tfrac 32\pi i (A-\tfrac 8{27}) (1+o(1)).
$$
By Proposition \ref{prop8.7}
the inverse image of the interval $(-\varepsilon,\epsilon)\subset \mathbb{R}$
under $I(A)$ is a local trajectory consisting of $A_{\phi}$ for 
$|\phi|<\varepsilon_0$, each $A_{\phi}$ solving (BE)$_{\phi}$, where 
$\varepsilon$ and $\varepsilon_0$ are sufficiently small. This local 
trajectory is expressed as
$$
A_{\phi}=\tfrac 8{27} +i\rho(\phi), \quad 
\tfrac 32\pi\re \rho(\phi)\in (-\varepsilon,\varepsilon), 
\quad \text{$\im \rho(\phi)=o(\re\rho(\phi) )$ as $\phi\to 0$.}
$$
Similarly, there exists a local trajectory for $|\phi \pm \pi/4|<\varepsilon_0$
(cf. Proof of Proposition \ref{prop8.6}).
\end{rem}
%%%%%%%%%%%%%%%%%%%%%%%%%%%%%%%%%%%
Suppose that $0<|\phi|<\pi/4$. Write
$$
J_{\mathbf{a}_*}(A)=u(A)+iv(A),\quad 
J_{\mathbf{b}_*}(A)=U(A)+iV(A),\quad A=x+iy.
$$
Then $A$ solves (BE)$_{\phi}$ if and only if
%%%%%%%%%%%%%%% (8.2) %%%%%%%%%%%%%
\begin{equation}\label{8.2}
u(A)-v(A) \tan 2\phi=0, \quad U(A)- V(A) \tan 2\phi=0,
\end{equation} 
which define the trajectory of $A_{\phi}$. The Jacobian for \eqref{8.2} 
around $A=A_{\phi}$ is
\begin{align*}
\det J(x,y)=&\det \begin{pmatrix} u_x-v_x\tan 2\phi & u_y-v_y\tan 2\phi  \\
            U_x-V_x\tan 2\phi  & U_y-V_y\tan 2\phi   \end{pmatrix}
=(1+\tan^2 2\phi) (v_xV_y -v_yV_x)
%% =4(1+\tan^2\phi)(\im \Omega_a \re \Omega_b-\re \Omega_a \im \Omega_b)
%% =-2i(1+\tan^2\phi)(\Omega^*_{\mathbf{a}_*} \overline{\Omega^*_{\mathbf{b}_*}}
%% -  \overline{\Omega^*_{\mathbf{a}_*}}\Omega^*_{\mathbf{b}_*} )
\\
=& -4(1+\tan^2 2\phi)\im 
 \overline{\Omega^*_{\mathbf{a}_*}}\Omega^*_{\mathbf{b}_*} 
= -4(1+\tan^2 2\phi) |\Omega^*_{\mathbf{a}_*}|^2 \im \frac 
{\Omega^*_{\mathbf{a}_*}}{\Omega^*_{\mathbf{b}_*}} \not=0
\end{align*}
by Proposition \ref{prop8.6}. This fact implies that a given local trajectory
for $|\phi-\phi_0|<\varepsilon_0$ with $|\phi_0|<\pi/4$ is extended smoothly 
for $0<|\phi|<\pi/4$ and continuously for $|\phi|\le \pi/4$, and so are the
trajectories described in Remark \ref{rem8.2}. 
For such an extended trajectory, which is bounded by Proposition \ref{prop8.8}, 
let $A_{\phi_n}$ with $\phi_n \to 0$ be a given sequence. 
By the boundedness there exists a subsequence convergent to some $A_0^* \in 
\mathbb{C}$ solving (BE)$_{\phi=0}$, and then $A_0^*=\tfrac 8{27},$ which
implies the uniqueness of the trajectory for $|\phi|\le \pi/4.$
Thus we have the following.
%%%%%%%%%%%%%%%%%%%%%%%%%%%%%%%%%%%
%%%%%%% Proposition 8.12 %%%%%%%%%%
%%%%%%%%%%%%%%%%%%%%%%%%%%%%%%%%%%%
\begin{prop}\label{prop8.12}
There exists a trajectory $A=A_{\phi}$ for $|\phi|\le \pi/4$ with 
the properties$:$
\par
$(1)$ for each $\phi$, $A_{\phi}$ is a unique solution of 
$(\mathrm{BE})_{\phi};$
\par
$(2)$ $A_{\phi}$ is smooth in $\phi$ for $0<|\phi|<\pi/4,$ and continuous in
$\phi$ for $|\phi|\le \pi/4;$
\par
$(3)$ $A_0=\tfrac{8}{27},$ $A_{\pm \pi/4}=0.$
\end{prop}    
%%%%%%%%%%%%%%%%%%%%%%%%%%%%%%%%%
For any $\phi \in \mathbb{R}$ we have 
$$
e^{2i\phi}J_{\mathbf{a}_*,\,\mathbf{b}_*}(A_{\phi})=
-e^{2i(\phi\pm \pi/2)}J_{\mathbf{a}_*,\,\mathbf{b}_*}(A_{\phi}), \quad
\overline{e^{2i\phi}J_{\mathbf{a}_*,\,\mathbf{b}_*}(A_{\phi})}=
e^{-2i\phi}J_{\overline{\mathbf{a}_*},\,\overline{\mathbf{b}_*}}
(\overline{A_{\phi}}),
$$ 
which leads to the following.
%%%%%%%%%%%%%%%%%%%%%%%%%%%%%%%
%%%%%% Proposition 8.13 %%%%%%%%%%
\begin{prop}\label{prop8.13}
$A_{\phi \pm \pi/2}=A_{\phi}$ and $A_{-\phi}=\overline{A_{\phi}}.$
\end{prop} 
To know the shape of the trajectory $A=A_{\phi}$ it is sufficient to examine
it for $|\phi|\le \pi/4$. For $0<|\phi|<\pi/4$ the derivative of \eqref{8.2}
along $A=A_{\phi}$ with respect to $t=\tan 2\phi$ is written in the form 
$J(x,y){}^{\mathrm{T}}\!(x'(t),y'(t)) -{}^{\mathrm{T}}\!(v(A_{\phi}),V(A_{\phi}))
=\mathbf{o}$ with $A_{\phi}=x(t)+iy(t),$ where $J(x,y)$ is the Jacobi
matrix above. Then, for $0<|\phi|<\pi/4$ 
%%%%%%%%% (8.3) %%%%%%%%%%%%
\begin{equation}\label{8.3}
(x'(t),y'(t)) \not=(0,0), \quad (d/d\phi)A_{\phi}=2(x'(t)+iy'(t))\cos^{-2}2\phi
\not=0.
\end{equation} 
By Propositions \ref{prop8.7} and \ref{prop8.11}, 
for every $\phi$ such that $|\phi|\le \pi/4$,
$$
\frac d{dt}I(A_{\phi})=- \frac{8\pi i(x'(t)+iy'(t))}{J_{\mathbf{a}_*}
(A_{\phi})^{2}} \,\,\,\text{or} \,\,\,
\frac d{dt}(1/I)(A_{\phi})= \frac{ 8\pi i(x'(t)+iy'(t))}{J_{\mathbf{b}_*}
(A_{\phi})^{2}} \in \mathbb{R},
$$ 
where $(J_{\mathbf{a}_*}(A_{\phi}), J_{\mathbf{b}_*}(A_{\phi}))\not=(0,0)$ 
by Corollary \ref{cor8.5}.
Setting $J_{\mathbf{a}_*}(A_{\phi})^{-1}$ or $J_{\mathbf{b}_*}(A_{\phi})^{-1}
=P+iQ =i(Q-iP)$ with $P^2+Q^2>0$, 
and observing $\im (d/dt)I(A_{\phi})^{\pm 1}=0$, we have 
$$
x'(t)(P^2-Q^2)-2y'(t)PQ=0.
$$
For $0<|\phi|<\pi/4,$ by \eqref{8.3}, $x'(t)\not=0$, and 
if $y'(t)=0,$ then $P=\pm Q$, implying $2\phi=\mp \pi/4.$ 
For $-\pi/4<\phi<0$, $x'(t)>0$, and for $0<\phi<\pi/4$,
$x'(t)<0$, since $A_{\pm \pi/4}=0$ and $A_{0}=\tfrac 8{27}$. 
If $-\pi/8<\phi<0$ (respectively, $0<\phi<\pi/8$) then $y'(t)<0$, since
$|P|<|Q|$, $PQ>0$ (respectively, $PQ<0$). 
%%%%%%%%%%%%%%%%%%%%%%%%%%%%%%%%%%%%%%%%%%%%%%%%%%%%
%%%%% Proposition 8.14 %%%%%%%%%%%%%%%%
\begin{prop}\label{prop8.14}
Let $A_{\phi}=x(t)+iy(t)$ with $t=\tan 2\phi.$ Then, for $|\phi|\le \pi/4,$
\par
$(1)$ $x'(t)>0$ for $-\pi/4<\phi<0,$ $x'(t)<0$ for $0<\phi<\pi/4;$
\par
$(2)$ $y'(t)<0$ for $0<|\phi|<\pi/4,$ $y'(t)>0$ for $\pi/4 <|\phi|<\pi/2;$
\par
$(3)$ $x'(0)=x'(\pm\tan(\pi/2))=0,$ $y'(\pm \tan (\pi/4))=0.$
\end{prop}
By this proposition and Remark \ref{rem8.2} the trajectory $A_{\phi}$ for
$|\phi|\le \pi/4$ is roughly drawn as in Figure \ref{trajectory}.
%%%%%%%%%%%%%%%%%%%%%%%%%%%%%%%%%%%%%%%%%%%%%%%%%%%%%%%%%%%
%%%%%%%% Figure 8.6 %%%%%%%%%%%%%%%%%%%%%%%%%%%%
{\small
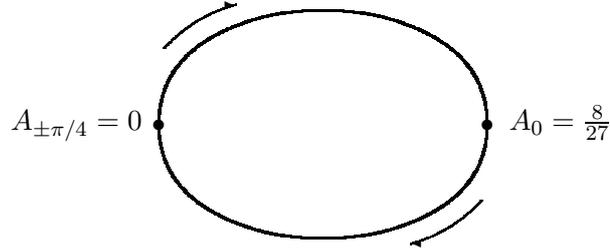
\begin{figure}[htb]
\begin{center}
\unitlength=0.72mm
%%%%%%%%%%%%%%%%%%%%%%%%%%%%%%%%%%
%%%%%%%%%%%%%%%%%%%%%%%%%%%%%%%%%%
%%%%%%%%%%%%%%%%%%%%%%%%%%%%%%%%%%
\begin{picture}(50,40)(-25,-21)
 \put(34,-1){\makebox{$A_0=\tfrac 8{27}$}}
 \put(-57,-1){\makebox{$A_{\pm \pi/4}=0$}}
  \qbezier(-29,14)(-23.5,20)(-16,22)
\put(-16,22){\vector(1,0){0}}
  \qbezier(29,-14)(23.5,-20)(16,-22)
\put(16,-22){\vector(-1,0){0}}
 \put(30,0){\circle*{2}}
 \put(-30,0){\circle*{2}}
\thicklines
\qbezier(10,20)(30,15)(30,0)
\qbezier(-10,20)(0,22)(10,20)
\qbezier(-10,20)(-30,15)(-30,0)
\qbezier(10,-20)(30,-15)(30,0)
\qbezier(-10,-20)(0,-22)(10,-20)
\qbezier(-10,-20)(-30,-15)(-30,0)
\end{picture}
%%%%%%%%%%%%%%%%%%
%%%%%%%%%%%%%%%%%%%%%%%%%%%%%%%%%
\end{center}
\caption{Trajectory $A_{\phi}$ for $|\phi|\le \pi/4$}
\label{trajectory}
\end{figure}
}
%%%%%%%%%%%%%%%%%%%%%%%%%%%%%%%%%%%%%%%%%
%%%%%%%%%%%%%%%%%%%%%%%%%%%%%%%%%%%%%%%%%%%%%%%
%%%%%%%%% Proposition 8.15 %%%%%%%%%%%%%%%%%
\begin{prop}\label{prop8.15}
There exists a trajectory $A=A_{\phi}$ for $\phi \in \mathbb{R}$ with the
properties$:$
\par
$(1)$ for each $\phi$, $A_{\phi}$ is a unique solution of 
$(\mathrm{BE})_{\phi};$
\par
$(2)$ $A_{\phi+\pi/2}=A_{\phi},$ $A_{-\phi}=\overline{A_{\phi}};$
\par
$(3)$ $A_0=\tfrac 8{27},$ $A_{\pm \pi/4}=0$ and $0 \le \re A_{\phi} \le \tfrac
8{27};$
\par
$(4)$ $A_{\phi}$ is continuous in $\phi \in \mathbb{R}$, and smooth in
$\phi\in \mathbb{R}\setminus \{m\pi/4\,|\, m\in \mathbb{Z}\}.$
\end{prop}
%%%%%%%%%%%%%%%%%%%%%%%%%%%%%%%%%%%%%%%%%%%%
As in the derivation of Corollary \ref{cor8.2} by the change of variables
$\zeta=e^{-\phi}z$ Proposition \ref{prop8.15} is converted to the results
on the trajectory for the Boutroux equations \eqref{2.2} 
with the cycles $\mathbf{a}$ and $\mathbf{b}$ on the elliptic curve 
$\Pi_{A,\phi}:$ $w(A,z)^2=z^4+4e^{i\phi}z^3+4e^{2i\phi}z^2 +4e^{3i\phi}Az.$ 
%%%%%%%%%%%%%%%%%%%%%%%%%%%%%%%%%%%%%%%%%
%%%%%%%% Corollary 8.16 %%%%%%%%%%%%%%%%%%%%
\begin{cor}\label{cor8.16}
The trajectory $A=A_{\phi}$ in Proposition $\ref{prop8.15}$ fulfils
\par
$(1')$ for each $\phi\in \mathbb{R}$, $A_{\phi}$ is a unique solution 
of \eqref{2.2}.
\end{cor}   
%%%%%%%%%%%%%%%%%%%%%%%%%%%%%%%%%%%%%%%%%
%%%% Remark 8.3 %%%%%%%%%%%%%%%%%
\begin{rem}\label{rem8.3}
In the corollary above $(1')$ may be replaced with
\par
$(1'')$ for each $\phi$, there exists $A_{\phi}$ uniquely such that, for
every cycle $\mathbf{c}$ on $\Pi_{A_{\phi},\phi}$,
$$
 \re \int_{\mathbf{c}} \frac{w(A_{\phi},z)}z dz =0.
$$
\end{rem}
%%%%%%%%%%%%%%%%%%%%%%%%%%%%%%%%%%%%%%%%%%%%%%%%%%%%%
\subsection{Coalescing turning points}\label{ssc8.4}
Let us observe coalescing turning points as an application of Proposition 
\ref{prop8.14}. On the trajectory $A_{\phi}$, at $A_{\phi=0}=\tfrac 8{27}$ and 
$A_{\phi=\pm \pi/4}=0$, coalescences of, respectively, $\zeta_3$ and $\zeta_1,$ 
and $\zeta_5$ and $\zeta_3$ occur, and then the elliptic curve 
$\Pi^*_A$ degenerates. 
\par
For $A_{\phi}=\tfrac 8{27}+\varepsilon$ the zeros of $v(A_{\phi},\zeta)^2$ 
are denoted
by $\zeta_5=-\tfrac 83 +\delta_5,$ $\zeta_{3,1}=-\tfrac 23 +\delta_{3,1}$,
where $\delta$ and $\delta_j$ $(j=1,3,5)$ are small.
Then $\delta_{3,1} =\pm i (2\varepsilon)^{1/2} +O(\varepsilon),$ 
$\delta_5=-\varepsilon+O(\varepsilon^2).$ By Proposition \ref{prop8.14}, 
$\varepsilon$ may be written in the
form $2\varepsilon= e^{-i\theta(\phi)}\rho(\phi)(1+o(1))$, where $\theta(\phi)=
\pi/2$ if $\phi>0$ and $=-\pi/2$ if $\phi<0$, and $\rho(0)=0$ and
$\rho(\phi)> 0$ for $|\phi|>0$ around $\phi=0$.  
Hence $\zeta_{3,1}=-\tfrac 23 (1 \mp \tfrac 32 i e^{-i \theta(\phi)/2}
\rho(\phi)^{1/2}(1+o(1)))$, $\zeta_5=-\tfrac 83(1+\tfrac 3{16}e^{-i\theta(\phi)}
\rho(\phi)(1+o(1)))$ as $\phi\to 0,$ where each of $\mp$ is chosen in such a way
that $\re \zeta_3 <\re \zeta_1.$
By $\lambda^2=z=e^{i\phi}\zeta$, the turning points $\lambda_j$ $(j=1, 3, 5)$
in Section \ref{sc4} have the following property.
%%%%%%%%%%%%%%%%%%%%%%%%%%%%%%%%%%%%%%%%%%%%%%%%%%
%%%%%% Proposition 8.17 %%%%%%%%%%%%%%%%%
\begin{prop}\label{prop8.17}
Around $\phi=0$,
\begin{align*}
\lambda_5&= i \sqrt{\tfrac 83}+\hat{\varrho}(\phi)(1+o(1)), \quad
\\
\lambda_3&= i \sqrt{\tfrac 23}+e^{3\pi i/4}\varrho(\phi)(1+o(1)),\quad
\lambda_1= i \sqrt{\tfrac 23}+e^{-\pi i/4}\varrho(\phi)(1+o(1))
\end{align*}
as $\phi \to 0+,$ and
\begin{align*}
\lambda_5&= i \sqrt{\tfrac 83}-\hat{\varrho}(\phi)(1+o(1)), \quad
\\
\lambda_3&= i \sqrt{\tfrac 23}+e^{\pi i/4}\varrho(\phi)(1+o(1)),\quad
\lambda_1= i \sqrt{\tfrac 23}+e^{-3\pi i/4}\varrho(\phi)(1+o(1))
\end{align*}
as $\phi \to 0-,$ 
where $\varrho(0)=0$, $\varrho(\phi)> 0$ for $|\phi|>0$,
and $\hat{\varrho}(\phi)
\asymp \varrho(\phi)^2.$
\end{prop}
%%%%%%%%%%%%%%%%%%%%%%%%%%%%%%%%%%%
For $A_{\phi}=\varepsilon$ near $\phi=\pm\pi/4$, 
the zeros of $v(A,\zeta)^2$ are $\zeta_{5,3}=-2
+\delta_{5,3},$ $\zeta_1=\delta_1$ with $\delta_{5,3}=\pm (2\varepsilon)^{1/2}
+O(\varepsilon),$ $\delta_1=-\varepsilon +O(\varepsilon^2).$ 
Writing $2\varepsilon=e^{i\theta
(\phi')}\tilde{\rho}(\theta')$, where $\phi=\pm \pi/4 \mp \phi'$, $\phi'\ge 0$,
and $\tilde{\rho}(0)=0,$ $\tilde{\rho}(\phi')\ge 0$, 
we have $\zeta_{5,3}=-2\pm e^{i\theta(\phi')/2}
\tilde{\rho}(\phi')^{1/2}(1+o(1))$, $\zeta_1=-\tfrac 12e^{i\theta(\phi')} \tilde{\rho}(\phi')(1+o(1))$. 
%%%%%%%%%%%%%%%%%%%%%%%%%%%%%%%%%%%%%%%%%%
%%%%% Proposition 8.18 %%%%%%%%%%%%%%%%
\begin{prop}\label{prop8.18}
Around $\phi =\pm \pi/4$,
\begin{align*}
\lambda_5&= e^{-\pi i/8} (i\sqrt{ 2}+e^{3\pi i/4}\varrho(\phi')(1+o(1))),
\quad
\lambda_3=  e^{-\pi i/8} (i\sqrt{ 2}+e^{-\pi i/4}\varrho(\phi')(1+o(1))),
\\
\lambda_1&= \tfrac 12 e^{-\pi i/8} e^{3\pi i/4}{\varrho}(\phi')(1+o(1))), \quad
\end{align*}
as $\phi'=\phi+\pi/4 \to 0 +,$ and
\begin{align*}
\lambda_5&= e^{\pi i/8} (i\sqrt{ 2}+e^{\pi i/4}\varrho(\phi')(1+o(1))),
\quad
\lambda_3=  e^{\pi i/8} (i\sqrt{ 2}+e^{-3\pi i/4}\varrho(\phi')(1+o(1))),
\\
\lambda_1&= \tfrac 12e^{\pi i/8} e^{\pi i/4}{\varrho}(\phi')(1+o(1))), 
\end{align*}
as $\phi'=\phi-\pi/4 \to 0-,$ 
where ${\varrho}(0)=0$, and ${\varrho}(\phi')> 0$ for $|\phi'|>0$.
\end{prop}
%%%%%%%%%%%%%%%%%%%%%%%%%%%%%%%%%%%%%%%%%%%%%%%%%
The quantities above are also written as functions of $\phi.$ We have
the following, say, around $\phi=0.$
%%%%%%%%%%%%%%%%%%%%%%%%%%%%%%%%%%%%%%%
%%%%% Proposition 8.19 %%%%%%%%%%%%%%%%%%%%%%
\begin{prop}\label{prop8.19}
As $\phi\to 0-$ $($respectively, $\phi\to 0+$$)$,
$\zeta_{p}= -\tfrac 23 -i^{p} \sqrt{2\varepsilon_{\phi}}
+O(\varepsilon_{\phi})$ $($respectively,
$= -\tfrac 23 +i^{p} \sqrt{2\varepsilon_{\phi}}
+O(\varepsilon_{\phi})$$)$ for $p=1,3$, and
$\zeta_{5}= -\tfrac 83 -\varepsilon_{\phi}+O(\varepsilon_{\phi}^2),$ where
$A_{\phi}=\tfrac 8{27}+\varepsilon_{\phi}$ with
$$
\varepsilon_{\phi}=\frac{8i}3 \frac{\phi}{\ln\phi}\Bigl(1+O\Bigl(\frac{\ln\ln
\phi}{\ln\phi}\Bigr)\Bigr).
$$
\end{prop}
%%%%%%%%%%%%%%%%%%%%%%%%%%%%%%%
\begin{proof}
In the case $\phi\to 0-,$ 
set $A_{\phi}= \tfrac {8}{27} +\varepsilon$,
and
$\zeta_{1}=-\tfrac 23- i\sqrt{2\varepsilon}+O(\varepsilon),$ 
$\zeta_{3}=-\tfrac 23+ i\sqrt{2\varepsilon}+O(\varepsilon),$ 
$\zeta_5=-\tfrac 83 -\varepsilon+O(\varepsilon^2),$ where $\varepsilon=
i|\varepsilon|(1+o(1))$. Then 
\begin{align*}
\tfrac 12 J_{\mathbf{a}_*}(A_{\phi})&= i^3\int^{\zeta_3}_{\zeta_5}
\sqrt{(\zeta-\zeta_5)(\zeta_3-\zeta)(\zeta_1-\zeta)(-\zeta)}\frac{d\zeta}{\zeta}
\\
&= -i\int^{-2/3+i\sqrt{2\varepsilon}}_{-8/3}
\sqrt{(\zeta+\tfrac 83 )(-\tfrac 23 +i\sqrt{2\varepsilon}-\zeta)
(-\tfrac 23-i\sqrt{2\varepsilon}-\zeta)(-\zeta)}\frac{d\zeta}{\zeta}
+O(\varepsilon)
\\
&= i\int^{2+i\sqrt{2\varepsilon}}_0 \sqrt{\frac t{\frac 83-t}}
\sqrt{(2-t)^2 +2\varepsilon}
\,dt+O(\varepsilon) \qquad (\zeta+\tfrac 83 =t)
\\
&= i\int^{2+i\sqrt{2\varepsilon}}_0 \sqrt{\frac t{\frac 83-t}}\Bigl(2-t+
{\varepsilon}(2-t)^{-1} +O(\varepsilon^2(2-t)^{-3})\Bigr)dt+O(\varepsilon) 
\\
&= i(\tfrac 23\sqrt{3} -\tfrac 12\sqrt{3} \varepsilon \ln\varepsilon +
O(\varepsilon)),
\end{align*} 
and
\begin{align*}
\tfrac 12 J_{\mathbf{b}_*}(A_{\phi})&= i^2\int^{\zeta_1}_{\zeta_3}\sqrt{
(\zeta-\zeta_5)(\zeta-\zeta_3)(\zeta_1-\zeta)(-\zeta)}\frac{d\zeta}{\zeta}
\\
&= -\int^{-2/3-i\sqrt{2\varepsilon}}_{-2/3+i\sqrt{2\varepsilon}}
\sqrt{(\zeta+\tfrac 83 )(\zeta+\tfrac 23 -i\sqrt{2\varepsilon})
(-\tfrac 23-i\sqrt{2\varepsilon}-\zeta)(-\zeta)}\frac{d\zeta}{\zeta}
+O(\varepsilon^{3/2})
\\
&= \sqrt{3} \int^{-i\sqrt{2\varepsilon}}_{i\sqrt{2\varepsilon}}
\sqrt{-2\varepsilon-t^2}\,dt 
+O(\varepsilon^{3/2}) \qquad(\zeta+\tfrac 23=t)
\\
&= -\sqrt{3}\pi +O(\varepsilon^{3/2}).
\end{align*}
Thus we have $J_{\mathbf{a}_*}(A_{\phi})=i\sqrt{3}(\tfrac 43 -\varepsilon
\ln \varepsilon +O(\varepsilon))$ and $J_{\mathbf{b}_*}(A_{\phi})
=-2\sqrt{3} \pi\varepsilon +O(\varepsilon^{3/2})$.
For small $\phi$, the Boutroux equations $\re e^{2i\phi}J_{\mathbf{a}_*}
(A_{\phi})=\re r^{2i\phi}J_{\mathbf{b}_*}(A_{\phi})=0$ yield $\varepsilon
=\varepsilon_{\phi}$ as in the proposition.
\end{proof}
By $2\Omega^*_{\mathbf{a}_*,\,\mathbf{b}_*}=(\partial/\partial\varepsilon_{\phi})
J_{\mathbf{a}_*,\,\mathbf{b}_*}(A_{\phi})$ we have the following corollary.
%%%%%%%%%%%%%%%%%%%%%%%%%%%%%%%%%%%%%%%%
%%%%%% Corollary 8.20 %%%%%%%%%%%%%%
\begin{cor}\label{cor8.20}
Around $\phi=0$,
\begin{align*}
& J_{\mathbf{a}_*}(A_{\phi})=\tfrac 43 i {\sqrt{3}} 
(1-2i\phi(1+O(\delta_{\phi}))),
\quad 
 J_{\mathbf{b}_*}(A_{\phi})=-\tfrac {16}3\pi i{\sqrt{3}}\, {\phi} \,
({\ln\phi})^{-1}
(1+O(\delta_{\phi})),
\\
&\Omega_{\mathbf{a}_*}^* =-\tfrac 12 i\sqrt{3} \ln \phi \,(1+O(\delta_{\phi})),
\quad
\Omega_{\mathbf{b}_*}^* =-{\sqrt{3} \pi}+O({\phi}^{1/2}),
\end{align*}
where $\delta_{\phi}=\ln\ln\phi \,(\ln\phi)^{-1}$ as $\phi\to 0.$
\end{cor}
%%%%%%%%%%%%%%%%%%%%
\par
{\bf Acknowledgements.} The author is grateful to Professor Yousuke Ohyama
for a stimulating conversation informing circumstances of Kapaev's
announcement \cite{Kapaev-3} and inspiring the author to tackle this work.
%%%%%%%%%%%%%%%%%%%%%%%%%%%%%%%%%
%%%%% References %%%%%%%%%%%%%%%%%%%%%%%%%
%%%%%%%%%%%%%%%%%%%%%%%%%


\begin{thebibliography}{99}

\bibitem{AS}
M.~{\sc Abramowitz} and I.~A.~{\sc Stegun},
{\it Handbook of Mathematical Functions with Formulas, Graphs, and
Mathematical Tables}, Dover, New York, 1972.

%% \bibitem{Andreev}
%% F.~V.~{\sc Andreev}, 
%% On some special functions of the fifth Painlev\'e equation, (Russian)
%% Zap. Nauchn. Sem.~S.-Peterburg.~Otdel.~ 
%% Mat.~Inst.~Steklov. (POMI) {\bf 243} (1997), Kraev.~Zadachi Mat.~Fiz. i 
%% Smezh.~Vopr.~Teor.~Funktsii. 28, 10--18, 338; translation in J.~Math.~Sci. 
%% {\bf 99} (2000), 802--807. 




%% \bibitem{Andreev-Kitaev-1}
%% F.~V.~{\sc Andreev} and A.~V.~{\sc Kitaev}, 
%% On connection formulas for the asymptotics of some special solutions of the 
%% fifth Painlev\'e equation, (Russian) Zap.~Nauchn.~Sem.~S.-Peterburg.~Otdel.~ 
%% Mat.~Inst.~Steklov. (POMI) {\bf 243} (1997), Kraev.~Zadachi Mat.~Fiz. i 
%% Smezh.~Vopr.~Teor.~Funktsii. 28, 19--29, 338; translation in J.~Math.~Sci. 
%% {\bf 99} (2000), 808--815. 


%% \bibitem{Andreev-Kitaev-2}
%% F.~V.~{\sc Andreev} and A. V. {\sc Kitaev}, 
%% Exponentially small corrections to divergent asymptotic expansions of 
%% solutions of the fifth Painlev\'e equation, Math.~Res.~Lett. {\bf 4} (1997), 
%% 741--759. 


%% \bibitem{Andreev-Kitaev}
%% F.~V.~{\sc Andreev} and A.~V.~{\sc Kitaev},
%% Connection formulae for asymptotics of the fifth Painlev\'e transcendent 
%% on the real axis, Nonlinearity {\bf 13} (2000), 1801--1840.  



%% \bibitem{Andreev-Kitaev-2019}
%% F.~V.~{\sc Andreev} and A.~V.~{\sc Kitaev},
%% Connection formulae for asymptotics of the fifth Painlev\'e transcendent 
%% on the imaginary axis: I,
%% Stud.\ Appl.\ Math. {\bf 145} (2020), 397--482. 
% https://doi.org/10.1111/sapm.12323



%% \bibitem{Bol}
%% A.~A.~{\sc Bolibruch}, S.~{\sc Malek} and C.~{\sc Mitschi},  
%% On the generalized Riemann-Hilbert problem with irregular singularities, 
%% Expo.\ Math. {\bf 24} (2006), 235--272. 




%% \bibitem{Boutroux} 
%% P.~{\sc Boutroux}, 
%% Recherches sur les transcendantes de M.~Painlev\'e et l'etude asymptotique des 
%% e\'quations diff\'erentielles du second ordre, Ann.\ Sci.\ \'Ecole Norm.\ Sup. 
%% (3) {\bf 30} (1913), 255--375.


%% \bibitem{Boutroux-2}
%% P.~{\sc Boutroux},
%% Recherches sur les transcendantes de M.~Painlev\'e et l'etude asymptotique des 
%% \'equations diff\'erentielles du second ordre (suite), Ann.\ Sci.\ \'Ecole 
%% Norm.\ Sup. (3) {\bf 31} (1914), 99--159.



\bibitem{HTF}
A.~{\sc Erdelyi}, W.~{\sc Magnus}, F.~{\sc Oberhettinger} and F.~G.~{\sc
Tricomi},
{\it Higher Transcendental Functions} Vols I, II, III 
(Bateman Manuscript Project), McGraw-Hill, New York, 1953.


\bibitem{F}
M.~V.~{\sc Fedoryuk},
{\it Asymptotic Analysis}, Springer-Verlag, New York, 1993.

\bibitem{FIKN}
A.~S.~{\sc Fokas}, A.~R.~{\sc Its}, A.~A.~{\sc Kapaev} and
V.~Yu.~{\sc Novokshenov},
{\it Painlev\'e Transcendents, The Riemann-Hilbert Approach},
Math.\ Surveys and Monographs 128, AMS Providence, 2006.

\bibitem{Gromak-1}
V.~I.~{\sc Gromak}, 
Single-parameter families of solutions of Painlev\'e equations, (Russian) 
Differentsial'nye Uravneniya {\bf 14} (1978), no. 12, 2131--2135, 2298. 

\bibitem{Gromak-2}
V.~I.~{\sc Gromak}, 
On the theory of the fourth Painleve equation, (Russian) 
Differentsial'nye Uravneniya {\bf 23} (1987), no. 5, 760--768, 914; 
translation in Differential Equtions {\bf 23} (1987), 506--513. 

\bibitem{GLS}
V.~I.~{\sc Gromak}, I.~{\sc Laine} and S.~{\sc Shimomura}, 
{\it Painlev\'e Differential Equations in the Complex Plane}, 
De Gruyter Studies in Mathematics, 28. Walter de Gruyter \& Co., Berlin, 2002. 



\bibitem{HC}
A.~{\sc Hurwitz} and R.~{\sc Courant},
{\it Vorlesungen \"{u}ber allgemeine Funktionentheorie und elliptische 
Funktionen}, Berlin, Springer, 1922. 


\bibitem{I-Kapaev}
A.~R.~{\sc Its} and A.~A.~{\sc Kapaev},  
Connection formulae for the fourth Painlev\'e transcendent; Clarkson-McLeod
solution, J.~Phys.~A: Math.~Gen. {\bf 31} (1998), 4073--4113.


%% \bibitem{Its-Kapaev}
%% A.~R.~{\sc Its} and A.~A.~{\sc Kapaev},  
%% The nonlinear steepest descent approach to the asymptotics of the second 
%% Painlev'e transcendent in the complex domain,
%% {\it MathPhys odyssey}, 2001, 273--311, Prog. Math. Phys., 23, Birkhauser 
%% Boston, Boston, MA, 2002. 


%% \bibitem{IN}
%% A.~R.~{\sc Its} and V.~Yu.~{\sc Novokshenov},
%% {\it The Isomonodromy Deformation Method in the Theory of Painlev\'e
%% Equations}, Lect.\ Notes in Math. 1191, Springer-Verlag, New York, 1986.





%% \bibitem{Jimbo}
%% M.~{\sc Jimbo}, Monodromy problem and the boundary condition for some Painlev\'e 
%% equations, 
%% Publ.\ Res.\ Inst.\ Math.\ Sci. {\bf 18} (1982), 1137--1161.  


\bibitem{JM}
M.~{\sc Jimbo} and T.~{\sc Miwa}, Monodromy preserving deformation 
of linear ordinary 
differential equations with rational coefficients. II, Phys.\ D
{\bf 2} (1981), 407--448.  


%% \bibitem{Joshi-Kruskal-1}
%% N.~{\sc Joshi} and M.~D.~{\sc Kruskal}, 
%% An asymptotic approach to the connection problem for the first and the second 
%% Painlev\'e equations, 
%% Phys.~Lett. A {\bf 130} (1988), 129--137. 



\bibitem{Kapaev-2}
A.~A.~{\sc Kapaev}, 
Global asymptotics of the second Painle\'e transcendent, Phys.~Lett. A
{\bf 167} (1992), 356--362.



\bibitem{Kapaev-1}
A.~A.~{\sc Kapaev}, 
Essential singularity of the Painlev\'e function of the 
second kind and the nonlinear Stokes phenomenon. (Russian) 
Zap.~Nauchn.~Sem.~Leningrad.~Otdel.~Mat.~Inst.~Steklov. (LOMI) {\bf 187} 
(1991), Differentsial'naya Geom.~Gruppy Li i Mekh. 12, 139--170, 173, 176; 
translation in J.~Math.~Sci. {\bf 73} (1995), 500--517. 



\bibitem{Kapaev-3}
A.~A.~{\sc Kapaev}, 
Global asymptotics of the fourth Painle\'e transcendent, Steklov Mat.~Inst.
and IUPUI, Preprint \# 96-06, 1996; available at http://www.pdmi.ras.ru/
preprint/1996/index.html.

\bibitem{Kapaev-4}
A.~A.~{\sc Kapaev}, 
Connection formulae for the degenerated asymptotic solutions
of the fourth Painlev\'e equation, arXiv:solv-int/9805011v1, 1998.


\bibitem{Ka-Ki}
A.~A.~{\sc Kapaev} and A.~V.~{\sc Kitaev}, 
Connection formulae for the first Painlev\'e transcendent in the complex domain,
Lett.~Math.~Phys. {\bf 27} (1993), 243--252. 


\bibitem{Kitaev}
A.~V.~{\sc Kitaev}, 
Self-similar solutions of the modified nonlinear Schr\"{o}dinger equation,
Teor.~Mat.~Fiz. {\bf 64} (1985), 347--369; translation in
Theor.~Math.~Phys. {\bf 64} (1985) 878--894.

\bibitem{Kitaev-D}
A.~V.~{\sc Kitaev}, 
Method of isomonodromy deformations for complete third and fourth
Painlev\'e equations, (Russian) PhD Thesis, Leningrad State University (1988).

\bibitem{Kitaev-0}
A.~V.~{\sc Kitaev},
Asymptotic description of solutions of the fourth Painlev\'e equation 
on analogues of Stokes's rays, (Russian) 
Zap.~Nauchn.~Sem.~Leningrad.~Otdel.~Mat.~Inst.~Steklov. (LOMI) {\bf 169} (1988),
Voprosy Kvant.~Teor.~Polya i Statist.~Fiz. 8, 84--90, 187; 
translation in J.~Soviet Math. {\bf 54} (1991), no. 3, 916--920. 



\bibitem{Kitaev-1}
A.~V.~{\sc Kitaev}, 
The justification of asymptotic formulas that can be obtained by the method 
of isomonodromic deformations, (Russian) 
Zap.~Nauchn.~Sem.~Leningrad.~Otdel.~ Mat.~Inst.~Steklov. (LOMI) {\bf 179} 
(1989), Mat.~Vopr.~Teor.~Rasprostr.~Voln. 19, 101--109, 189--190; 
translation in J.~Soviet Math. {\bf 57} (1991), no. 3, 3131--3135. 


\bibitem{Kitaev-2}
A.~V.~{\sc Kitaev}, 
The isomonodromy technique and the elliptic asymptotics of the first 
Painlev\'e transcendent, (Russian) Algebra i Analiz {\bf 5} (1993), 179--211; 
translation in St.~Petersburg Math.~J. {\bf 5} (1994), 577--605. 


\bibitem{Kitaev-3}
A.~V.~{\sc Kitaev}, 
Elliptic asymptotics of the first and second Painlev\'e transcendents, 
(Russian) Uspekhi Mat.~Nauk {\bf 49} (1994), 77--140; 
translation in Russian Math.~Surveys {\bf 49} (1994), 81--150. 


\bibitem{Murata}
Y.~{\sc Murata},
Rational solutions of the second and the fourth Painlev\'e equations, 
Funkcial.~Ekvac. {\bf 28} (1985), 1--32.




\bibitem{Novokshenov-1}
V.~Yu.~{\sc Novokshenov}, A modulated elliptic function as a solution of the 
second Painlev\'e equation in the complex plane, (Russian) Dokl.~Akad.~Nauk 
SSSR {\bf 311} (1990), 288--291; translation in Soviet Math.~Dokl. {\bf 41} 
(1990), 252--255.


\bibitem{Novokshenov-2}
V.~Yu.~{\sc Novokshenov}, The Boutroux ansatz for the second Painlev\'e 
equation in the complex domain, (Russian) Izv.~Akad.~Nauk SSSR Ser.~Mat. 
{\bf 54} (1990), 1229--1251; translation in Math.~USSR-Izv. {\bf 37} (1991), 
587--609. 


\bibitem{Novokshenov-3}
V.~Yu.~{\sc Novokshenov}, Asymptotics in the complex plane of the third 
Painlev\'e transcendent, {\it Difference equations, special functions and 
orthogonal polynomials}, 432--451, World Sci.~Publ., Hackensack,~NJ, 2007. 



\bibitem{Novokshenov-4}
V.~Yu.~{\sc Novokshenov}, Connection formulas for the third Painlev\'e 
transcendent in the complex plane, {\it Integrable systems and random matrices}, 
55--69, Contemp.~Math., {\bf 458}, Amer. Math. Soc., Providence, RI, 2008. 



%\bibitem{O}
%F.~W.~J.~{\sc Olver},
%{\it Asymptotics and Special Functions}, Academic Press, New York, 1974.




\bibitem{SS}
S.~{\sc Shimomura}, 
Boutroux ansatz for the degenerate third Painleve transcendents,
Publ.\ Res.\ Inst.\ Math.\ Sci. (to appear)
arXiv:2207.11495 math.CA

\bibitem{SSS}
S.~{\sc Shimomura}, 
Elliptic asymptotics for the complete third Painleve transcendents,
Funkcial.\ Ekvac. (to appear)
arXiv:2211.00886  math.CA



\bibitem{S}
S.~{\sc Shimomura}, Elliptic asymptotic representation of the fifth 
Painlev\'e transcendents, Kyushu J. Math. {\bf 76} (2022), 43--99. 
Corrigendum to `Elliptic asymptotic representation 
of the fifth Painlev\'e transcendents', Kyushu J. Math. {\bf 77}
(2023), 191--202. arXiv:2012.07321 math.CA



% \bibitem{Stein}
% N.~{\sc Steinmetz}, A unified approach to the Painlev'e transcendents, 
% Ann.\ Acad.\ Sci.\ Fenn.\ Math. {\bf 42} (2017), 17--49. 

\bibitem{UW}
H.~{\sc Umemura} and H.~{\sc Watanabe},
Solutions of the second and fourth Painlev\'e equations, I, Nagoya Math.~J.
{\bf 148} (1997), 151--198.

\bibitem{Vere}
V.~L.~{\sc Vereshchagin}, 
Global asymptotics for the fourth Painlev\'e transcendent, (Russian) Mat.~Sb. 
{\bf 188} (1997), 11--32; translation in Sb.~Math. {\bf 188} (1997), 
1739--1760. 


\bibitem{WW}
E.~T.~{\sc Whittaker} and G.~N.~{\sc Watson},  
{\it A course of modern analysis}, 
%% An introduction to the general theory of infinite processes and of analytic 
%% functions; with an account of the principal transcendental functions},
Reprint of the fourth (1927) edition. 
Cambridge Mathematical Library. Cambridge University Press, Cambridge, 1996. 

 
\end{thebibliography}
\end{document}